\documentclass[12pt]{article}
\usepackage[utf8]{inputenc} 
\usepackage[T1]{fontenc}
\usepackage{authblk}
\usepackage{geometry}

\let\OLDthebibliography\thebibliography
\renewcommand\thebibliography[1]{
  \OLDthebibliography{#1}
  \setlength{\parskip}{0pt}
  \setlength{\itemsep}{0pt plus 0.6ex}
}

 \makeatletter

 \providecommand{\keywords}[1]
{
  \small	
  \textbf{\textit{Keywords---}} #1
}

\usepackage{titlesec}

\titleformat*{\section}{\bfseries}

\titleformat*{\subsection}{\bfseries}

\titleformat*{\subsubsection}{\bfseries} 

\usepackage{mathtools}
\usepackage{bm}
\usepackage{dsfont}
\usepackage{caption}
\usepackage{subcaption}
\usepackage[utf8]{inputenc}
\usepackage{graphicx} 
\usepackage{amsmath,amssymb,amsthm} 
\usepackage{mathrsfs} 
\usepackage{verbatim} 
\usepackage{bbm}
\usepackage[dvipsnames, svgnames, x11names]{xcolor}
\usepackage{marginnote} 

\usepackage[T1]{fontenc} 
\usepackage[english]{babel} 
\usepackage[plainpages=false,pdfpagelabels,colorlinks=true,linkcolor=red,urlcolor=red,citecolor=blue]{hyperref}

\def\P{\mathbb{P}}
\def\E{\mathbb{E}}
\def\H{\mathbb{H}}
\def\er{ Erd{\"o}s-R{\`e}nyi }
\usepackage{amsmath}
\DeclareMathOperator*{\argmax}{arg\,max}

\DeclareMathOperator{\co}{co}    
\usepackage{bm} 

\DeclareMathOperator{\Var}{Var}   

\newcommand{\subdiff}{\partial} 

\newcommand{\homc}{\mathrm{hom}}
\newcommand{\injc}{\mathrm{inj}}
\newcommand{\Aut}{\mathrm{Aut}}

\newcommand{\fall}[2]{(#1)_{#2}} 

\numberwithin{equation}{section}
\newtheorem{theorem}{Theorem}[section]
\newtheorem{corollary}{Corollary}[section]
\newtheorem{lemma}{Lemma}[section]

\newtheorem{test}{Test}[section]

\def\bbeta{\boldsymbol{\beta}}

\patchcmd{\@maketitle}{\LARGE}{\fontsize{22pt}{28pt}\selectfont}{}{}

\title{\textbf{Maximum entropy based testing in network models: \\ ERGMs and constrained optimization}}

\author{
Subhro Ghosh\textsuperscript{*}\textsuperscript{$\dagger$}\\
Department of Mathematics, National University of Singapore\\
\href{mailto:matghos@nus.edu.sg}{\texttt{subhrowork@gmail.com}}
\and
Rathindra Nath Karmakar\textsuperscript{*}\textsuperscript{$\ddagger$}\\
Department of Mathematics, Kyushu University\\
\href{mailto:karmakar.rathindra.735@s.kyushu-u.ac.jp}{\texttt{karmakar.rathindra.735@s.kyushu-u.ac.jp}}
\and
Samriddha Lahiry\textsuperscript{*}\textsuperscript{$\ddagger$}\\
Department of Statistics and Data Science, National University of Singapore\\
\href{mailto:slahiry@nus.edu.sg}{\texttt{slahiry@nus.edu.sg}}
}

\date{}

\begin{document}

\renewcommand\Authands{ and }

\maketitle

\begingroup
\renewcommand{\thefootnote}{\fnsymbol{footnote}}
\setcounter{footnote}{0}
\footnotetext{\textsuperscript{*}Authors are listed in alphabetical order of their surnames.}
\footnotetext{\textsuperscript{$\dagger$}{Supported in part by Singapore MOE grants R-146-000-312-114, A-8002014-00-00, A-8003802-00-00, E-146-00-0037-01, A-8000051-00-00, A-0009806-01-00 and A-0004586-00-00}}
\footnotetext{\textsuperscript{$\ddagger$}Corresponding authors.}
\endgroup

\begin{abstract}
Stochastic network models play a central role across a broad range of scientific disciplines, and questions of statistical inference arise naturally in this context. Maximum entropy ({\it abbrv.} MaxEnt) principles are well-regarded as a standard approach to such inferential questions on network data.  

In this paper, we investigate goodness-of-fit and two-sample testing procedures for statistical networks grounded in the principle of maximum entropy. Our approach is based on considering a constrained entropy-maximization problem on the space of networks, subject to a prescribed set of structural constraints (informed, for instance, by the goodness-of-fit setup). Our test statistics are based on the Lagrange multipliers associated with such constrained optimization problems, which to our knowledge is a novelty in the statistical networks literature.

We first establish consistency of the proposed tests in the classical setting where the number of vertices is fixed. We then turn to regimes of greater contemporary interest, wherein the size of the graph grows commensurately with the sample size, and develop tests in both the sparse and dense settings. In the dense setting, we analyze graphs generated from exponential random graph models ({\it abbrv.} ERGM)  as well as \er models, while in the sparse regime our theory applies to \er graphs. The scope of our results are commensurate with the development of the theory of large scale network asymptotics, varying in robustness from dense to sparse regimes.

Our techniques strongly interface with the burgeoning theory of {\it non-linear large deviations} in the space of random graphs, an approach which we believe would be of much wider applicability in statistical network problems. In a different direction, our inferential approach based on Lagrange multipliers turns out to connect to classical {\it score tests} for constrained MLE in real-valued data that are widely popular in econometrics. Our results provide a unified entropy-based inferential framework for network model assessment across a wide range of network growth regimes. More broadly, it points to a general overarching strategy whereby goodness-of-fit testing can be framed as a constrained optimization problem and the associated Lagrange multiplier can be used for devising the test.
\end{abstract}


\begingroup
\hypersetup{linkcolor=RawSienna}
\tableofcontents
\endgroup

\newpage

\section{Introduction \label{sec:intro}}

Over the past two decades, the statistical modeling of network data using random graph models has received significant attention, with applications spanning diverse domains such as social networks, brain networks, and omics networks. From a statistical inference perspective, these networks have been extensively studied in terms of both estimation and hypothesis testing (cf. 
 \cite{kolaczyk2009statistical} and the references therein). Estimation focuses on inferring the underlying random network from observed data, whereas hypothesis testing aims to determine whether the given data conforms to a specified random network. In this article, we consider the setting where i.i.d. copies of a random graph of possibly growing size are provided, and the objective is to perform statistical testing under constraints based on expected motif counts of the underlying random graph model.

\subsection{Testing in random network models and applications} In line with classical statistical testing setups, hypothesis testing in random networks can be formulated in two distinct ways. The first approach pertains to goodness-of-fit testing, which seeks to determine whether a given network or a collection of networks originates from a specific random graph model with predefined parameters. In the context of random networks, goodness-of-fit testing has been widely explored in the literature; we refer the reader to \cite{xu2021stein, dan2020goodness, GOF_sbm, brune2025goodness, jin2025network, maugis2020testing, oudah2020, gao2017subgraph,li2013assessing,csiszar2012testing,Chung-Lu} for a partial list, and the references therein. The second approach concerns two-sample testing, where two independent samples of random networks are provided, and the objective is to assess whether they are generated from the same underlying random graph model. This problem has been extensively studied for various random graph models; see for example \cite{chatterjee2023two, ghoshdastidar2017networkstatistics, ghoshdastidar2020two, dotproduct2018statistical, tang2017nonparametric, agterberg2020nonparametric} among many others.

Both goodness-of-fit tests and two-sample tests are well-motivated by applications across a wide range of disciplines. In particular, goodness-of-fit testing is utilized to evaluate the suitability of protein-protein interaction (PPI) networks \cite{ospina2019assessment, elliott2018nonparametric} and to assess the fit of functional neuroimaging data \cite{Ginestet2015}. Similarly, two-sample testing naturally arises in various applications. For example, it has been employed to analyze gene regulatory networks, where it is used to study topological changes under two different breast cancer treatments \cite{zhang2009differential}; to investigate structural brain differences, where it helps compare anatomical variations between healthy individuals and schizophrenic patients \cite{bassett2008hierarchical}; and to advance computational biology, where it is applied in graph-based classification tasks \cite{shervashidze11a}

\subsubsection{Goodness-of-fit testing in random networks}In the context of goodness-of-fit testing, various approaches have been developed for different random graph models. For instance, \cite{ouadah2022motif} and \cite{maugis2020testing} proposed motif-based tests for bipartite network models and kernel-based tests for general random graphs, respectively. For inhomogeneous random graph models, a degree-based test has been introduced in \cite{oudah2020}, while \cite{dan2020goodness} develops a test based on the adjacency matrix of the random graph. In the context of the stochastic block model, \cite{GOF_sbm} also develops an adjacency matrix-based test. However, unlike \cite{dan2020goodness}, which considers multiple independent samples, \cite{GOF_sbm} focuses on a test based on a single sample. Several other methods have also been explored. For example, \cite{xu2021stein} studies goodness-of-fit testing for exponential random graph models (ERGMs) and introduces a test based on kernel Stein discrepancies. Additionally, \cite{brune2025goodness} presents a framework for testing the homogeneity of a random graph using graph functionals, which generalizes the subgraph count approach. We also refer the reader to \cite{bickel_sarkar,chatterjee2024higher,motif_fourth,networkminimaxrates,gao2017subgraph,comm_detec,gao2015rate} for related inference problems in other statistical network models.

\subsubsection{Two sample testing in random networks}On the other hand, several two-sample tests have also been developed. A two-sample version of the adjacency matrix-based test for inhomogeneous random graphs is proposed in \cite{dan2020goodness}, while \cite{ghoshdastidar2017networkstatistics} explores similar questions using network summary statistics. Other two sample tests include \cite{tang2017nonparametric}  which investigates whether two random dot product graphs share the same generating latent positions, using a test statistic based on spectral decomposition of the adjacency matrix, \cite{Ginestet2015} which constructs tests based on Fréchet means of the Laplacians, and  \cite{Li2022} which examines comparisons of population means in network data, focusing on individual links and utilizing symmetric matrices. More recently, \cite{chen2023} proposed a general procedure for hypothesis testing in network data, aimed at distinguishing the distributions of two samples of networks.\\

\subsection{Testing based on maximum entropy} In this article, we present a different test procedure based on subgraph counts. In particular, rather than directly developing a test based on raw subgraph counts, we maximize the entropy of the empirical distribution of subgraph counts while imposing specific constraints on this distribution (see the next section for details). Maximizing entropy subject to constraints is a standard approach in statistical mechanics and information theory \cite{Jaynes_maxent,Shannon1948}, grounded in Jaynes’ principle of maximum entropy \cite{jaynes1982rationale}. In the context of network models, entropy maximization under constraints has been extensively explored in the statistical physics literature (see, for instance, \cite{park2004statistical}). In this paradigm, one considers probability distribution on space of all graphs with a given set of nodes. Many random graph models arise out of such maximization and the constraints represent the so called topological properties of the network concerned (see \cite{squartini2017maximum} and the references therein).

Our method is inspired by a maximal entropy principle similar to the classical setting; however, we maximize the entropy of the empirical distribution from $n$ samples of the random graph and construct our test statistics based on the Lagrange multiplier. On the other hand, the condition imposed by the null hypothesis translates to constraints on expected motif counts (which is analogous to moment constraints on distributions) and mirrors the topological constraints described in \cite{squartini2017maximum}. \\
 
\subsection{Connection to Lagrange Multiplier tests} Maximizing the entropy of an empirical distribution  under constraints (moment constraints or otherwise) naturally leads to a Lagrange multiplier based formulation from an optimization point of view. Such an optimization procedure canonically yields an optimal Lagrange multiplier. For most statistical applications, modern day practitioners often regard this object largely as a computational device, and it is typically not investigated for its implications vis-a-vis the big picture.

In this work however, we demonstrate in our network-based setup that, under appropriate centering and scaling, this optimal Lagrange multiplier exhibits asymptotic normality under the null hypothesis. This asymptotic behavior forms the foundation of our proposed goodness-of-fit test. We further extend this framework to construct a two-sample test, utilizing the optimal Lagrange multipliers computed independently for each sample.

It turns out, in fact, that Lagrange multiplier-based tests have a long history in the classical statistical literature, particularly in the context of constrained likelihood maximization \cite{silvey1958, silvey1959}. Indeed, the well-known {\it score test} can be shown to be equivalent to the so-called Lagrange Multiplier test \cite{silvey1959}. These methods are widely used in econometrics \cite{LM_econ}, where they underpin classical procedures such as the Breusch–Pagan test for heteroscedasticity \cite{bp_test}. 

However, the Lagrange Multiplier test has predominantly been applied to random variables in Euclidean spaces, and, to the best of our knowledge, hypothesis testing procedures based on Lagrange multipliers have not been developed for network models, where the growing size of the networks introduces additional analytical challenges. While the conceptual idea of using a Lagrange multiplier remains analogous to classical settings, the investigation of its limiting distribution in our framework relies on fundamentally different techniques. In particular, our analysis draws heavily from the state-of-the-art {\it nonlinear large deviation theory}.

A related direction was explored in \cite{ghosh2023} (also cf. \cite{ghosh_abc}), where the authors studied semi-parametric maximum likelihood estimation (MLE) under constraints, where the optimal Lagrange multiplier appeared as a key analytical device. While \cite{ghosh2023} focuses largely on distributions on Euclidean spaces, it includes suggestive empirical results for constrained likelihood maximization in a network-based setup, linking the optimal Lagrange multiplier to degeneracies in the exponential random graph model (ERGM). 

In contrast, our approach maximizes entropy rather than likelihood, which is more canonical from a statistical networks perspective, and we demonstrate that the resulting optimizer is highly informative from a hypothesis testing viewpoint.\\

\textbf{Our Contributions}:

In this work, we develop a systematic procedure for hypothesis testing for random networks via an entropy maximization approach with constraints on motif counts, which serves as an analogue of moment constraints for network-structured data. In particular, we address the problems of goodness-of-fit testing and two sample testing for random networks in the setting of \emph{Exponential Random Graph} ({\it abbrv.} ERGM) models.  Our method leverages the optimal Lagrange multiplier associated with the maximum entropy principle, whose asymptotic statistical properties are established as a cornerstone of our analysis. On a technical level, we demonstrate that this optimal Lagrange multiplier can be characterized as the root of an equation involving an exponentially tilted version of the relevant empirical motif count, which allows us to perform the asymptotic analysis necessary to derive our statistical results. \\


We focus on three distinct setups:\\

\textit{Graphs with fixed number of vertices:} In this setting, our analysis is quite general and encompasses the case where the underlying graph has a fixed number of vertices, i.e., the size of the graph does not grow with the sample size. In particular, we assume that samples are generated from an arbitrary distribution over graphs on a finite vertex set. We establish the asymptotic normality of the Lagrange multiplier and construct a goodness-of-fit test based on expected motif counts. Specifically, we test whether the expected motif counts of the generating random graph model coincide with the expected motif count of a fixed distribution $\mathcal{G}_0$ on the same vertex set.\\

\textit{Dense graphs with increasing number of vertices:} 
Here we consider the ERGM model where the number of vertices is allowed to grow with the sample size. Based on the principle of entropy maximization under constraints we analyze the asymptotic behavior of the Lagrange multiplier which then yields a goodness-of-fit test based on the expected motif counts. While the study of the optimization problem relies on techniques similar to those used in examining the free energy of ERGMs \cite{chatterjee2013estimating}, understanding the behavior of the corresponding Lagrange multiplier 
requires more refined tools, drawing on the log-sum-exp approximation framework from the nonlinear large deviations literature \cite{chatterjee2016nonlinear}. Finally since the dense \er model ($G(N,p)$ with $p=O(1)$ is a special case of our framework, we obtain a goodness-of-fit test for the edge probability $p$. In this setting our results further yield a natural two-sample testing procedure.\\

\textit{Sparse graphs with increasing number of vertices:} In this case, we consider \er models in the sparse regime, where the count of the relevant motif $H$ converges to a Poisson distribution as the graph size grows. While a maximum-entropy–based test can be envisioned for more general sparse ERGM models—such as the framework in \cite{cook2024typical}—these models are not readily amenable to the techniques developed for the dense setting and thus posit significant challenges.
Accordingly, we focus on the sparse \er setup and show that the exponentially tilted empirical $H$ count satisfies a central limit theorem. Leveraging the asymptotic theory for $\mathcal{Z}$-estimators,
 we then establish the asymptotic normality of the centered and scaled Lagrange multiplier 
$\hat{\lambda}_n$. As in the dense regime, this result immediately yields a goodness-of-fit test, as well as a two-sample test, for the edge probability 
$p$.\\

 Although our analysis is centered on strictly balanced graph counts, which include important motifs such as cliques and cycles (see the notation section for a definition of strictly balanced graphs), the arguments can, in principle, be extended to general motif counts. However, we focus on a strictly balanced graph to maintain clarity of exposition. Moreover, while our study is restricted to the \er model (both sparse and dense) and ERGMs (in the dense regime), the underlying principle of entropy maximization is quite general and may be adapted to other random graph models.
 However, even in these cases, the analysis of the Lagrange multiplier is technically involved, and we leave the extensions to the more general models for future work.

\subsection{Outline of the paper}

 In Section \ref{background} we describe the ERGM model and the general hypothesis testing problem. We also describe our testing procedure based on the Lagrange multiplier which is in turn obtained from the empirical entropy maximization under constraints on motif counts. The analysis in the dense regime uses the theory of graphons and hence we also include a brief introduction to graphons in Section \ref{background}.  Section \ref{main_res} contains main results of our paper, while the proofs are deferred to the Appendix.     Finally, we conclude in Section
\ref{discussion} with a discussion on possible generalizations of our
results.

\subsection{Notations}
We will use the following notation throughout the paper. For a graph $G$ we will denote its set of vertices as $V(G)$ or simply $V$, when there is no chance of confusion. Similarly we will denote the edge set as $E(G)$ or $E$. The number of vertices and edges are denoted as $v(G)=|V(G)|$ and $e(G)=|E(G)|$ respectively. We will denote the \er random graph on $N$ vertices and edge-connection probability $p$ as $G(N,p)$.

Further, we define the density of G as \(d(G) := e(G)/v(G)\), and \(k(G) := \mathrm{max}\{d(H)\mid H\subseteq G, e(H)\geq 1\}\).
A graph is said to be balanced if \(k(G) = d(G)\), i.e., it is its densest subgraph. It is said to be strictly balanced if it is strictly denser than all other subgraphs. For example, cliques, cycles, trees are strictly balanced, whereas subgraph formed by the union of two cliques is balanced but not strictly balanced. As mentioned before we will restrict our motifs to strictly balanced subgraphs.

Let $Aut(G)$ be the group of automorphisms of the graph $G$ and we will denote the number of automorphisms as $a(G)=|Aut(G)|$.  Next we rigorously define the number of copies of a subgraph $H$ in a graph $G$. Let
\[
\homc(H,G)
:= \bigl|\{\varphi: V(H)\to V(G): \ \{u,v\}\in E(H)\Rightarrow \{\varphi(u),\varphi(v)\}\in E(G)\}\bigr|,
\]
\[
\injc(H,G)
:= \bigl|\{\varphi: V(H)\hookrightarrow V(G): \ \varphi \text{ injective and edge-preserving}\}\bigr|.
\]

We define unlabelled copy counts (of the motif $H$ in $G$)
\[
\mathcal{T}(H, G):=\frac{\injc(H,G)}{|\Aut(H)|}.
\]

We will also denote $\mathcal{T}(H, G)$ as 
 $H(G)$ and use the two notations interchangeably. Henceforth, we will refer to the number of unlabelled copies simply as "number of copies" or "counts" for brevity.

Now let \(N=v(G)\) and \(k=v(H)\). With these notations, we finally define the following notion of density of a motif, which is shown later to be asymptotically equal to the proportion of unlabelled copies of $H$ among all the possible $k$- vertex complete subgraphs of $G$ :

\[
t(H,G) := \frac{\homc(H,G)}{N^k}.
\]

We use the $O(.)$ and $o(.)$ to denote asymptotically bounded and asymptotically goes to $0$ respectively i.e we say $a_n=O(\gamma_n)$ if $\limsup|a_n/\gamma_n|\leq C$ and $a_n=o(\gamma_n)$ if $\limsup|a_n/\gamma_n|=0$. The notation $a_n\gg b_n$ will also be used to denote $b_n=o(a_n)$. Similarly $O_p(.)$ and $o_p(.)$ are reserved for the stochastic versions of the same quantities. We will use $\xrightarrow{P}$ to denote convergence in probability while $\xrightarrow{d}$ is used to denote convergence in distribution. The $\ell_1$ and $\ell_2$ norms of vectors are denoted by $\|.\|_1$ and $\|.\|$ respectively unless otherwise mentioned.



\section{Background \label{background}}

In this section, we explain the main ideas behind the construction of our statistical tests, which are rooted in the Maximum Entropy principle, and provide a brief overview of the mathematical tools used in the analysis. 

We begin in Section~\ref{subsec:gof} by describing the general testing framework. The null hypothesis naturally gives rise to a constrained optimization problem, in which the empirical entropy is maximized subject to structural constraints. The resulting Lagrange multipliers play a central role, and our test statistics are defined through their asymptotic behavior.

In Section~\ref{sec:lag_mult}, we place our method in a broader statistical context by relating it to classical hypothesis testing procedures based on Lagrange multipliers, most notably the score test. Section~\ref{ERGM_intro} then introduces the exponential random graph model (ERGM) setting and outlines the basic assumptions of the model. Finally, we review the two main analytical tools required for our proofs: graph limit theory in Section~\ref{sec:graph_limit} and large deviation principles for random graphs in Section~\ref{sec:LD_ERGM}.

\subsection{Hypothesis testing via MaxEnt principle}{\label{subsec:gof}}

Suppose $X_1,\ldots,X_n$ are i.i.d. $P\in \mathcal P$ where $\mathcal P$ is an arbitrary family of distributions.
Let $\vartheta(\cdot)$ be a statistical functional; that is, $\vartheta(P)$ is a real-valued function of $P$, defined for $P \in \mathcal P$. For example, if $\mathcal P$ is the set of all distributions on $\mathbb{R}$ with finite variance, $\vartheta(P)$ could be the mean of $P$ or the variance of $P$. The general problem of testing functional of distributions is as follows:

\begin{equation}
\H_0: \vartheta(P)\in \mathcal{C} \text{ vs } \H_1: \vartheta(P)\notin \mathcal{C}. \label{functional testing}
\end{equation}
for an arbitrary set $\mathcal C\subseteq \mathbb{R}$.

This problem reduces to testing of parameters in parametric models if the model is \textit{identifiable} i.e.  $\mathcal P =\{P_{\theta}\}_{\theta \in \Theta}$ and $P_{\theta_1}=P_{\theta_2}$ implies $\theta_1=\theta_2$. Indeed we can define $\vartheta(P_{\theta}):=\theta$ and state the hypothesis testing problem at the parameter level.

Now let $\mathfrak{G}_N$ be the set of graphs on $N$ vertices  (where $N=m_n$ may grow with the sample size $n$) and $\Theta\subseteq \mathbb{R}^d$. For $\theta\in \Theta$, we denote a parametrized family of probability measures on $\mathfrak{G}_N$ by $P_{\theta}$.
Let $\mathcal{G}_1,\mathcal{G}_2,\ldots \mathcal{G}_n$ be i.i.d. realizations random graphs on $N$ vertices with common distribution $P_{\theta}$. 
Fix $\theta_0\in \Theta$ and consider the function
\[
h:\mathfrak{G}_N\rightarrow \mathbb{R}
\]
given by
\[
G\mapsto \mathcal{T}(H,G)-\E_{G\sim \P_{\theta_0}}[\mathcal{T}(H,G)].
\]
Note that $h$ denotes the centered \(H\)-counts in the graph $G$ and observe that \begin{equation}\mathbb{E}_{G\sim P_{\theta}}[h(G)]=0\label{unbiased_count}\end{equation} if $\theta=\theta_0$. Now consider the following hypothesis testing problem:
\begin{equation}
\H_0: \mathbb{E}_{G\sim P_{\theta}}[h(G)]=0 \text{ vs } \H_1: \mathbb{E}_{G\sim P_{\theta}}[h(G)]\neq 0. \label{general_gof_statement}
\end{equation}
Indeed  defining $\vartheta(P_{\theta})=\mathbb{E}_{G\sim P_{\theta}}[h(G)]$ this is of the form \eqref{functional testing} when $\mathcal C=\{0\}$. In other words using the functional $\mathbb{E}_{G\sim \mathbb{P}_{\theta}}[h(G)]$ we want to test whether the random graph is sampled from the distribution $P_{\theta}$.

We state the problem in this form since many random graph models like the ERGM are not identifiable through a single subgraph count.
However, note that for the \er model $G(N,p)$ (in the general setup $p=p_N$ can vary as a function of $N$, but for brevity we will use $p$ when there is no chance of confusion) the distribution $\P_{\theta}$ can then be denoted as $\P_{p}$ and our hypothesis testing problem reduces to the following problem:
\begin{equation}
\H_0: p=p_0 \text{ vs } \H_1: p\neq p_0. 
\end{equation}
in other words \er models are identifiable with respect to the expected motif counts.


In the maximum-entropy under constraints paradigm one considers an arbitrary distribution $\{p(G)\}_{G\in \mathfrak{G}_N}$ and writes the following optimization problem:
\begin{equation}\max_{p(G)} -\sum_{G \in \mathfrak{G}_N}p(G)\log p(G) \quad \text{ such that } \sum_{G \in \mathfrak{G}_N} c(G)p(G)=C_*\label{max_ent_pop}\end{equation}
where the maximum is taken over all probability distribution on $\mathfrak{G}_N$ and $c(G)$ is a real valued function with the domain $\mathfrak{G}_N$. The constraint $\sum_{G \in \mathfrak{G}_N} c(G)p(G)=C_*$ is called topological constraint in \cite{squartini2017maximum}, and the authors discuss the properties of such maximizing  distributions under different topological constraints.
Our test is motivated from such maximum entropy principle; in particular to design the test we will maximize the entropy of the empirical distribution subject to the condition that the empirical version of \eqref{unbiased_count} holds. Let $w:=(w_1,\ldots,w_n)$ be the \textit{weight vector} where $w_i$ is the weight assigned to the graph $\mathcal{G}_i$. The empirical version of \eqref{max_ent_pop} is given by:
\begin{align*}
\argmax_{w\in \mathcal{W}} -\sum_{i=1}^n w_i\log w_i 
\text{ where }  \mathcal{W}:= & \left\{w:\sum_{i=1}^nw_ih(\mathcal{G}_i)=0, w_i\geq 0, \sum_{i=1}^n w_i=1\right\}.
\end{align*}

 We will denote $h(\mathcal{G}_i)$ by $h_{n,i}$ (and, whenever the number of vertices does not grow with $n$, by $h_i$). Note that this problem is feasible for $ \min h_i<0<\max h_i$. Further, as a function of the weight vector, this is a strictly concave maximization problem on a compact convex set. Hence, the maximum exists and is attained in the interior. This allows us to locate the point of maximum using Lagrange multipliers.

Take the Lagrangian to be
\[
L(w,\lambda,\alpha)=-\sum_{i=1}^n w_{i} \log w_{i}-\lambda \sum_{i=1}^n w_{i} h_{n,i}-\alpha\left(\sum_{i=1}^n w_{i}-1\right).
\]
Then the Lagrangian equations yield the following relations between the optimal variables:
\begin{align}
\sum_{i=1}^n h_{n,i} e^{-\hat{\lambda}_nh_{n,i}}=0\label{root_eqn}\\
1+\hat{\alpha}=\log \sum_{i=1}^n e^{-\hat{\lambda}_nh_{n,i}}\label{opt_alpha}\\
\hat{w}_i=e^{-(1+\hat{\alpha}+\hat{\lambda}_nh_{n,i})}\label{opt_weight}
\end{align}

We will design our test based on the asymptotic behaviour of the Lagrange multiplier $\hat{\lambda}_n$.
In particular, we will demonstrate that, under the null hypothesis, the random variable $\hat{\lambda}_n$ will converge to some constant $\lambda^{\circ}$ at a certain rate, which will lead to a natural goodness-of-fit test.

\subsection{Tests based on Lagrange multipliers}{\label{sec:lag_mult}}

In this subsection we review the classical Lagrange multiplier test associated with constrained likelihood maximization, that is widely used in the econometrics literature. Consider a simple setup where \(X_1,\dots,X_n\) are i.i.d. samples from a model with parameter \(\theta\in\mathbb{R}^p\) and log-likelihood \(\ell_n(\theta)=\sum_{i=1}^n \ell(X_i;\theta)\). Consider testing equality constraints
\begin{equation}\label{eq:H0}
H_0: r(\theta)=0,\qquad r: \mathbb{R}^p\to\mathbb{R}^q,\ q<p.
\end{equation}
The three classical test families are Wald, Likelihood Ratio (LR), and Lagrange Multiplier (LM/score/Rao) with the LM tests being attractive when the unrestricted MLE is hard but the restricted MLE under \(H_0\) is easy to compute. We briefly review the LM test below. 

 Let \(S_n(\theta)=\nabla_\theta \ell_n(\theta)\) be the score, \(I_n(\theta)=-\nabla_\theta^2 \ell_n(\theta)\) the observed information, and \(R(\theta)=\partial r(\theta)/\partial \theta\) be the \(q\times p\) Jacobian. Under \(H_0\), write \(\theta_0\) for the true parameter and \(\hat\theta_0\) for the restricted MLE. To compute the restricted MLE we consider the constrained maximization
\begin{equation*}
\max_{\theta\in\mathbb{R}^p} \ \ell_n(\theta)\quad\text{s.t.}\quad r(\theta)=0,
\end{equation*}
and the Lagrangian can be written as
\begin{equation*}
\mathcal{L}_n(\theta,\lambda)=\ell_n(\theta)+\lambda^\top r(\theta),\qquad \lambda\in\mathbb{R}^q.
\end{equation*}

At an optimum \((\hat\theta_0,\hat\lambda)\) , the first-order conditions  are
\begin{equation}\label{eq:FOC}
S_n(\hat\theta_0)+R(\hat\theta_0)^\top\hat\lambda=0,\qquad r(\hat\theta_0)=0.
\end{equation}
The distinguishing feature here is that the \emph{optimal Lagrange multiplier} \(\hat\lambda\) at the restricted optimum carries all the information needed for the test; the test statistic can be written as a (scaled) quadratic form in \(\hat\lambda\).
If regularity assumptions hold, \cite{silvey1958} showed that under $H_0$, the statistic $n^{-1/2}\hat{\lambda}$ is asymptotically normally distributed  
with variance-covariance matrix $\Big(R(\theta_0)I(\theta_0)^{-1}R(\theta_0)^\top\Big)^{-1}$, which is of rank $q$. 
Consequently 
\[
\frac{1}{n}\hat{\lambda}^\top \Big(R(\theta_0)I(\theta_0)^{-1}R(\theta_0)^\top\Big)^{-1} \hat{\lambda}
\]
is asymptotically distributed as $\chi^2$ with $q$ degrees of freedom, 
when $h(\theta_0)=0$, and \cite{silvey1958} proposed to choose as a region of acceptance of the hypothesis that $h(\theta_0)=0$ 
the set of $x$ for which
\[
\frac{1}{n}\hat{\lambda}^\top \Big(R(\theta_0)I(\theta_0)^{-1}R(\theta_0)^\top\Big)^{-1} \hat{\lambda} < \chi^2_{q,\,1-\alpha}.
\]

Since the seminal work of \cite{silvey1958}, the Lagrange Multiplier (LM) test has been extensively explored and refined. In \cite{silvey1959}, the author extended the framework to accommodate degenerate cases and established the asymptotic equivalence among the Wald, likelihood ratio (LR), and LM tests as the sample size $n$ grows large. Following this development, the LM test has become a foundational tool in econometric literature, forming the basis for a wide range of specification and diagnostic tests. In particular, building on the Lagrange Multiplier principle, \cite{bp_test} proposed a test for heteroskedasticity in regression models, \cite{godfrey1978testing1, godfrey1978testing2} developed LM tests for serial correlation in dynamic regressions, \cite{engle1982autoregressive} introduced LM-type procedures for detecting ARCH effects, and \cite{bera1984testing} constructed LM tests for assessing normality in limited dependent variable models, among others.

Our proposed tests are grounded in a similar foundational principle, yet they differ from the classical LM constructions in several important conceptual and methodological aspects, which we summarize below.
\begin{enumerate}
    \item The Lagrange Multiplier test developed in this paper is formulated for statistical networks models in which the number of nodes is allowed to grow with the sample size. The resulting asymptotic behavior in the sparse and dense regimes requires a completely different analytical approach from that of the fixed-dimensional parametric models commonly studied in econometrics.
    
    \item Our test is based on maximizing \emph{ entropy} rather than the likelihood. Maximizing the  empirical entropy corresponds to finding the root of the function 
    \[
    \sum_{i=1}^n h_{n,i} e^{-\hat{\lambda}_n h_{n,i}},
    \]
    whose behavior can be characterized using the analysis of the exponential random graph model (ERGM) free energy, as outlined in \cite{chatterjee2013estimating}. While a similar analysis could, in principle, be developed under the likelihood framework, its exact equivalence and comparative merit are not yet clear and are left for future investigation. Furthermore, since the maximum-entropy principle is natural in the analysis of network models (see \cite{squartini2017maximum}), our construction can be viewed as a direct extension of that principle.
    
    \item We construct our test based on a \emph{scalar} Lagrange multiplier. In contrast, the framework of \cite{silvey1958} suggests the possibility of tests based on a \emph{vector} of multipliers arising from hypotheses involving multiple parameters. While this generalization appears promising, it lies beyond the scope of the present work and will be explored in future research.
    \item Finally, in conjunction with \cite{ghosh2023} (see also the classical works \cite{silvey1958,silvey1959}), this article shows that the Lagrange multiplier obtained from maximizing empirical entropy (or empirical likelihood in the case of \cite{ghosh2023}) can be used effectively for statistical inference in complex models. More broadly, it points to a general overarching strategy whereby goodness-of-fit testing can be framed as a constrained optimization problem and the associated Lagrange multiplier can be used for devising the test.
\end{enumerate}

\subsection{Exponential random graph model \label{ERGM_intro}}

An \emph{Exponential Random Graph Model (ERGM)} defines a probability distribution on labeled simple graphs \(G=(V,E)\) with \(|V|=N\). Let $T_1$ be an edge and $T_2,\ldots,T_k$ be graphs with at least two edges. Fix sufficient statistics \(t(T_1,G),\dots,t(T_K,G)\) (homomorphism densities of the motifs such as edges, triangles etc.) and parameters \(\beta=(\beta_1,\dots,\beta_K)\in\mathbb{R}^K\). One considers the following probability distribution,
\begin{equation}
\label{eq:ergm}
\mathbb{P}_\beta(G)
= \frac{1}{Z_N(\beta)}\exp\!\left\{\,N^2\sum_{k=1}^K \beta_k\, t(T_k,G)\right\},
\end{equation}
where \(Z_N(\beta)\) is the normalizing constant. 
For $k=1$, we recover the \er model $G(N,p)$ where every edge is present with probability 
\[
p = p(\beta) = e^{2\beta_1}(1+e^{2\beta_1})^{-1},
\]
independent of each other. If $k \geq 2$ the probability measure ``encourages'' the presence of the corresponding subgraphs if the corresponding coefficient $\beta_k$ is positive. A canonical example is the edge-triangle model, with \(t(T_1,G)\) the edge density and \(t(T_2,G)\) the triangle density; see \cite{chatterjee2013estimating} for a detailed discussion.

ERGMs grew from exponential-family models in social network analysis \cite{frank1986markov,holland1981exponential,wasserman1994social}, and since then this model has been widely studied (cf. \cite{park2004solution,park2005solution,chatterjee2013estimating,radin2013phaseComplex,radin2013phaseERGM,yin2013critical,shalizi2013consistency,eldan2018exponential,mukherjee2023statistics,fang2024normal,fang2025conditionalcentrallimittheorems,bhamidi2008mixing,reinert2019approximating,bresler2024metastable,ganguly2024sub,winstein2025concentration,winstein2025quantitative} for a partial list). While the literature on this topic is extensive (and the above list is by no means exhaustive), we note that hypothesis testing within the ERGM framework has been relatively underexplored. In particular, while a few recent works such as \cite{xu2021stein} and \cite{reinert2019approximating} develop tests based on Stein discrepancies and \cite{mukherjee_signal} proposes tests based on sums of degrees, the wider literature on the topic is rather sparse.

Our tests are based on the treatment of ERGMs in \cite{chatterjee2013estimating}, where one uses graph limits to study the asymptotic behaviour of ERGMS. We note in passing that while this encapsulates the so-called {\it dense regime}, some recent work has also
 studied modifications which yield sparse graphs instead \cite{yin2017asymptotics,mukherjee_sparse,cook2024typical}. However these sparse ERGMs are outside the scope of the current paper and we study sparse models for the \er random graphs where the analysis proceeds via a different route.

 While analyzing the ERGMs, we consider the so called \textit{ferromagnetic regime} where the parameters $\beta_k$-s are positive for $k>1$. Let $e_k$ is the number of edges in $T_k$; we define the following functions:
 
\begin{equation}
\Phi_\beta(a) := \sum_{k=1}^K \beta_ke_k a^{e_k-1}, 
\qquad 
\varphi_\beta(a) := \frac{e^{2\Phi_\beta(a)}}{1+ e^{2\Phi_\beta(a)}}.
\end{equation}
 The parameter $\beta$ is said to lie in the \emph{subcritical regime} (see \cite{bhamidi2011mixing, chatterjee2013estimating}) if there is a unique solution $p(\beta) \in (0,1)$ to
\begin{equation} \label{eq:subcrit-eq}
\varphi_\beta(p) = p, \qquad \varphi'_\beta(p) \leq 1
\end{equation}
and the solution \(p(\beta)\) satisfies \(\varphi'_\beta(p(\beta)) < 1\).
It may be shown that \eqref{eq:subcrit-eq} is equivalent to $p$ satisfying
\begin{equation}
2\Phi_\beta(p) = \log\!\left(\frac{p}{1-p}\right).
\end{equation}

In this regime, ERGMs asymptotically resemble an \er $G(N,p)$ and is amenable to analysis. On the other hand, if the parameter $\beta$ lies outside this region (in other words, it lies in the supercritical regime), the graphs exhibit clustering and metastability, and are much harder to analyze (see \cite{winstein2025concentration,winstein2025quantitative} for details). We will restrict ourselves to the subcritical regime of a ferromagnetic ERGM and will use the tools from \cite{chatterjee2013estimating} to analyze the Lagrange multiplier and thus develop a test.

\subsection{Graph Limits \label{sec:graph_limit}}

To analyze the ERGMs in the asymptotic limit, we need to define several key quantities that appear in the variational formulations of this limit. These definitions build on the theory of graph limits and graphons, which provide a natural framework for analyzing dense graph sequences. In this section, we provide a quick introduction to some essential ingredients from this theory that are necessary to understand this connection.

Let $h: [0,1]^2 \rightarrow [0,1]$ be a symmetric measurable function, which, up to an equivalence (to be defined shortly),  we can interpret as a graphon or a graph limit. For any such function $h$, we define the following functionals:

\begin{itemize}
    \item \textbf{\(H\) density:} 
    \begin{equation}t(H,h): = \int_{[0,1]^{|V(H)|}} \prod _{(i,j)\in E(H)}h(x_i,x_j) \, dx_1 \ldots \, dx_k\label{H_density},\end{equation}.
    
    \item \textbf{Edge density:} 
    \begin{equation}e(h): = \int_{[0,1]^2} h(x,y) \, dx \, dy\label{edge_density},\end{equation}.
    
    \item \textbf{Entropy functional:} 
    \begin{equation}I(h): = \int_{[0,1]^2} \Big[ h(x,y)\log h(x,y) + (1-h(x,y))\log(1-h(x,y)) \Big] \, dx \, dy\label{entropy_func}.\end{equation}
\end{itemize}

 We denote by $W$ the space of all pre-graphons:
\[W = \{h: [0,1]^2 \rightarrow [0,1] \mid h \text{ is symmetric and measurable}\}.\]
Two pre-graphons, $h_1, h_2 \in W$, are equivalent ($h_1 \sim h_2$) if there exists a measure-preserving bijection $\sigma: [0,1] \to [0,1]$ such that $h_2(x,y) = h_1(\sigma(x),\sigma(y))$ for almost all $(x,y)$. The space of graphons is the quotient space $\tilde{W} = W/\sim$.

Since the functionals $t(H,h)$, $e(h)$, and $I(h)$ are invariant under this equivalence, the functions are well-defined on $\tilde{W}$. For an element $h\in W$, the corresponding equivalence class will be denoted by $\tilde{h}$. For a function $f$ on $W$ invariant under the above equivalence, we will use the same notation $f$ to denote its lift to the space $\tilde{W}$, so that $f(\tilde{h}) = f(h)$. For $p \in (0,1)$, we also define, by abuse of notation, the constant function $p \in \tilde{W}$ as $p(x,y) = p$ for all $x,y \in [0,1]$.

Throughout the dense regime, we will work with variational problems involving these functionals, in particular, 
\begin{equation}
\sup_{\tilde{h} \in \tilde{W}}\left(-\frac{\lambda}{\lvert Aut(H)\rvert} t(H,\tilde{h}) +  \sum_{k=1}^K \beta_k t(T_k, \tilde{h}) -\frac{1}{2} I(\tilde{h})\right)\label{var_form_1}
\end{equation}
where $t(H,\tilde{h}), t(T_k,\tilde{h})$ are the H-densities (see \eqref{H_density}) of the subgraphs $H$ and $T_k$, while $I(\tilde{h})$ is the entropy functional as defined in \eqref{entropy_func}. In fact \eqref{var_form_1} represents the free energy of the system (cf. \cite{chatterjee2013estimating} for an interpretation) and plays a central role in determining the limiting behavior of the empirical critical points. When specializing to \er
 random graphs $G(N,p)$ we will consider the setup where $p=\Theta(1)$. In this case we will have $t(H,p) = p^{|E(H)|}$ and $e(p) = p$.

\subsection{Large deviation in random graphs \label{sec:LD_ERGM}}

The large deviation principle for the \er random graph was first formulated  in the pioneering work of \cite{chatterjee2011large} where the authors extended Sanov's theorem to the \er model. Using the same result \cite{chatterjee2013estimating} showed the following result for the graphons.

Let $T:\widetilde{\mathcal W}\to\mathbb R$ be a bounded continuous function on the metric space $(\widetilde{\mathcal W},\delta_\square)$ (see \cite{chatterjee2013estimating} for exact definition of the metric). Fix $N$ and let $\mathfrak G_N$ denote the set of simple graphs on $N$ vertices. Then $T$ induces a probability mass function $p_N$ on $\mathcal G_N$ defined as
\[
p_N(G):=e^{N^2\,(T(\tilde G)-\psi_N)}
\]
where $\tilde G$ is the image of $G$ in the quotient space $\widetilde{\mathcal W}$. Noting that 
\begin{equation}
\psi_N=\frac{1}{N^2}\log\sum_{G\in\mathfrak G_N} e^{\,N^2 T(\tilde G)},
\end{equation}
the following result was proved in \cite{chatterjee2013estimating}:
\[
\psi:=\lim_{N\to\infty}\psi_N
=\sup_{\tilde h\in\widetilde{\mathcal W}}\big(T(\tilde h)-\frac{1}{2}I(\tilde h)\big).
\]

Using the above result the authors further proved that
the normalizing constant in \eqref{eq:ergm} behaves asymptotically as follows:

\begin{equation}
\lim_{N\to\infty}\frac{1}{N^2}
\log Z_N(\beta)
=\sup_{0\le u\le 1}\left(\sum_{i=1}^{K}\beta_i\,u^{e(T_i)}-\frac12 I(u)\right),
\label{ergm_lim_CD}
\end{equation}
where $I(u)=u\log u+(1-u)\log(1-u)$ and $e(T_i)$ is the number of edges in $T_i$. Here $I(u)$ can be interpreted as the entropy functional \eqref{entropy_func} for the constant function $u(x,y)=u$.

Understanding the behavior of the log normalizing constant is a crucial aspect in our analysis of the Lagrange multiplier. Although details are given in Appendix \ref{sec:proof_outline}, we note here that the solution of \eqref{root_eqn} can be seen as a critical point of the function $\sum_{i=1}^n e^{-{\lambda}h_{n,i}}$. It can be shown that the expected value of the same function is exactly the normalizing constant $Z_N(\beta)$ for a certain ERGM and hence \eqref{ergm_lim_CD} features heavily in the analysis of our solution $\hat{\lambda}_n$. However, the above argument is at the expectation level and the same argument cannot be easily generalized to analyze the empirical version, and we need more nuanced tools.

To analyze the critical point of the function $\sum_{i=1}^n e^{-\lambda h_{n,i}}$, we borrow tools from the large deviation theory of ERGMs.  While the large deviation results in the dense case follows from \cite{chatterjee2011large}, the proof in that work involves the use of Szemeredy regularity lemma, which does not supply quantitative error bounds. To bypass the difficulty, \cite{chatterjee2016nonlinear} (also see \cite{chatterjee2016introduction} for a unified treatment on the subject) developed the theory of nonlinear large deviation with quantitative error estimates. The key result in \cite{chatterjee2016nonlinear} is Theorem 1.6 where the authors approximate the log normalizing constant by a variational problem (similar to \eqref{ergm_lim_CD} but using a different analysis).

We apply this approximation to study the asymptotic behavior of the function 
 $\sum_{i=1}^n e^{-\lambda h_{n,i}}$  in the large $n$ limit. In particular, we characterize its critical point by analyzing the corresponding critical point of the limiting variational problem and exploiting the convexity properties of the approximating sequence of functions.

\section{Main Results \label{main_res}}

\subsection{ Generalities}
In this section, we present the main results of our paper.  The main goal will be to construct the goodness-of-fit and the two sample tests by analyzing the asymptotic behaviour of the Lagrange optimizer $\hat{\lambda}_n$. 

We consider two distinct setups. In the first case, the total number of vertices in the graph $N$, is set to a fixed integer  
$m$, and the results hold for general random graph models. In the second case, we allow the number of vertices ($N=m_n$) to grow with the sample size. In the latter growing sample case we first analyze the sparse regime, and subsequently the dense regime. 

In the sparse setting, we construct tests for the \er model, the general sparse ERGM models being out of the scope due to fundamental bottlenecks, as already explained earlier. In particular we consider samples from $G(m_n,p_n)$ with, $m_np_n^{k(H)}\rightarrow c$.
Next, we analyze the dense regime - here we formulate tests for the more general ferromagnetic ERGM model in the subcritical regime and derive the results for the dense \er model (samples obtained from $G(m_n,p)$ with $p=\Theta(1)$) as a special case. Finally we end this section with discussing a few key technical ingredients, related to the results stated in this section.\\

\textbf{Remarks:}

\begin{enumerate}
\item The existence and uniqueness of the Lagrange multiplier are established in
Lemma~\ref{exis_uniq}. As shown in the Appendix, this result holds 
across all regimes considered—namely, the fixed-size, sparse, and dense regimes.
Consequently, throughout this section we tacitly assume the existence and
uniqueness of the Lagrange multiplier in all stated theorems.

\item Although the general Goodness-of-fit test stated in \ref{general_gof_statement} is a two sided test, we will prove the results for one-sided tests due to technical reasons. A brief discussion in this regard is given in Section \ref{tech_ingre}.

\item It is possible to extend the results to the full sparse regime (\(p_n\rightarrow 0\) at any given rate) and also to an arbitrary fixed small graph \(H\), but these would involve additional technical and notational difficulties without adding to our conceptual understanding of the problem. Therefore, for clarity of exposition, we will restrict ourselves to a strictly balanced motif $H$ at the threshold (\(m_np_n^{k(H)}\rightarrow c\)), and refer to this as "the sparse" regime.

\end{enumerate}

\subsection{Networks of fixed size}

In this case our result is general: we assume that the graphs are sampled from any random graph model $\mathcal{G}$ with $m$ vertices. We will denote by $\mathbb{P}_{\mathcal{G}}$ the distribution of $h$ under the random graph model $\mathcal{G}$. 

In this setup, independent samples $\mathcal{G}_{1}, \ldots, \mathcal{G}_{n}$ are obtained from a random graph model $\mathcal{G}$ with a fixed number of vertices (say $m$). Let $\mathcal{H}$ be the set of values $h_i$ can assume. Note that the set $\mathcal{H}$ is finite as the number of vertices is finite.

The convergence of empirical distribution of \(H\) count then leads to the root of \eqref{root_eqn} converging to a particular value $\lambda^{\circ}$, which is captured by the following theorem.

\subsubsection{Consistency of the Lagrange Multiplier}

 To analyze the root of \eqref{root_eqn}, we first rewrite the LHS as:

\begin{equation}\frac{1}{n}\sum_{i=1}^n h_{i} e^{-\lambda h_{i}}=\sum_{a\in \mathcal{H}} a e^{-\lambda a}\frac{\lvert\{i\mid h_i=a\}\rvert}{n}\label{root_eq_form_2}.\end{equation}

Thus it is enough to analyze the root of the RHS of \eqref{root_eq_form_2}.

\begin{theorem}{\label{finite_consistency}}
Let  $\hat{\lambda}_n$ be the unique real root of the function
\[
\sum_{a\in \mathcal{H}} a e^{-\lambda a}\frac{\lvert\{i\mid h_i=a\}\rvert}{n}.
\] 
and $\lambda^{\circ}$ be the unique real root of the function 
\[\sum_{a\in \mathcal{H}}a e^{-\lambda a}\mathbb{P}_{\mathcal{G}}(h=a).\]
Then $\hat{\lambda}_n\rightarrow \lambda^{\circ}$ almost surely.
\end{theorem}

\subsubsection{Asymptotic normality of the Lagrange Multiplier}

Next we consider the asymptotic normality of the Lagrange multiplier. For notational convenience we will redefine $h_i$, centering $\mathcal{T}(H,\mathcal{G}_i)$ with an arbitrary constant $h_0$ as follows:
$$h_{i}=\mathcal{T}(H,\mathcal{G}_{i})-h_0.$$
Choosing $h_0$ as the expected motif count under a particular random graph model $\mathcal G_0$, then yields the results needed for the corresponding Goodness-of-fit test.

\begin{theorem}{\label{finite_an}}
Let $\mathcal{G}_{1}, \ldots, \mathcal{G}_{n}$ be sampled independently from a random graph model $\mathcal{G}$ with a fixed number of vertices.
If $\hat{\lambda}_n$ is the unique real root of the equation
\[
\sum_{i=1}^n h_{i} e^{-\lambda h_{i}} = 0,
\]
then the following holds: 
\[\sqrt{n}(\hat{\lambda}_n - \lambda^{\circ})\xrightarrow{{d}}\mathcal{N}\left(0, \frac{\mathbf{Var}\left[(\mathcal{T}(H,\mathcal{G})-h_0)e^{-\lambda^{\circ}(\mathcal{T}(H,\mathcal{G})-h_0)}\right]}{(\mathbb{E}\left[(\mathcal{T}(H,\mathcal{G})-h_0)^2e^{-\lambda^{\circ}(\mathcal{T}(H,\mathcal{G})-h_0)}\right])^2}\right)
\]
where  $\lambda^{\circ}$ is the unique real root of the equation \[\sum_{a\in \mathcal{H}} a e^{-\lambda a}\mathbb{P}_{\mathcal{G}}(h=a) = 0.\]
\end{theorem}

\subsubsection{Goodness-of-fit test}{\label{GOF_finite}}

In this subsection we consider Goodness-of-fit tests for graphs with fixed number of vertices. Recall that $\mathcal{T}(H,G)$ is the number of $H$ counts in the graph $G$. Let
$\mathcal{G}_0$ is a given random graph model on the same set of $m$ vertices and consider the statistic:
$$h(G)= \mathcal{T}(H,G)-\E_{G\sim \P_{\mathcal{G}_0}}[\mathcal{T}(H,G)]$$
which denote the centered \(H\) counts in the graph $G$.
Now consider the following hypothesis testing problem:

\begin{equation}
\H_0: \mathbb{E}_{G\sim\mathbb{P}_{\mathcal{G}}}[h(G)]=0 \text{ vs } \H_1: \mathbb{E}_{G\sim\mathbb{P}_{\mathcal{G}}}[h(G)]> 0. 
\end{equation}

Before using Theorem \ref{finite_an} to construct a test, we note that setting 
$h_0=\E_{G\sim \P_{\mathcal{G}_0}}[\mathcal{T}(H,G)]$
results in $\lambda^{\circ}=0$ under $\H_0$. Indeed 
defining $$\mathfrak{b}(\lambda):=\sum_{a\in \mathcal{H}}a e^{-\lambda a}\mathbb{P}_{\mathcal{G}}(h=a)$$
we have
$$\mathfrak{b}(0)=\sum_{a\in \mathcal{H}}a\mathbb{P}_{\mathcal{G}}(h=a)=\mathbb{E}_{G\sim\mathbb{P}_{\mathcal{G}}}[h(G)]=0,$$
where the last equality holds under $\H_0$. Since $\lambda^{\circ}$ is defined to the unique root of $\mathfrak{b}(\lambda)$, it follows that $\lambda^{\circ}=0$. 

Further defining 
\[\sigma^2_0:=\frac{\mathbf{Var}\left[(\mathcal{T}(H,\mathcal{G})-h_0)e^{-\lambda^{\circ}(\mathcal{T}(H,\mathcal{G})-h_0)}\right]}{\left(\mathbb{E}\left[(\mathcal{T}(H,\mathcal{G})-h_0)^2e^{-\lambda^{\circ}(\mathcal{T}(H,\mathcal{G})-h_0)}\right]\right)^2},\]
we note that for $\lambda^{\circ}=0$ we have
\[\sigma^2_0=\frac{\mathbf{Var}\left[(\mathcal{T}(H,\mathcal{G})-h_0)\right]}{\left(\mathbb{E}\left[(\mathcal{T}(H,\mathcal{G})-h_0)^2\right]\right)^2}=(\mathbf{Var}\left[(\mathcal{T}(H,\mathcal{G})-h_0)\right])^{-1},\]
and it follows that 
\[\hat{\sigma}^2_0=\left(\frac{1}{n-1}\sum_{i=1}^n\left(\mathcal{T}(H,\mathcal{G}_i)-\frac{1}{n}\sum_{i=1}^n\mathcal{T}(H,\mathcal{G}_i)\right)^2\right)^{-1}\]
is a consistent estimator of $\sigma^2_0$. We consider the following test:

\begin{test} [Goodness-of-fit test for fixed number of vertices] The test is given by {\label{gof_fixed}}
\[\phi_{n}=\mathbb{I}\{\hat{\lambda}_n>\frac{\hat{\sigma}_0}{\sqrt{n}}z_{\alpha}\}\]
where $z_{\alpha}$ is the $1-\alpha$ quantile of the standard normal distribution.
\end{test}

\begin{corollary}
 The Goodness-of-fit test for fixed number of vertices i.e. Test \ref{gof_fixed} is a level $\alpha$ consistent test.
\end{corollary}

\begin{proof}
 From Theorem \ref{finite_an} it follows that the above is asymptotically a level $\alpha$ test.
Indeed, under $\H_0$ $\lambda^{\circ}=0$, and hence
$$\P_{\H_0}(\text{Reject Null})=\E_{\H_0}\left[\mathbb{I}\left\{\hat{\lambda}_n>\frac{\hat{\sigma}_0}{\sqrt{n}}z_{\alpha}\right\}\right]=\P_{\H_0}[\sqrt{n}(\hat{\lambda}_n-\lambda^{\circ})>\hat{\sigma}_0z_{\alpha}]\rightarrow \alpha.$$
Further we observe that

\[\P_{\H_1}(\text{Reject Null})=\E_{\H_1}[\phi_n]=\P[\hat{\lambda}_n>\frac{\sigma_0}{\sqrt{n}}z_{\alpha}]=\P[\sqrt{n}(\hat{\lambda}_n-\lambda^{\circ})>-\sqrt{n}\lambda^{\circ}+\hat{\sigma}_0z_{\alpha}]\]
The RHS converges to $1$ since $\lambda^{\circ}>0$ when the graphs are sampled from a distribution corresponding to the alternate hypothesis, so that the test is consistent.

\end{proof}

\subsection{Networks of growing size: the sparse regime}{\label{subsec:sparse_regime}}


We now examine the problem where the graph size $N=m_n$ increases with the sample size $n$. In this section, we consider the setup where independent samples $\mathcal{G}_{n,1}, \ldots, \mathcal{G}_{n,n}$  are obtained from the \er model $G(m_n,p_n)$, with the edge probability $p_n$ (also denoted as $p(m_n)$ or simply $p$, when there is no chance of confusion) vanishing with the graph size.

We restrict ourselves to a strictly balanced subgraph \( H \) and assume the scaling
\begin{equation}
m_n p_n^{k(H)} \to c > 0,\quad
m_n p_{0,n}^{k(H)} \to c_0 > 0, \label{growth_cond_sparse}
\end{equation}
for constants $c$ and $c_0$ with $c \geq c_0$, where $p_{0,n}$ corresponds to the null hypothesis i.e. the graphs are sampled from $G(m_n,p_{0,n})$.
The above scaling ensures that the expected number of copies of \( H \) remains of constant order and we define
\begin{equation}
\mu(m): = \frac{m^{v(H)}p^{e(H)}}{|Aut(H)|}\label{mu_eq}.
\end{equation}
which is asymptotically equal to the expected number of copies of $H$ in $G(m,p)$ as $m\rightarrow\infty$. Let $\mathcal{G}_{n,i}$ be the $i^{th}$ copy of the $n$ i.i.d. copies sampled from $G(m_n,p_n)$. For notational convenience we will redefine $h_{n,i}$, centering $\mathcal{T}(H,\mathcal{G}_{n,i})$ with an arbitrary constant $h_0$ as follows:
$$h_{n,i}=\mathcal{T}(H,\mathcal{G}_{n,i})-h_0.$$

Later while constructing the Goodness-of-fit test we will choose the centering term to be the limiting value of the expected \( H \)-count under the null model i.e.  
\[
h_0=\lim_{n\rightarrow \infty}\mathbb{E}_{G\sim\mathbb{P}_{p_{0,n}}}[\mathcal{T}(H,\mathcal{G}_{n,i})].
\]
Since the underlying random graph model is \er, it can be easily shown that under the null, \( h_0 = \lim_{n\rightarrow \infty}\mu(m_n) \) where $\mu(m)$ is defined in \eqref{mu_eq}. For general values of the parameter $p_n$ (not necessarily the null) we assume that
\begin{equation}
\lim_{n \to \infty} \frac{\mu(m_n)}{h_0} = e^{\lambda^{\circ}}\label{lambda_0_defn},
\end{equation}
for some constant \( \lambda^{\circ} \geq 0\) (the existence of the limit guaranteed by \eqref{growth_cond_sparse}, non-negativity follows from $c \geq c_0$).

In this setting, we will show a central limit theorem for exponentially tilted empirical $H$ count $\sum_{i=1}^n h_{n,i} e^{-\lambda h_{n,i}}$. The asymptotic normality of the centered and scaled Lagrange multiplier $\hat{\lambda}_n$ then follows by analyzing the root of the above function. But first, we show the consistency of the root $\hat{\lambda}_n$.

\subsubsection{Consistency  of the Lagrange Multiplier}

We show the consistency in two steps. In the first step, we establish the convergence of the root of the population version of \eqref{root_eqn}.

Classical Poisson approximation results (see, for example, \cite{barbour1982poisson}, Eq. (2.9) and the ensuing discussion) imply that the \(H\)-count (denoted as $H_m$ in this setup) in the \(G(m,p)\) graph converges in total variation to the distribution of a Poisson random variable $Z_{\mu}$ with mean \( \mu=h_0e^{\lambda^{\circ}} \). 
Since the function \(x\mapsto x e^{-\lambda x}\) is bounded on any interval \([a,\infty)\), total variation convergence implies
\begin{equation*}\mathbb{E}[(H_m-h_0)e^{-\lambda(H_m-h_0)}]=\mathbb{E}[(Z_{\mu}-h_0)e^{-\lambda(Z_{\mu}-h_0)}](1\pm o(1)).
\end{equation*}
Monotonicity of the function \(x\mapsto x e^{-\lambda x}\) can be used to show  that the root of $\mathbb{E}[(H_m-h_0)e^{-\lambda(H_m-h_0)}]$ converges to that of $\mathbb{E}[(Z_{\mu}-h_0)e^{-\lambda(Z_{\mu}-h_0)}]$. This result is established in the following theorem.

\begin{theorem}{\label{poisson_consistency_1}}
    Let $H_m$ be the number of copies of a graph $H$ in a graph drawn from an \er random graph $G(m,p(m))$ i.e. $H_m=\mathcal{T}(H,G(m,p(m)))$ where $mp(m)^{k(H)} \rightarrow c$ as $m\rightarrow \infty$. Let  $\lambda_m$ be the unique real root of the equation
    $$\mathbb{E}[(H_m-h_0)e^{-\lambda(H_m-h_0)}]=0.$$
    Further, let $Z_{\mu}$ be a Poisson$(\mu)$ random variable with $\mu=h_0e^{\lambda^{\circ}}$ and  $\lambda^{\circ}_m$ be the unique real root of the equation
    $$\mathbb{E}[(Z_{\mu}-h_0)e^{-\lambda(Z_{\mu}-h_0)}]=0.$$
 Then as $m\rightarrow\infty$,
    $$\lvert \lambda_m-\lambda^{\circ}_m\rvert\rightarrow 0.$$
\end{theorem}

In the second step we establish the empirical version of the above result.

\begin{theorem}{\label{poisson_consistency_2}}
Let $\mathcal{G}_{n,1}, \ldots, \mathcal{G}_{n,n}$ be sampled independently from $G\left(m_n, (\frac{c}{m_n})^{1/k(H)}\right)$, where \(m_n\rightarrow \infty\).
If $\hat{\lambda}_n$ is the unique real root of
\[
\sum_{i=1}^n h_{n,i} e^{-\lambda h_{n,i}} = 0,
\]
where \(h_{n,i}=\mathcal{T}(H,\mathcal{G}_{n,i})-h_0\),
then $\hat{\lambda}_n \xrightarrow{P} \lambda^{\circ}$ where $\lambda^{\circ} = \log\left(\frac{c^{|V(H)|}/\lvert Aut(H)\rvert}{h_0}\right)$.
\end{theorem}

\subsubsection{Asymptotic normality of the Lagrange Multiplier}

Next we establish the asymptotic normality of the tilted empirical $H$ counts $\sum_{i=1}^n h_{n,i} e^{-\lambda h_{n,i}}$ (properly scaled and centred). We will need the following lemma which is a direct consequence of the Lindeberg-Feller Central Limit Theorem.

\begin{lemma}{\label{CLT_1}}
    In the setup of Theorem \ref{poisson_consistency_2}, we have
    \[\sum_{i=1}^n \frac{h_{n,i} e^{-\lambda h_{n,i}} - \mathbb{E}[h_{n,1} e^{-\lambda h_{n,1}}]}{\sqrt{n Var(h_{n,1} e^{-\lambda h_{n,1}})}}\xrightarrow{d} \mathcal{N}(0,1)\]
    for every \(\lambda>0\).
\end{lemma}


Finally we are ready to state the asymptotic normality of the root. Since the function $\sum_{i=1}^n h_{n,i} e^{-\lambda h_{n,i}}$ exhibits asymptotic normality, using standard $\mathcal{Z}$ estimator techniques we establish the asymptotic normality of the root $\hat{\lambda}_n$.

\begin{theorem}{\label{CLT_root}}
Let $\mathcal{G}_{n,1}, \ldots, \mathcal{G}_{n,n}$ be sampled independently from $G\left(m_n, (\frac{c}{m_n})^{1/k(H)}\right)$, where \(m_n\gg n^{\frac{1}{2(2-k(H))}}\) is a sequence of natural numbers.
Let $\hat{\lambda}_n$ is the unique real root of
\[
\sum_{i=1}^n h_{n,i} e^{-\lambda h_{n,i}} = 0,
\]
where  \(h_{n,i}=\mathcal{T}(H,\mathcal{G}_{n,i})-h_0\). Further let \(\lambda^{\circ} = \log\left(\frac{\mu}{h_0}\right)\) , and $\mu=c^{|V(H)|}/\lvert Aut(H)\rvert$,
then the following holds: 
\[\sqrt{n}(\hat{\lambda}_n - \lambda^{\circ})\xrightarrow{d}\mathcal{N}\left(0, \frac{\mathbf{Var}\left[(Z_{\mu}-h_0)e^{-\lambda^{\circ}(Z_{\mu}-h_0)}\right]}{(\mathbb{E}\left[(Z_{\mu}-h_0)^2e^{-\lambda^{\circ}(Z_{\mu}-h_0)}\right])^2}\right).
\]
\end{theorem}


\subsubsection{Goodness-of-fit test}

In this subsection we consider Goodness-of-fit tests for sparse Erdos-Renyi random graph, in particular we assume that the sample $\mathcal{G}_{n,1}, \ldots, \mathcal{G}_{n,n}$ is generated from $G\left(m_n, (\frac{c}{m_n})^{1/k(H)}\right)$ and we want to test between the following hypotheses:

\begin{equation}
\H_0: c=c_0 \text{ vs } \H_1: c>c_0. 
\end{equation}
 Let 
\begin{equation}\sigma^2_0=\frac{\mathbf{Var}\left[(Z_{\mu_0}-h_0)e^{-\lambda^{\circ}_{c_0}(Z_{\mu_0}-h_0)}\right]}{\left(\mathbb{E}\left[(Z_{\mu_0}-h_0)^2e^{-\lambda^{\circ}_{c_0}(Z_{\mu_0}-h_0)}\right]\right)^2}.\label{sigma_sparse_form}\end{equation}
Define $\mu=c^{v(H)}/\lvert Aut(H)\rvert$, $\mu_0=c_0^{v(H)}/\lvert Aut(H)\rvert$, $\lambda^{\circ}_{c_0} = \log\left(\frac{\mu_0}{h_0}\right)$, and $\lambda^{\circ}_c = \log\left(\frac{\mu}{h_0}\right)$. 



Further we note that $h_0$ can be chosen to be $\mu_0$ in which case $\lambda^{\circ}_{c_0}=0$ and $\sigma^2_0=1/\mu_0$ (by plugging in $\lambda^{\circ}_{c_0}=0$ in \eqref{sigma_sparse_form}).
We are now ready to construct the test:

\begin{test}[Goodness-of-fit test in the sparse regime]{\label{gof_sparse}} The test is given by 
\[\phi_{n}=\mathbb{I}\{\hat{\lambda}_n>\frac{1}{\sqrt{\mu_0n}}z_{\alpha}\}\]
where $z_{\alpha}$ is the $1-\alpha$ quantile of the normal distribution.  
\end{test}

\begin{corollary}
    The Goodness-of-fit test in the sparse regime i.e. Test \ref{gof_sparse} is a level $\alpha$ consistent test.
\end{corollary}

\begin{proof}
 From Theorem \ref{CLT_root} it follows that the above is a level $\alpha$ test.
Indeed, under $\H_0$ $\lambda^{\circ}_{c_0}=0$, and hence
$$\P_{\H_0}(\text{Reject Null})=\E_{c_0}\left[\mathbb{I}\left\{\hat{\lambda}_n>\frac{1}{\sqrt{\mu_0n}}z_{\alpha}\right\}\right]=\P_{c_0}\left[\frac{\sqrt{n}(\hat{\lambda}_n-\lambda_{c_0}^{\circ})}{1/\sqrt{\mu_0}}>z_{\alpha}\right]\rightarrow \alpha$$
Further we observe that

\begin{align*}\P_{\H_1}(\text{Reject Null})&=\E_{c}[\phi_n]=\P[\hat{\lambda}_n>\frac{1}{\sqrt{\mu_0n}}z_{\alpha}]\\
&=\P[\sqrt{n}(\hat{\lambda}_n-\lambda^{\circ}_c)>-\sqrt{n}\lambda^{\circ}_c+\frac{1}{\sqrt{\mu_0}}z_{\alpha}].\end{align*}
The RHS converges to $1$ as $\lambda_c^{\circ}>0$ when $c>c_0$ (corresponding to the alternate hypothesis), so the test is consistent.
\end{proof}

\subsubsection{Two sample test}

Next, we consider the two-sample setting in which we observe $n_1$ i.i.d.\ graphs drawn from 
$G\!\left(m_{n_1}, (\frac{c_1}{m_{n_1}})^{1/k(H)}\right)$ and $n_2$ i.i.d.\ graphs drawn from 
$G\!\left(m_{n_2}, (\frac{c_2}{m_{n_2}})^{1/k(H)}\right)$. Our goal is to test the null hypothesis 
$c_1=c_2$, where $c_1$ and $c_2$ are some unknown positive reals with $c_j>c_0$ for $j\in \{1,2\}$ and some known $c_0>0$.  When $m_{n_1}=m_{n_2}$, this reduces to testing equality of edge probabilities; in the 
general case, where $m_{n_1}\neq m_{n_2}$, it corresponds to testing whether the two graph sequences 
have asymptotically identical subgraph counts. 

The two-sample test is obtained through a careful and indirect application of
Theorem~\ref{CLT_root}. Rather than reformulating the null hypothesis as a
constrained optimization problem, we exploit the asymptotic theory developed
earlier in a different manner. Specifically, we construct two auxiliary
hypothesis testing problems, whose corresponding test statistics can be
appropriately combined to yield the desired two-sample test statistic.

Set
\[
h_0 = \frac{c_0^{|V(H)|}}{\mathrm{Aut}(H)} .
\]
For the first sample, we consider the auxiliary hypothesis testing problem
\[
\H_0: c = c_0 \quad \text{vs.} \quad \H_1: c = c_1 .
\]
Define
\[
h_{n_1,i} = h_{\mathcal{G}_{n_1,i}} - h_0 ,
\]
and let $\hat{\lambda}_{n_1}$ denote the solution to
\[
\sum_{i=1}^{n_1} h_{n_1,i} e^{-\lambda h_{n_1,i}} = 0 .
\]
Let
\[
\mu_1 = \frac{c_1^{|V(H)|}}{\lvert \mathrm{Aut}(H) \rvert},
\qquad
\lambda^{\circ}_{c_1} = \log\left( \frac{\mu_1}{h_0} \right).
\]
Then, by Theorem~\ref{CLT_root}, we have
\begin{equation}
\sqrt{n_1}\bigl(\hat{\lambda}_{n_1} - \lambda_{c_1}^{\circ}\bigr)
\xrightarrow{d}
\mathcal{N}\!\left(
0,
\frac{\mathbf{Var}\!\left[(Z_{\mu_1}-h_0)e^{-\lambda^{\circ}(Z_{\mu_1}-h_0)}\right]}
{\bigl(\mathbb{E}\!\left[(Z_{\mu_1}-h_0)^2 e^{-\lambda^{\circ}(Z_{\mu_1}-h_0)}\right]\bigr)^2}
\right).
\label{lambda_1_conv}
\end{equation}

Similarly, for the second sample we consider the auxiliary hypothesis testing
problem
\[
\H_0: c = c_0 \quad \text{vs.} \quad \H_1: c = c_2 ,
\]
which yields
\begin{equation}
\sqrt{n_2}\bigl(\hat{\lambda}_{n_2} - \lambda_{c_2}^{\circ}\bigr)
\xrightarrow{d}
\mathcal{N}\!\left(
0,
\frac{\mathbf{Var}\!\left[(Z_{\mu_2}-h_0)e^{-\lambda^{\circ}(Z_{\mu_2}-h_0)}\right]}
{\bigl(\mathbb{E}\!\left[(Z_{\mu_2}-h_0)^2 e^{-\lambda^{\circ}(Z_{\mu_2}-h_0)}\right]\bigr)^2}
\right).
\label{lambda_2_conv}
\end{equation}

When $c_1 = c_2$, we have $\mu_1 = \mu_2$, and consequently
$\lambda^{\circ}_{c_1} = \lambda^{\circ}_{c_2}$. It therefore follows from
\eqref{lambda_1_conv} and \eqref{lambda_2_conv} that the difference
$\hat{\lambda}_{n_1} - \hat{\lambda}_{n_2}$ serves as a natural test statistic
for testing the null hypothesis $H_0: c_1 = c_2$.

To formally construct a two-sample test, one must establish the asymptotic
normality of $\hat{\lambda}_{n_1} - \hat{\lambda}_{n_2}$ under appropriate
centering and scaling. While it may appear that this follows directly from
combining two instances of Theorem~\ref{CLT_root}, the possibility of unequal
sample sizes $n_1 \neq n_2$ necessitates additional technical arguments. The
following theorem makes this precise.

\begin{theorem}[Two sample asymptotics]{\label{two_sample_asymptotics}}

Let the graphs $\mathcal{G}_{{n_1},1}, \ldots, \mathcal{G}_{{n_1},{n_1}}$ be sampled independently from $G\left(m_{n_1}, (\frac{c_1}{m_{n_1}})^{1/k(H)}\right)$ and $\tilde{\mathcal{G}}_{n_2,1}, \ldots, \tilde{\mathcal{G}}_{n_2,n_2}$ be sampled independently from $G\left(m_{n_2}, (\frac{c_2}{m_{n_2}})^{1/k(H)}\right)$ where \(m_{n_i}\gg{n_i}^{\frac{1}{2(2-k(H))}}\) are sequences of natural numbers for $i \in \{1,2\}$ and $n_1/n_2 \to \rho \in (0, \infty)$.

Let $\hat{\lambda}_{n_1}$ be the unique real root of
\[
\sum_{i=1}^{n_1} h_{n_1,i} e^{-\lambda h_{n_1,i}} = 0,
\] and $\hat{\lambda}_{n_2}$ be the unique real root of
\[
\sum_{i=1}^{n_2} h_{n_2,i} e^{-\lambda h_{n_2,i}} = 0,
\]
where  \(h_{n_1,i}=\mathcal{T}(H,\mathcal{G}_{n_1,i})-h_0\) and  \(h_{n_2,i}=\mathcal{T}(H,\mathcal{G}_{n_2,i})-h_0\). 

Now let 
 \[\sigma^2_{n_1,n_2}=\frac{\mathbf{Var}\left[(Z_{\mu_1}-h_0)e^{-\lambda_{c_1}^{\circ}(Z_{\mu_1}-h_0)}\right]}{n_1\left(\mathbb{E}\left[(Z_{\mu_1}-h_0)^2e^{-\lambda_{c_1}^{\circ}(Z_{\mu_1}-h_0)}\right]\right)^2}+\frac{\mathbf{Var}\left[(Z_{\mu_2}-h_0)e^{-\lambda_{c_2}^{\circ}(Z_{\mu_2}-h_0)}\right]}{n_2\left(\mathbb{E}\left[(Z_{\mu_2}-h_0)^2e^{-\lambda_{c_2}^{\circ}(Z_{\mu_2}-h_0)}\right]\right)^2}
\]
and
 \(\lambda^{\circ}_{c_j}= \log\left(\frac{\mu_j}{h_0}\right)\) , and $\mu_j=c_j^{v(H)}/\lvert Aut(H)\rvert$ for $j\in\{1,2\}$.

Then the following holds: 
\[\frac{(\hat{\lambda}_{n_1} - \hat{\lambda}_{n_2})-(\lambda^{\circ}_{c_1}-\lambda^{\circ}_{c_2})}{\sigma_{n_1,n_2}}\xrightarrow{{d}} \mathcal{N}(0,1).\]
\end{theorem}
Suppose we want to test between the following hypotheses
\begin{equation*}
\H_0: c_1=c_2 \text{ vs } \H_1: c_1 \neq c_2. 
\end{equation*}
Under the null hypothesis, $\lambda^{\circ}_{c_1}=\lambda^{\circ}_{c_2}=\lambda^{\circ}$ (say). 
Further the variance simplifies to 
\[\sigma^2_{n_1,n_2}=\left(\frac{1}{n_1}+\frac{1}{n_2}\right)\frac{\mathbf{Var}\left[(Z_{\mu}-h_0)e^{-\lambda^{\circ}(Z_{\mu}-h_0)}\right]}{\left(\mathbb{E}\left[(Z_{\mu}-h_0)^2e^{-\lambda^{\circ}(Z_{\mu}-h_0)}\right]\right)^2}
\]
with $\mu=h_0e^{\lambda^{\circ}}$. Since $\hat{\lambda}_{c_j}\xrightarrow{P}\lambda^{\circ}$ for $j=\{1,2\}$ (see Theorem \ref{poisson_consistency_2}), $\hat{\lambda}=\frac{\hat{\lambda}_{n_1}+\hat{\lambda}_{n_2}}{2}$ is a consistent estimator of $\lambda^{\circ}$ and thus 
$\hat{\mu}=h_0e^{\hat{\lambda}^{\circ}}$ is a consistent estimator of $\mu$.

Let $Z_{\hat{\mu}}$ be a Poisson random variable with the random parameter $\hat{\mu}$. Further we denote the conditional expectation and conditional variance with respect to the random variable $Z_{\hat{\mu}}$ by  $\mathbb{E}\left[.\vert \hat{\lambda}\right]$ and $\mathbf{Var}\left[.\vert \hat{\lambda}\right]$ respectively.

The following lemma shows that $\mathbf{Var}\left[(Z_{\mu}-h_0)e^{-\lambda^{\circ}(Z_{\mu}-h_0)}\right]$ and $\mathbb{E}\left[(Z_{\mu}-h_0)^2e^{-\lambda^{\circ}(Z_{\mu}-h_0)}\right]$ can be consistently estimated.

\begin{lemma}{\label{lambda_cont}}
   \begin{align*}
   \mathbb{E}\left[(Z_{\hat{\mu}}-h_0)^2e^{-\hat{\lambda}(Z_{\hat{\mu}}-h_0)}\vert \hat{\lambda}\right]&\xrightarrow{P} \mathbb{E}\left[(Z_{\mu}-h_0)^2e^{-\lambda^{\circ}(Z_{\mu}-h_0)}\right]\\
       \mathbf{Var}\left[(Z_{\hat{\mu}}-h_0)e^{-\hat{\lambda}(Z_{\hat{\mu}}-h_0)}\vert \hat{\lambda}\right]&\xrightarrow{P}\mathbf{Var}\left[(Z_{\mu}-h_0)e^{-\lambda^{\circ}(Z_{\mu}-h_0)}\right]\\
   \end{align*} 
\end{lemma}

Thus defining $\hat{\sigma}^2_{n_1,n_2}$ by:
$$\hat{\sigma}^2_{n_1,n_2}:=\left(\frac{1}{n_1}+\frac{1}{n_2}\right)\frac{\mathbf{Var}\left[(Z_{\hat{\mu}}-h_0)e^{-\hat{\lambda}(Z_{\hat{\mu}}-h_0)}\vert \hat{\lambda}\right]}{\left(\mathbb{E}\left[(Z_{\hat{\mu}}-h_0)^2e^{-\hat{\lambda}(Z_{\hat{\mu}}-h_0)}\vert \hat{\lambda}\right]\right)^2}$$
 we observe that $\hat{\sigma}^2_{n_1,n_2}$ is a consistent estimator of $\sigma^2_{n_1,n_2}$ and consequently by Slutsky's theorem
\[\frac{\hat{\lambda}_{n_1} - \hat{\lambda}_{n_2}}{\hat{\sigma}_{n_1,n_2}}\xrightarrow{d} \mathcal{N}(0,1)\]
under $H_0$. Thus we can consider the following test

\begin{test}[Two sample test in the sparse regime]{\label{two_samp_sparse}}The test is given by
 \[\phi_{n}=\mathbb{I}\{|\hat{\lambda}_{n_1}-\hat{\lambda}_{n_2}|>\hat{\sigma}_{n_1,n_2}z_{\alpha/2}\}\]
where $z_{\alpha}$ is the $1-\alpha$ quantile of the normal distribution.   
\end{test}
 From Theorem \ref{two_sample_asymptotics} it follows that the above is a level $\alpha$ test. The consistency of this test can be shown as in the Goodness-of-fit case and we omit the details.

 \begin{corollary}
      The two sample test in the sparse regime i.e. Test \ref{two_samp_sparse} is a level $\alpha$ consistent test.
 \end{corollary}

\subsection{Networks of growing size: the Dense Regime}

In this setup we consider $n$ i.i.d. copies of random graphs sampled from the Exponential Random Graph Model (cf. \eqref{eq:ergm}) with probability distribution $\P_{\beta}(.)$ with

\begin{equation}
\mathbb{P}_\beta(G)
= \frac{1}{Z_N(\beta)}\exp\!\left\{\,N^2\sum_{k=1}^K \beta_k\, t(T_k,G)\right\}.
\end{equation}

We particularly focus on the ERGM model in the ferromagnetic (where the parameters $\beta_k$-s are positive for $k>1$) and the subcritical regime (see \eqref{eq:subcrit-eq}).
In contrast to the sparse setting, the dense regime presents considerably greater analytical challenges. A key technical reason for this is that the asymptotic normality of the exponentially tilted empirical $H$ count is not available in this regime. 

Instead, we interpret $\hat{\lambda}_n$ as a critical point of the function $\frac{1}{n}\sum_{i=1}^n e^{-\frac{\lambda}{m_n^{v(H)-2}} h_{n,i}}$ (note the different scaling). We consider a ferromagnetic ERGM in the subcritical regime and show that the logarithm of the optimization objective (described above)—after appropriate scaling— converges to the optimum of a convex functional defined over the space of {\it graphons}. In particular this log objective function can be interpreted as the free energy of a tilted exponential random graph models (ERGMs) and the convergence follows from the results established in \cite{chatterjee2013estimating}.

The study of the corresponding Lagrange multiplier 
$\hat{\lambda}_n$ requires a more delicate analysis, which relies on the log-sum-exp approximation techniques from the nonlinear large deviations literature \cite{chatterjee2016nonlinear}. We note that the results for dense \er graphs emerge as a natural special case of the ERGM framework discussed here, and complement the results for sparse \er graphs in Section \ref{subsec:sparse_regime}.

Recall that the $i^{th}$ centred motif count is given by:
$$h_{n,i}=\mathcal{T}(H,\mathcal{G}_{n,i})-h_0.$$
We set the centering term to be the expected \( H \)-count under the null model,  
\[
\mathbb{E}_{G\sim\mathbb{P}_{\beta_0}}[\mathcal{T}(H,G)]=:h_0
\]
where \(\mathbb{P}_{\beta_0}\) is an ERGM with parameter vector \(\beta_0\in\mathbb{R}^K\). In this dense setting, $h_0$ is of order $m_n^{v(H)}$, while large deviations are governed by an $m_n^2$ scale. This motivates the exponential tilt
\begin{equation}
F_n(\lambda) = \frac{1}{n}\sum_{i=1}^n \exp\left\{-\frac{\lambda}{m_n^{v(H)-2}} h_{n,i}\right\},\label{ergm_emp}
\end{equation}
so that $\frac{1}{m_n^2}\log F_n(\lambda)$ has a nondegenerate limit. 

\subsubsection{Consistency of the Lagrange Multiplier}

As in the sparse \er setup we first consider the consider the convergence of the critical point of a population version of \eqref{ergm_emp}.
\begin{theorem}
\label{normal_consistency_1}
Let $H_m$ be the number of copies of a graph $H$ in a graph drawn from a ferromagnetic ERGM $\P_{\bbeta}$ on $m$ vertices, where $\bbeta$ is in the subcritical regime. Let $\lambda_m$ be the unique critical point of the function
$\E[e^{-\frac{\lambda}{m^{v(H)-2}}(H_m-h_0)}]$ . We define the function
\begin{align}
\mathfrak{g}(\lambda) := &\frac{\lambda h_0}{m^{v(H)}} + \sup_{\tilde{h}\in\tilde{W}}\left(-\frac{\lambda}{\lvert \text{Aut}(H)\rvert} t(H,h) + \sum_{k=1}^K \beta_k t(T_k, h) - \frac{1}{2}I(h) \right)\nonumber\\
&- \sup_{\tilde{h}\in\tilde{W}}\left( \sum_{k=1}^K \beta_k t(T_k, h) - \frac{1}{2}I(h) \right)\label{var_form_lambda}.
\end{align}
Then the function $\mathfrak{g}(\lambda)$ is strictly convex and has a unique critical point. Further, if we let $\lambda^{\circ}$ be the unique critical point of this function, then as $m\to\infty$, we have $|\lambda_m - \lambda^{\circ}| \to 0$.
\end{theorem}
Next we establish the empirical version of the above theorem.

\begin{theorem}{\label{normal_consistency_2}}
Let $\mathcal{G}_{n,1}, \dots, \mathcal{G}_{n,n}$ be sampled independently from a ferromagnetic ERGM $\P_{\bbeta}$ (with $\bbeta$ in the subcritical regime) on $m_n$ vertices, where $\binom{m_n}{2}=o(\log n)$ and let $\hat{\lambda}_n$ be the unique critical point of $\frac{1}{n}\sum_{i=1}^n e^{-\frac{\lambda}{m_n^{v(H)-2}} h_{n,i}}$. Let $\lambda^{\circ}$ be the unique critical point of the function $\mathfrak{g}(\lambda)$ defined in \eqref{var_form_lambda}. Then as $n\to\infty$, $|\hat{\lambda}_n - \lambda^{\circ}| \xrightarrow{P} 0$.
\end{theorem}

\subsubsection{Sharp rates}

Next we establish sharper results for the above convergence which can be used to construct consistent tests. The analysis of the $\lambda^{\circ}=0$ case  and $\lambda^{\circ}<0$ are quite different and hence we state the results in two separate theorems.
\begin{theorem}
{\label{CLT_root_2}}
Suppose $\lambda^{\circ} = 0$ . Then, under the assumptions of Theorem \ref{normal_consistency_2}, we have
    $$ m_n^2(\hat{\lambda}_n-\lambda^{\circ})\xrightarrow{P} 0.$$

\end{theorem}

\begin{theorem}{\label{sharp_rates}}
Suppose $\lambda^{\circ}<0$ is such that maximizer is unique. Further define \(p_0: = \lim_{m_n\to\infty}\left(\frac{\lvert Aut(H) \rvert h_0}{m_n^{v(H)}}\right)^{\frac{1}{e(H)}}\). Then under the assumptions of the Theorem \ref{CLT_root_2}, we have
    $$ m_n^2(\hat{\lambda}_n-\lambda^{\circ})\xrightarrow{P} \frac{1}{\frac{(u^*)^{e(H)}}{\lvert Aut(H)\rvert}-\frac{p_0^{e(H)}}{\lvert Aut(H)\rvert}}$$
    where $u^*$ is the unique maximizer of 
\[
\sup_{u\in[0,1]}\left(-\lambda^{\circ} \frac{u^{e(H)}}{\lvert\text{Aut}(H)\rvert} + \sum_{k=1}^K \beta_k u^{e(T_k)} - \frac{1}{2}I(u)\right).
\]

\end{theorem}


\subsubsection{Goodness-of-fit test}

In this section we consider Goodness-of-fit tests for a ferromagnetic ERGM model in subcritical regime, analogous to our investigations in the sparse regime. In particular, we assume that the sample $\mathcal{G}_{n,1}, \ldots, \mathcal{G}_{n,n}$ is generated from 
$\P_{\bbeta}$ with $\bbeta$ in the subcritical regime (see Section \ref{ERGM_intro} for details). Now consider the following hypothesis testing problem:

\begin{equation}
\H_0: \mathbb{E}_{G\sim\mathbb{P}_{\bbeta}}[h(G)]=0 \text{ vs } \H_1: \mathbb{E}_{G\sim\mathbb{P}_{\bbeta}}[h(G)]< 0 
\end{equation}
where
$$h(G)= \mathcal{T}(H,G)-\E_{G\sim \P_{{\bbeta_0}}}[\mathcal{T}(H,G)]$$
denote the centered \(H\) counts in the graph $G$. We note that if $\bbeta=\bbeta_0$, then $\mathbb{E}_{G\sim\mathbb{P}_{\bbeta}}[h(G)]=0$, while the converse may not be true. As a consequence, we note that this hypothesis testing problem cannot be reduced to the simpler problem of testing parameter values as in the \er case.

We can construct the test as follows.

\begin{test}[Goodness-of-fit test in the dense regime]{\label{gof_dense}}Let $c_n\rightarrow \infty$ such that $c_n=o(m^2_n)$. We consider the following test 
\[\phi_{n}=\mathbb{I}\{\hat{\lambda}_n<-c^{-1}_n\}\]
\end{test}

\begin{corollary}
The Goodness-of-fit test in the dense regime  i.e. Test \ref{gof_dense} is a consistent test.
\end{corollary}

\begin{proof}
By Theorem \ref{CLT_root_2} under $H_0$, $\hat{\lambda}_n>-\frac{1}{2m^2_n}$ with probability converging to $1$ and hence $$\E_{H_0}[\phi_n]\rightarrow 0.$$
On the other hand for $\lambda_0<0$ (under $H_1$),
$$\P(\hat{\lambda}_n<-c_n^{-1})\geq \P\left(\hat{\lambda}_n<\lambda_0+\frac{1}{m^2_n}\left(\frac{1}{2}+\frac{1}{\frac{u^{e(H)}}{\lvert Aut(H)\rvert}-\frac{p_0^{e(H)}}{\lvert Aut(H)\rvert}}\right)\right).$$
The RHS converges to $1$ by Theorem \ref{sharp_rates}, so that the test is consistent. 
 For the \er model $G(m_n,p)$ the distribution $\P_{\bbeta}$ can then be denoted as $\P_{p}$ and our hypothesis testing problem reduces to testing whether $p=p_0$(null hypothesis) or $p<p_0$.
 \end{proof}

\subsubsection{Two sample test}

If we have two samples from two \er models, analogous to the sparse regime we can test whether they come from an \er with same parameter based on the Lagrange multipliers obtained from the two samples. While a theory of two sample testing for motif density can be shown to hold for the ERGM model as well, this does not automatically translate into two sample testing of parameters (see the discussion before Test \ref{gof_dense}), and is thus less interpretable. Therefore for two sample testing, we focus here only on the dense \er graphs.

\begin{theorem}{\label{dense_two_sample}}
Assume \(p,\tilde{p}\in (0,1-\epsilon)\) where \(\epsilon\in(0,1)\) is fixed. Let $\mathcal{G}_{n,1}, \ldots, \mathcal{G}_{n,n}$ be sampled independently from $G\left(m_n, p\right)$, $\tilde{\mathcal{G}}_{n,1}, \ldots, \tilde{\mathcal{G}}_{n,n}$ be sampled independently from $G\left(m_n, \tilde{p}\right)$, and independent of the first sample, where \(\binom{m_n}{2}=o(\log n)\) is a sequence of natural numbers.
Let $\hat{\lambda}_n$ be the unique critical point of
\[
\frac{1}{n}\sum_{i=1}^n e^{-\frac{\lambda}{m_n^{v(H)-2}} (\mathcal{T}(H,\mathcal{G}_{n,i})-{m_n \choose v(H)}(1-\epsilon)^{e(H)})} 
\]
and let $\tilde{\lambda}_n$ be the unique critical point of
\[
\frac{1}{n}\sum_{i=1}^n e^{-\frac{\lambda}{m_n^{v(H)-2}} (\mathcal{T}(H,\tilde{\mathcal{G}}_{n,i})-{m_n\choose v(H)}(1-\epsilon)^{e(H)})} 
\]
Then as $n\rightarrow\infty$,
    \[m_n^2(\hat{\lambda}_n-\tilde{\lambda}_n)\xrightarrow{P} 
    \begin{cases}
        0 & \text{if } p=\tilde{p}\\
        +\infty & \text{if } p>\tilde{p}\\
        -\infty & \text{if } p<\tilde{p}
    \end{cases}
    \]
\end{theorem}

We can easily design a two sample test based on the above theorem:

\begin{test}[Two sample test in the dense regime]{\label{two_sample_dense}}
The two sample test is given by: $\phi=\mathbb{I}(|\hat{\lambda}_n-\tilde{\lambda}_n|>c m^{-2}_n)$.
\end{test}

\begin{corollary}
    The two sample test in the dense regime i.e. Test \ref{two_sample_dense} is a consistent test.
\end{corollary}

\subsection{A few key technical ingredients \label{tech_ingre}}

\subsubsection{Uniqueness of roots }

In this section we discuss the existence and uniqueness of the Lagrange multiplier $\hat{\lambda}_n$ that is crucially used in Theorems \ref{finite_consistency}-\ref{dense_two_sample}. We note that in the sparse and fixed number of vertex case the Lagrange multiplier is scaled differently than in the dense regime.

\begin{lemma}{\label{exis_uniq}}
Let $h_{n,i}$ denote the centered motif counts defined in
Theorems~\ref{finite_consistency}--\ref{dense_two_sample}.
In the networks of fixed size and the sparse regime, let $\hat{\lambda}_n$ be a real root of the estimating equation
\[
\sum_{i=1}^n h_{n,i} e^{-\lambda h_{n,i}} = 0
\]
whenever such a root exists.
In the dense regime, let $\hat{\lambda}_n$ be a critical point of the objective function
\[
\frac{1}{n} \sum_{i=1}^n 
\exp\!\left(-\frac{\lambda}{m_n^{v(H)-2}}\, h_{n,i}\right),
\]
whenever such a critical point exists.
Then in both the cases $\hat{\lambda}_n$ exists and is unique.

\end{lemma}
\subsubsection{One sided vs two sided tests}

We note that the general Goodness-of-fit testing as discussed in \ref{subsec:gof}, in particular the statement in \eqref{general_gof_statement} is a two-sided hypothesis testing problem i.e. the alternate hypothesis considers a significant difference in either direction (greater than or less than) from a null hypothesis. On the other hand the tests we consider in this article are one-sided hypothesis tests only 
\begin{itemize}
    \item In the sparse regime $
\H_0: c=c_0 \text{ vs } \H_1: c>c_0$
\item In the dense regime: $\H_0: \mathbb{E}_{G\sim\mathbb{P}_{\bbeta}}[h(G)]=0 \text{ vs } \H_1: \mathbb{E}_{G\sim\mathbb{P}_{\bbeta}}[h(G)]< 0 $
\end{itemize}
In the sparse regime, the condition $c \geq c_0$ is equivalent to
$\lambda^{\circ} \geq 0$.
For $\lambda^{\circ} < 0$, the function
\[
x \longmapsto (x - h_0)\, e^{-\lambda (x - h_0)}
\]
is unbounded, and consequently Poisson convergence of subgraph counts does not
imply $L_1$ convergence of this function.
In fact, one can show that in Theorem~\ref{poisson_consistency_1},
\[
\mathbb{E}\!\left[(H_m - h_0)e^{-\lambda (H_m - h_0)}\right]
\to \infty \qquad \text{as } m \to \infty,
\]
so that the asymptotic behaviour of $\lambda_m$ and $\hat{\lambda}_n$ cannot be
characterized using the techniques developed in this paper.
For this reason, we restrict attention to a one-sided test in the sparse regime.

In the dense regime, the condition
\[
\mathbb{E}_{G\sim\mathbb{P}_{\beta}}[h(G)] \leq 0
\]
corresponds to $\lambda^{\circ} \leq 0$ in Theorem \ref{normal_consistency_1}.
When $\lambda^{\circ} > 0$, the supremum in \eqref{var_form_lambda} need not be
attained by a constant graphon.
Since our derivation of sharp rates (Theorems~\ref{CLT_root_2}
and~\ref{sharp_rates}) crucially relies on the structure of this supremum, this does not allow us to construct a valid test in this setting using the available techniques.
Accordingly, the two-sided testing problem lies outside the scope of the present
work and is left for future investigation.

\section{Discussion \label{discussion}}

In this paper, we construct tests for statistical network models based on the principle of constrained entropy maximization. Although we analyze the distributions of the optimal Lagrange multipliers (LM)—and the associated tests—for particular random graph families, the construction itself is general and, in principle, applies to any random graph model. Thus, this work provides a first step toward a general hypothesis-testing framework for random graphs, and we anticipate that LM-type tests can be used for many other models, much as LM tests are widely used in the econometrics literature. We conclude with several open questions and extensions.

In the sparse regime, we develop our tests for \er random graphs. It would be interesting to extend these ideas to sparse ERGMs, such as the model in \cite{cook2024typical}, where the Hamiltonian is a multivariate function of motif densities rather than the linear form considered here. A crucial ingredient in our analysis is the behavior of dense ferromagnetic ERGMs in the subcritical regime, where the free energy is approximated by a variational problem whose optimizer corresponds to \er{} graphs. This picture ceases to hold in sparse ERGMs: typical samples need not resemble an \er{} model (or a mixture) but rather an \er{} model with planted substructure. Extending our results to accommodate such sparsity-induced structure is compelling but poses significant challenges.

Another important direction is to move beyond the subcritical regime of ferromagnetic ERGMs. The uniqueness of the variational optimizer is central to our analysis of the optimal Lagrange multiplier, and as per the state of the art in the random graph literature this is so far available in the subcritical regime. Recent developments in the supercritical regime have been explored in \cite{winstein2025quantitative}; extending our testing results to that setting will require nontrivial technical advances.

We also note that, while under the null the \er model is fully identified—in the sense that expected motif counts and edge probability are in one-to-one correspondence—the same need not be true for general ERGMs. Consequently, our current tests determine whether the parameter vector 
$\bbeta$ lies in a specified subset, rather than testing individual coefficients for zero. We conjecture that formulating null hypotheses that constrain several motif counts simultaneously could yield tests for coefficient nullity in ERGMs. In that case, the Lagrange multiplier becomes a vector, and a natural test statistic would be a quadratic form in $\lambda$ in analogy with the classical Lagrange Multiplier tests (cf. Section \ref{sec:lag_mult}).

Finally, we note that in the sparse setting we obtain distributional limits for the centered and scaled Lagrange multiplier. In the dense case, such limit results are beyond the scope of this paper; instead, we establish a separation rate that distinguishes the competing hypotheses. We conjecture that asymptotic distributions can be derived in the dense setting as well, which would further strengthen our results.

\section*{Acknowledgements}
The authors thank Persi Diaconis, Vilas Winstein and Clarence Chew for helpful discussions, and Daren Wei and Huanchen Bao for their support. SG was supported in part by the Singapore MOE grants R-146-000-312-114, A-8002014-00-00, A-8003802-00-00, E-146-00-0037-01 and A-8000051-00-00. Most of this work was completed when R.N.K. was a research assistant at the National University of Singapore, supported in part by the Singapore MOE grants A-8000051-00-00, R-146-000-312-114, A-0009806-01-00 and A-0004586-00-00.


\bibliographystyle{plain}
\bibliography{ERGM}

\appendix

We begin by describing the layout of the appendices. Appendix \ref{sec:notn_details} provides a detailed description of the notations introduced in Section \ref{sec:intro}. Section \ref{sec:proof_outline} outlines the key ideas of the proofs, while the proofs of the main theorems are presented in Appendix \ref{sec_proof}. The proofs of Theorems \ref{normal_consistency_2} and \ref{sharp_rates} are lengthy; therefore, we provide only proof sketches in Appendix \ref{sec_proof} and defer the complete proofs to Appendix \ref{sec:3.8,3.10}. The technical lemmas are collected in Appendix \ref{sec:proof_lemmas}. Finally in Appendix \ref{sec:root_unicity}, we discuss the existence and uniqueness of the root of the equation \eqref{root_eqn} and give a proof of Lemma \ref{exis_uniq}.

\section{Notational Preliminaries \label{sec:notn_details}}

Before we discuss the proofs of the main theorems, we recall some graph theoretic concepts from the Notations subsection and elaborate them in this section.

For a graph $G$ we denote its set of vertices as $V(G)$ (or simply $V$) and the edge set as $E(G)$ or $E$. The number of vertices and edges are denoted as $v(G)=|V(G)|$ and $e(G)=|E(G)|$ respectively. Further, we define the density of G as \(d(G) := e(G)/v(G)\), and \(k(G) := \mathrm{max}\{d(H)\mid H\subseteq G, e(H)\geq 1\}\). A graph is said to be balanced if \(k(G) = d(G)\), i.e., it is its densest subgraph. It is said to be strictly balanced if it is strictly denser than all other subgraphs.
Let $Aut(G)$ be the group of automorphisms of the graph $G$ and we will denote the number of automorphisms as $a(G)=|Aut(G)|$. We will denote the \er random graph on $N$ vertices and edge-connection probability $p$ as $G(N,p)$. 

Here, we introduce a few other graph-theoretic concepts. This includes: \\

Homomorphisms and injective embeddings
\[
\homc(H,G)
:= \bigl|\{\varphi: V(H)\to V(G): \ \{u,v\}\in E(H)\Rightarrow \{\varphi(u),\varphi(v)\}\in E(G)\}\bigr|.
\]
\[
\injc(H,G)
:= \bigl|\{\varphi: V(H)\hookrightarrow V(G): \ \varphi \text{ injective and edge-preserving}\}\bigr|.
\]

Labeled and unlabeled copy counts (of the motif $H$ in $G$)
\[
\text{(\emph{labelled} copies)}\quad \injc(H,G), 
\qquad
\text{(\emph{unlabelled} copies)}\quad \frac{\injc(H,G)}{|\Aut(H)|}.
\]

We will denote the number of unlabelled copies of \(H\) in \(G\) as 
$\mathcal{T}(H, G)$ or $H(G)$ interchangeably.

Now let \(N=|V(G)|\) and \(k=|V(H)|\). With these notations, we are ready to define the following notions of density:

\[
t(H,G) := \frac{\homc(H,G)}{N^k},
\qquad
t_{\mathrm{inj}}(H,G) := \frac{\injc(H,G)}{\fall{N}{k}},
\quad \text{where } \fall{N}{k} := N(N-1)\cdots(N-k+1).
\]

Every homomorphism is either injective or has a collision (two distinct vertices of \(H\) mapped to the same vertex of \(G\)).
The number of collision patterns depends only on \(H\), and for each pattern the number of maps is $O(N^{k-1})$.
Hence $\homc(H,G) = \injc(H,G) + O_H(N^{k-1})$, so
\[
\frac{\homc(H,G)}{N^k} = \frac{\injc(H,G)}{N^k} + O_H\!\left(\frac{1}{N}\right).
\]
Finally, $\fall{N}{k} = N^k\bigl(1+O_k(1/N)\bigr)$, so
\[
\left|\frac{\injc(H,G)}{N^k} - \frac{\injc(H,G)}{\fall{N}{k}}\right|
= \injc(H,G)\,\Bigl|\frac{1}{N^k}-\frac{1}{\fall{N}{k}}\Bigr|
= O_H\!\left(\frac{1}{N}\right),
\]
which yields $t(H,G) - t_{\mathrm{inj}}(H,G) = O_H(1/N)$. For later use, we will also record the following equivalent version of this statement: \(\frac{t(H,G)}{\lvert Aut(H)\rvert}  = \frac{H(G)}{(N)_k} + O_H(1/N)\).
\\

To apply the large deviation theory it is important to view the \(t(H,G)\) as a function on \([0,1]^{{N\choose 2}}\). With that in mind, for every function \(f\) of \(G\), we will also consider its extension to \([0,1]^{{N\choose 2}}\) by defining \(f(x)\) to be the expectation of \(f(\mathbf{G})\) where \(\mathbf{G}\sim G(N,x)\) (every edge \(e\) is sampled independently of others with probability \(x_e\)). For instance, \(t(H,x):=\mathbb{E}_{\mathbf{G}\sim G(N,x)}[t(H,\mathbf{G})]\).

Note that this generalizes the notation \(t(H,G)\) as every graph \(G\) on \(m\) vertices can be represented as an element of \([0,1]^{{m\choose 2}}\) by fixing a labelling of the edges of \(K_m\) by \(\{1,\ldots {m\choose 2}\}\) and defining the adjacency vector \(x\in [0,1]^{{m\choose 2}}\) as \(x_i = 1\) if the edge labelled by \(i\) is present in \(G\) (and \(0\) otherwise) when \(G\) is viewed as a subgraph of \(K_m\). Then \(t(H,G) = t(H,x)\). 

\section{Proof Outline \label{sec:proof_outline}}

In this section, we provide a brief overview of the proof techniques used to establish the main results of the paper. As discussed earlier, our analysis proceeds under three distinct regimes: a fixed number of vertices, the sparse regime, and the dense regime. Accordingly, we divide the proof outline into three parts.

\subsection{Proofs for a fixed number of vertices}

In the fixed number of vertices setting, the results concern consistency and asymptotic normality of the Lagrange multiplier $\hat{\lambda}_n$. The consistency is established in Theorem \ref{finite_consistency}, while the asymptotic normality is shown in Theorem \ref{finite_an}.

\begin{itemize}
    \item \textbf{Theorem~\ref{finite_consistency}.} Using standard empirical process argument we show that the empirical mean of the tilted motif count
    (see~\eqref{root_eq_form_2}) converges uniformly over compact sets to its population counterpart. Specifically, we have

    \[
\sup\limits_{\lambda\in K}^{} \left\lvert \sum a e^{-\lambda a}\frac{\lvert\{i\mid h_i=a\}\rvert}{n} - \sum a e^{-\lambda a}\mathbb{P}_{\mathcal{G}}(h=a)\right\rvert \rightarrow 0
\]
almost surely for compact set $K$.
    Uniqueness of the root then yields consistency of the estimator.\\
    
    \item \textbf{Theorem~\ref{finite_an}.} We establish the asymptotic normality of the function given in ~\eqref{root_eq_form_2} after appropriate centering and scaling. Standard
    $\mathcal{Z}$-estimation theory then implies asymptotic normality of the corresponding root. The details are omitted as the proof follows in the same vein as the proof of Theorem \ref{CLT_root}.
\end{itemize}

\subsection{Proofs for the sparse regime}

In the sparse regime, proving consistency is more nuanced and is carried out in two steps. As a first step, in Theorem \ref{poisson_consistency_1} we show that root of the population version of~\eqref{root_eqn} converges to a constant $\lambda^{\circ}$. In the second step we establish  Theorem~\ref{poisson_consistency_2} - we combine a strong law of large numbers with
    Theorem~\ref{poisson_consistency_1}, to identify the limit of the empirical root of~\eqref{root_eqn}.
    
    Next, in Theorem \ref{CLT_root} we derive a central
    limit theorem for the exponentially tilted motif counts (LHS of \eqref{root_eqn})  and invoke $\mathcal{Z}$-estimation
    theory to obtain asymptotic normality of the root.  Finally in Theorem \ref{two_sample_asymptotics} we establish asymptotic normality for the difference of roots obtained from independent samples, which leads naturally to a two-sample test.
\begin{itemize}
    \item \textbf{Theorem~\ref{poisson_consistency_1}.} In this theorem we show that the root of the population version of~\eqref{root_eqn} converges to the root of the Poisson limit. Using classical Poisson convergence results for
    motif counts, we show that the expectation of the exponentially tilted motif count converges to
    the corresponding expectation under a Poisson limit. 
    \begin{equation}\mathbb{E}[(H_m-h_0)e^{-\lambda(H_m-h_0)}]=\mathbb{E}[(Z_{\mu}-h_0)e^{-\lambda(Z_{\mu}-h_0)}](1\pm o(1))\label{approx_poisson}.
\end{equation}
    Monotonicity of the function $xe^{-\lambda x}$ then ensures that $|\lambda_{m}-\lambda^{\circ}|\rightarrow 0$.\\
    
    \item \textbf{Theorem~\ref{poisson_consistency_2}.}
    This theorem proves consistency of the root of \eqref{root_eqn}. Using strong law of large numbers we have
    \[
\frac{\sum_{i=1}^n h_{n,i} e^{-\lambda h_{n,i}}}{n} - \mathbb{E}\left[h_{n,1} e^{-\lambda h_{n,1}}\right] \xrightarrow{\text{a.s.}} 0.
\]
Using standard arguments from empirical process theory and \eqref{approx_poisson}, we show that
\[
\sup_{\lambda \in K} \left\lvert\frac{1}{n} \sum_{i=1}^n h_{n,i} e^{-\lambda h_{n,i}} - \mathbb{E}\left[(Z_{\mu}-h_0)e^{-\lambda(Z_{\mu}-h_0)}\right]\right\rvert\to 0 \quad \text{a.s.}
\]
for compact set $K \subset [0, \infty)$. Uniqueness of root then ensures 
\(
 \hat{\lambda}_n \to \lambda^{\circ}\)in probability.\\
    
    \item \textbf{Theorem~\ref{CLT_root}.} In this theorem we prove asymptotic normality of the root of \eqref{root_eqn}. Let $\Psi_n(\lambda) = \frac{1}{n}\sum_{i=1}^n h_{n,i} e^{-\lambda h_{n,i}}$. Using a Taylor expansion about \(\lambda^{\circ}\), one shows that
\begin{equation}
\sqrt{n} (\hat{\lambda}_n - \lambda^{\circ}) = 
\frac{- \sqrt{n} (\Psi_n(\lambda^{\circ}) - \mathbb{E}[\Psi_n(\lambda^{\circ})])}{\dot{\Psi}_n(\lambda^{\circ}) + \frac{1}{2} (\hat{\lambda}_n - \lambda^{\circ}) \ddot{\Psi}_n(\tilde{\lambda}_n)} + o_p(1)
\label{Z_Taylor_short}
\end{equation}
where $\tilde{\lambda}_n$ is a point between $\hat{\lambda}_n$ and $\lambda_0$.
Further by arguments similar to those in Theorem \ref{poisson_consistency_2} we have
\[\dot{\Psi}_n(\lambda^{\circ}) \to \mathbb{E}\left[-(Z_{\mu}-h_0)^2e^{-\lambda^{\circ}(Z_{\mu}-h_0)}\right],\]
\[\ddot{\Psi}_n(\lambda^{\circ}) \to \mathbb{E}\left[(Z_{\mu}-h_0)^3e^{-\lambda^{\circ}(Z_{\mu}-h_0)}\right].\]

Then the consistency of \(\hat{\lambda}_n\) ensures that
the denominator in the RHS of \eqref{Z_Taylor_short} converges in probability: \[\dot{\Psi}_n(\lambda^{\circ}) + \frac{1}{2} (\hat{\lambda}_n - \lambda^{\circ}) \ddot{\Psi}_n(\tilde{\lambda}_n)\xrightarrow{P}\mathbb{E}\left[-(Z_{\mu}-h_0)^2e^{-\lambda^{\circ}(Z_{\mu}-h_0)}\right].\]
Next using Lemma \ref{CLT_1}, we show that 
\[\sqrt{n} (\Psi_n(\lambda^{\circ}) - \mathbb{E}[\Psi_n(\lambda^{\circ})])\xrightarrow{d}\mathcal{N}\left(0,\mathbf{Var}\left[(Z_{\mu}-h_0)e^{-\lambda(Z_{\mu}-h_0)}\right]\right).\]
Finally using Slutsky's theorem, we get
\[
\sqrt{n}(\hat{\lambda}_n - \lambda^{\circ}) \xrightarrow{d} \mathcal{N}\left(0, \frac{\mathbf{Var}\left[(Z_{\mu}-h_0)e^{-\lambda^{\circ}(Z_{\mu}-h_0)}\right]}{\mathbb{E}\left[(Z_{\mu}-h_0)^2e^{-\lambda^{\circ}(Z_{\mu}-h_0)}\right]^2}\right).
\]\\

    \item \textbf{Theorem~\ref{two_sample_asymptotics}.} Here, we establish asymptotic normality for the difference of roots obtained from independent samples. Using Theorem~\ref{CLT_root} we obtain,
   
\[
\sqrt{n_1}(\hat{\lambda}_{n_1} - \lambda^{\circ}_{c_1}) \xrightarrow{d} N(0, \sigma_1^2), 
\qquad 
\sqrt{n_2}(\hat{\lambda}_{n_2} - \lambda^{\circ}_{c_2}) \xrightarrow{d} N(0, \sigma_2^2),
\]
with $$\sigma^2_i=\frac{\mathbf{Var}\left[(Z_{\mu_i}-h_0)e^{-\lambda_{c_i}^{\circ}(Z_{\mu_i}-h_0)}\right]}{\mathbb{E}\left[(Z_{\mu_i}-h_0)^2e^{-\lambda_{c_i}^{\circ}(Z_{\mu_i}-h_0)}\right]^2}, \quad i\in\{1,2\}.$$
Using the fact that the samples are independent, we show that 
\[
T_n := \frac{(\hat{\lambda}_{n_1} - \hat{\lambda}_{n_2})-(\lambda^{\circ}_{c_1}-\lambda^{\circ}_{c_2})}{\sqrt{V_n^\ast}}\xrightarrow{d}\mathcal{N}(0,1),
\quad \text{ with }
V_n^\ast := \frac{\sigma_1^2}{n_1} + \frac{\sigma_2^2}{n_2}.
\]\\
We note that the last statement does not follow immediately from  Theorem~\ref{CLT_root} and a nuanced analysis is needed to prove this theorem as the sample sizes $n_1$ and $n_2$ are different. 
\end{itemize}

\subsection{Proofs for the dense regime}

In the dense regime, the general theory corresponds to a dense exponential random graph model (ERGM), with the dense \er model as a special case. The
optimal Lagrange multiplier is interpreted as a critical point of the objective
function~\eqref{ergm_emp}, and our analysis is based on the asymptotic behavior of this critical
point. 

In the spirit of
    Theorem~\ref{poisson_consistency_1}, as a first step, we show that the critical point of the population version
    of~\eqref{ergm_emp} converges to the solution of the variational problem~\eqref{var_form_lambda} in Theorem~\ref{normal_consistency_1}.  In the second step, in Theorem~\ref{normal_consistency_2}, we use log-sum-exp approximation techniques from
    the nonlinear large deviation literature to show that the critical point of
    \eqref{ergm_emp} converges to the same variational limit.

    Rates of convergence of the critical point $\hat{\lambda}_n$ are established in Theorem~\ref{CLT_root_2} and Theorem~\ref{sharp_rates}.
   In Theorem~\ref{CLT_root_2} we consider the case when the limiting parameter satisfies
    $\lambda^{\circ}=0$; we use Taylor expansion and concentration inequalities to show that the
    empirical critical point converges to zero at rate $o(m_n^{-2})$.
    In Theorem~\ref{sharp_rates}.we consider the case $\lambda^{\circ}<0$; a Taylor expansion around
    $\lambda^{\circ}$ combined with nonlinear large deviation techniques shows that
    $m_n^2(\hat{\lambda}_n-\lambda^{\circ})$ converges to a nondegenerate constant.
    Finally in Theorem~\ref{dense_two_sample} we combine Theorems~\ref{CLT_root_2} and
    \ref{sharp_rates} to construct a two-sample test for dense \er models.

\begin{itemize}
    \item \textbf{Theorem~\ref{normal_consistency_1}.} Here, we show that the critical point of the population version
    of~\eqref{ergm_emp} converges to the solution of the variational problem~\eqref{var_form_lambda}.
The proof is based on the large deviation principle for general ERGMs developed in \cite{chatterjee2013estimating}. The expectation $\E[e^{-\frac{\lambda}{m^{v(H)-2}}(H_m-h_0)}]$ can be approximated as the partition function of a new ERGM whose Hamiltonian is the sum of the original ERGM Hamiltonian and a perturbation term related to the $H$-count. In particular, we have:

\begin{align*}
&\E\left[e^{-\frac{\lambda}{m^{v(H)-2}}(H_m-h_0)}\right]\\ = &e^{\frac{\lambda h_0}{m^{v(H)-2}}}\times\frac{\sum \exp\{m^2(-\lambda t(H,G)/\lvert Aut(H)\rvert + O(\lambda/m) + \sum \beta_kt(T_k,G))\}}{\sum \exp\{m^2\sum \beta_kt(T_k,G)\}}
\end{align*}

From Theorem 3.1 of \cite{chatterjee2013estimating}, the normalized log-partition function of an ERGM converges to the supremum of a free energy functional over the space of graphons. Applying this result to our perturbed model, we obtain,
\begin{align}
&\frac{1}{m^2}\ln \left(\E_{\bbeta}\left[e^{-\frac{\lambda}{m^{v(H)-2}}(H_m-h_0)}\right]\right)\nonumber\\
\to & \frac{\lambda h_0}{m^{v(H)}} + \sup_{\tilde{h}\in\tilde{W}}\left(-\frac{\lambda}{|Aut(H)|} t(H,h) + \sum_{k=1}^K \beta_k t(T_k, h) - \frac{1}{2}I(h) \right)\nonumber\\
&- \sup_{\tilde{h}\in\tilde{W}}\left(\sum_{k=1}^K \beta_k t(T_k, h) - \frac{1}{2}I(h) \right)\nonumber\nonumber\\
=& \mathfrak{g}(\lambda)\label{var_formula}
\end{align}
for all $\lambda \in \mathbb{R}$.
Since $x\mapsto \ln(x)$ is strictly increasing, we need to show the convergence of critical points in the display above. The functions in the sequence above are all strictly convex in $\lambda$ by direct differentiation. The limiting function is convex being a sum of an affine part and a supremum over convex functions. Now the convergence of critical points will follow once we show that the limiting function is strictly convex. \\
    
    \item \textbf{Theorem~\ref{normal_consistency_2}.} 
 In this theorem our goal is to show the convergence in probability of the empirical critical point $\hat{\lambda}_n$ to the theoretical one $\lambda^{\circ}$ in the dense ERGM regime.

The proof hinges on the uniform convergence of the normalized log-partition function:
\[
\frac{1}{m_n^2}\log\left(\frac{1}{n}\sum_{i=1}^n e^{-\frac{\lambda}{m_n^{v(H)-2}} h_{n,i}}\right)
\]
to its limit. The core idea is to apply the log-sum-exp approximation framework from \cite{chatterjee2016nonlinear}.

Define $N(G)$  to be the count of a particular graph $G$ in the sample. Let $x$ be the adjacency vector of a graph and  $G_x$ be the graph corresponding to the adjacency vector $x$. We further define
$$N(x)=N(G_x), \quad H(x)=\mathcal{T}(G,H_x).$$ The energy function $f(x)$ is defined by 
$f(x)= \log\frac{N(x)}{n} - \frac{\lambda}{m_n^{v(H)-2}} H(x)$.

We rewrite the empirical objective function as a log-partition function over the space of all graphs on $m_n$ vertices:
\begin{equation}
\frac{1}{m_n^2}\log\left(\sum_{G} \frac{N(G)}{n} e^{-\frac{\lambda}{m_n^{v(H)-2}} (\mathcal{T}(H,G)-h_0)}\right) = \frac{\lambda h_0}{m_n^{v(H)}} + \frac{1}{m_n^2} \log \left(\sum_{x \in \{0,1\}^{\binom{m_n}{2}}} e^{f(x)}\right),
\label{f_rep}\end{equation}
 and then verify some smoothness conditions on the function $f$, under which we show that the following holds:

\begin{align} \frac{1}{m_n^2} \log \left(\sum_{x \in \{0,1\}^{\binom{m_n}{2}}} e^{f(x)}\right)=&\sup_{\tilde{h}\in\tilde{W}}\left(-\frac{\lambda}{|Aut(H)|} t(H,h) + \sum_{k=1}^K \beta_k t(T_k, h) - \frac{1}{2}I(h) \right)\nonumber\\
&- \sup_{\tilde{h}\in\tilde{W}}\left(\sum_{k=1}^K \beta_k t(T_k, h) - \frac{1}{2}I(h) \right)+o(1).\label{log_sum_exp_approx}\end{align}
Combining \eqref{f_rep} with \eqref{log_sum_exp_approx} we obtain 

\[
\frac{1}{m_n^2}\log\left(\frac{1}{n}\sum_{i=1}^n e^{-\frac{\lambda}{m_n^{v(H)-2}} h_{n,i}}\right)\rightarrow \mathfrak{g}(\lambda).
\]
Using convexity arguments as in Theorem \ref{normal_consistency_1}, we obtain convergence of the critical point $\hat{\lambda}_n$ to $\lambda^{\circ}$.\\

    \item \textbf{Theorem~\ref{CLT_root_2}.} 
In this theorem we consider the case when the limiting parameter satisfies
    $\lambda^{\circ}=0$ and show that the
    empirical critical point converges to zero at rate $o(m_n^{-2})$. We perform a Taylor expansion of the empirical critical point equation about $\lambda^{\circ}$ to obtain:
\begin{align*}
 m_n^2(\hat{\lambda}_n - \lambda^{\circ}) = \frac{\sum\limits_{i=1}^{n}\frac{h_{n,i}}{m_n^{v(H)}}e^{-\frac{\lambda^{\circ}}{m_n^{v(H)-2}}h_{n,i}}}{\sum\limits_{i=1}^{n}\frac{h_{n,i}^2}{m_n^{2v(H)}}e^{-\frac{\tilde{\lambda}_n}{m_n^{v(H)-2}}h_{n,i}} }
\end{align*}
For $\lambda^{\circ}=0$, we get
\begin{equation}
m_n^2 \,\hat{\lambda}_n
=
\frac{\frac{1}{n}\displaystyle \sum_{i=1}^{n} \frac{h_{n,i}}{m_n^{v(H)}}}
     {\frac{1}{n}\displaystyle \sum\limits_{i=1}^{n}\frac{h_{n,i}^2}{m_n^{2v(H)}}e^{-\frac{\tilde{\lambda}_n}{m_n^{v(H)-2}}h_{n,i}} }=\frac{\frac{1}{\sqrt{n}}\displaystyle \sum_{i=1}^{n} \frac{h_{n,i}}{m_n^{v(H)-1}}}{\frac1n\displaystyle \sum\limits_{i=1}^{n}\frac{h_{n,i}^2}{m_n^{2v(H)}}e^{-\frac{\tilde{\lambda}_n}{m_n^{v(H)-2}}h_{n,i}}}\frac{1}{\sqrt{n}{m_n}}.\label{lambda_0_asymp_1}
\end{equation}
We define
\[
    Z_n := \frac{1}{\sqrt{n}} \sum_{i=1}^{n} \frac{h_{n,i}}{m_n^{v(H)-1}}
\]
and show that
\begin{equation}
    \frac{m_n}{\sqrt[4]{n}} Z_n \xrightarrow{P} 0 .
    \label{num_conv_1_0}
\end{equation}
using concentration inequalities from \cite{ganguly2024sub} in the subcritical regime of the ferromagnetic ERGM.
 In the next step we show that 
 \begin{equation}V_n:= \sqrt[4]{n}\times m_n\times\frac{1}{n} \sum_{i=1}^{n} \frac{h_{n,i}^2}{m_n^{2v(H)}}e^{-\frac{\tilde{\lambda}_n}{m_n^{v(H)-2}}h_{n,i}} = \Omega_P\left(\frac{1}{m_n}\right).\label{denom_bound_0}\end{equation}
using a central limit theorem for subgraph counts in the subcritical regime of a ferromagnetic ERGM from \cite{winstein2025quantitative}.

Finally, we plug-in \eqref{num_conv_1_0} and \eqref{denom_bound_0} in \eqref{lambda_0_asymp_1} to conclude
that $m_n^2\hat{\lambda}_n\xrightarrow{P} 0$.\\

    \item \textbf{Theorem~\ref{sharp_rates}.} This theorem sharpens the consistency result of Theorem \ref{normal_consistency_2} by establishing the precise asymptotic behavior of $m_n^2(\hat{\lambda}_n-\lambda^{\circ})$. When $\lambda^{\circ}<0$, a Taylor expansion around
    $\lambda^{\circ}$ combined with nonlinear large deviation techniques shows that
    $m_n^2(\hat{\lambda}_n-\lambda^{\circ})$ converges to a nondegenerate constant.
First we use Taylor expansion of the empirical critical point equation about $\lambda^{\circ}$ to show that
\begin{align*}
 m_n^2(\hat{\lambda}_n - \lambda^{\circ}) = \frac{\sum\limits_{i=1}^{n}\frac{h_{n,i}}{m_n^{v(H)}}e^{-\frac{\lambda^{\circ}}{m_n^{v(H)-2}}h_{n,i}}}{\sum\limits_{i=1}^{n}\frac{h_{n,i}^2}{m_n^{2v(H)}}e^{-\frac{\tilde{\lambda}_n}{m_n^{v(H)-2}}h_{n,i}} }.
\end{align*}

To analyze the asymptotic behavior of the numerator and denominator separately, we employ two distinct variational formulations based on the log-sum-exp approximation.

For the numerator, we analyze the derivative of the normalized log-partition function. By Danskin's Theorem \cite{danskin1966theory} and the uniform convergence established in the proof of Theorem \ref{normal_consistency_2}, we have:
\begin{equation}
\frac{\sum\limits_{i=1}^{n}\frac{-h_{n,i}}{m_n^{v(H)}}e^{-\frac{\lambda^{\circ}}{m_n^{v(H)-2}}h_{n,i}}}{\sum\limits_{i=1}^{n}e^{-\frac{\lambda^{\circ}}{m_n^{v(H)-2}}h_{n,i}}} \xrightarrow{P} \frac{d}{d\lambda} \mathfrak{g}(\lambda) \bigg|_{\lambda=\lambda^{\circ}} = \frac{p_0^{e(H)}}{|Aut(H)|} - \frac{(u^*)^{e(H)}}{|Aut(H)|}
\label{num_conv}\end{equation}
where $u^*$ is the unique constant graphon that maximizes the variational problem defining $\mathfrak{g}(\lambda^\circ)$ from \eqref{var_form_lambda}.

For the denominator, we introduce a crucial auxiliary function by adding a squared $H$-count term:
\begin{align*}
\frac{1}{m_n^2} \log \sum_{G} \exp\left(-\frac{\lambda}{m_n^{v(H)-2}} h(G) + \log \frac{N(G)}{n} + \frac{\alpha}{m_n^{2v(H)-2}} h(G)^2\right)
\end{align*}
where $h(G) = \mathcal{T}(H,G) - h_0$. Once again using log-sum-exp approximation as in Theorem \ref{normal_consistency_2} we establish that the auxiliary function converges uniformly to its limiting variational form:
    \begin{align*}
     \sup_{\tilde{h}\in\tilde{W}}\bigg(&-\left(\frac{\lambda^{\circ}}{|Aut(H)|}+\frac{2\alpha p_0^{e(H)}}{|Aut(H)|^2}\right) t(H,h) + \sum_{k=1}^K \beta_k t(T_k, h) \\
     &+ \frac{\alpha}{|Aut(H)|^2} t(H,h)^2 - \frac{1}{2} I(h)\bigg) + \frac{\lambda^{\circ} p_0^{e(H)}}{\lvert Aut(H)\rvert} + \frac{\alpha p_0^{2e(H)}}{\lvert Aut(H)\rvert^2} \\&- \sup_{\tilde{h}\in\tilde{W}}\bigg( \sum_{k=1}^K \beta_k t(T_k, h) 
      - \frac{1}{2} I(h)\bigg).
    \end{align*}

Using Danskin's theorem \cite{danskin1966theory}, we compute the derivative of this limiting function with respect to $\alpha$ at $\alpha=0$, which captures the limiting behavior of the denominator. Since the supremum in the limit is uniquely attained at the constant graphon $u^*$, we obtain:
\begin{equation}
\frac{\sum\limits_{i=1}^{n}\frac{h_{n,i}^2}{m_n^{2v(H)}} e^{-\frac{\lambda^{\circ}}{m_n^{v(H)-2}} h_{n,i}}}{\sum\limits_{i=1}^{n} e^{-\frac{\lambda^{\circ}}{m_n^{v(H)-2}} h_{n,i}}}\xrightarrow{P} \left(\frac{(u^*)^{e(H)}}{|Aut(H)|} - \frac{p_0^{e(H)}}{|Aut(H)|}\right)^2
\label{denom_conv}\end{equation}

Combining the limits of the numerator and denominator yields the stated convergence:
\begin{equation}
m_n^2(\hat{\lambda}_n - \lambda^{\circ}) \xrightarrow{P} \frac{\frac{(u^*)^{e(H)}}{|Aut(H)|} - \frac{p_0^{e(H)}}{|Aut(H)|}}{\left(\frac{(u^*)^{e(H)}}{|Aut(H)|} - \frac{p_0^{e(H)}}{|Aut(H)|}\right)^2} = \frac{1}{\frac{(u^*)^{e(H)}}{|Aut(H)|} - \frac{p_0^{e(H)}}{|Aut(H)|}}.\label{sharp_rates_statement}
\end{equation}

    
    \item \textbf{Theorem~\ref{dense_two_sample}.} We combine Theorems~\ref{CLT_root_2} and
    \ref{sharp_rates} to construct a two-sample test for dense \er models.
\end{itemize}

\section{Proofs of the main Theorems \label{sec_proof}}

\begin{proof}[Proof of Theorem \ref{finite_consistency}]
Recall that $\hat{\lambda}_n$ is the unique zero of
\[
\sum a e^{-\lambda a}\frac{\lvert\{i\mid h_i=a\}\rvert}{n}
\]
where the sum ranges over all possible values of the statistic $h$ (and the number is finite in this case, as $N$ is fixed). By strong law of large numbers the above display converges, for each $\lambda$, to
\[
\sum a e^{-\lambda a}\mathbb{P}_{\mathcal{G}}(h=a)
\]
almost surely. 

Observe that for any compact subset $K\subseteq [0,\infty)$, the collection of functions $\{a\mapsto a e^{-\lambda a}\}_{\lambda\in K}$ is trivially a pointwise compact class and hence $\mathbb{P}_{\mathcal{G}}(h=\cdot)$ Glivenko-Cantelli (\cite{Vaart1998} Ch.19). Hence,
\[
\sup\limits_{\lambda\in K}^{} \left\lvert \sum a e^{-\lambda a}\frac{\lvert\{i\mid h_i=a\}\rvert}{n} - \sum a e^{-\lambda a}\mathbb{P}_{\mathcal{G}}(h=a)\right\rvert \rightarrow 0
\]
almost surely. Now $[0,\infty)=\cup_{n=1}^\infty [0,n]$ is a countable union. Hence,
\[
\sup\limits_{\lambda\in [0,n)}^{} \left\lvert \sum a e^{-\lambda a}\frac{\lvert\{i\mid h_i=a\}\rvert}{n} - \sum a e^{-\lambda a}\mathbb{P}_{\mathcal{G}}(h=a)\right\rvert \rightarrow 0 \quad \forall n\in\mathbb{N}
\]
almost surely. In other words, the convergence is uniform over compacts, almost surely.
We note that the root of the equation \[\sum a e^{-\lambda a}\mathbb{P}_{\mathcal{G}}(h=a)=0
\]
exists and is unique by Lemma \ref{existence_1} and \ref{uniqueness_1}.
Then by Hurwitz's theorem, $\hat{\lambda}_n\rightarrow \lambda^{\circ}$ almost surely, where $\lambda^{\circ}$ is the unique zero of $\sum a e^{-\lambda a}\mathbb{P}_{\mathcal{G}}(h=a)$.
\end{proof}

\begin{proof}[Proof of Theorem \ref{finite_an}]
The proof follows by the standard techniques of  proving CLT for $\mathcal{Z}$ estimators. In particular we perform a Taylor expansion of the function $\frac{1}{n}\sum_{i=1}^n h_{n,i} e^{-\lambda h_{n,i}}$ around $\lambda^{\circ}$ and use the facts that $\hat{\lambda}_n$ is the root of the function and that $\hat{\lambda}_n$ converges to $\lambda^{\circ}$ almost surely. The steps of the proof identical to that of Theorem \ref{CLT_root} (the bounds are indeed much simpler for the finite graph) and hence we omit the details.
\end{proof}

\begin{proof}[Proof of Theorem \ref{poisson_consistency_1}]



Let \( 0 < \epsilon < \lambda^{\circ}/2\). Since \( \lambda^{\circ} \) is a root of the equation  
\[
\mathbb{E}\left[(Z_{\mu} - h_0) e^{-\lambda(Z_{\mu} - h_0)}\right] = 0,
\]  
it follows that $\mathbb{E}[(Z_{\mu}-h_0)e^{-\lambda(Z_{\mu}-h_0)}] > 0$ (resp, $<0$) at $\lambda=\lambda^{\circ} - \epsilon$ (resp, $\lambda^{\circ} + \epsilon)$. Therefore, it suffices to show that for every fixed $\lambda>0$, we have 

\begin{equation}\mathbb{E}[(H_m-h_0)e^{-\lambda(H_m-h_0)}]=\mathbb{E}[(Z_{\mu}-h_0)e^{-\lambda(Z_{\mu}-h_0)}](1\pm o(1))
\label{triangle_poisson_asymp}
\end{equation}

Indeed if \eqref{triangle_poisson_asymp} is true, the two functions will be of the same sign for large enough $m$ implying that the LHS is positive at $\lambda^{\circ}-\epsilon$ and negative at $\lambda^{\circ}+\epsilon$ . This will imply that $|\lambda_m-\lambda^{\circ}|<\epsilon$.
To show \eqref{triangle_poisson_asymp}, first note that for $\lambda>0$, the function $x\mapsto (x-h_0)e^{-\lambda(x-h_0)}$ is bounded. Hence, by Poisson convergence of \(H\) counts in the regime $mp^{k(H)}\rightarrow c>0$, we have that
$$\mathbb{E}[(H_m-h_0)e^{-\lambda(H_m-h_0)}]=\mathbb{E}[(Z_{\mu}-h_0)e^{-\lambda(Z_{\mu}-h_0)}]\pm o(1)$$
Noting that in this regime, we also have $\mathbb{E}[(Z_{\mu}-h_0)e^{-\lambda(Z_{\mu}-h_0)}]=\Theta(1)$ for fixed $\lambda$, \eqref{triangle_poisson_asymp} holds.

\end{proof}

\begin{proof}[Proof of Theorem \ref{poisson_consistency_2}]
It is  straightforward to see that, $\forall \lambda > 0$, $\mathbb{E} \left[ (h_{n,1} e^{-\lambda h_{n,1}})^4 \right]$ is uniformly bounded in $n$. With this property, we get a strong law of large numbers for the triangular array \(\{h_{n,i}\}\):

\[
\frac{\sum_{i=1}^n h_{n,i} e^{-\lambda h_{n,i}}}{n} - \mathbb{E}\left[h_{n,1} e^{-\lambda h_{n,1}}\right] \xrightarrow{\text{a.s.}} 0.
\]
Recall that $\mu=c^{v(H)}/\lvert Aut(H)\rvert$. By repeating the same calculation as in Theorem \ref{poisson_consistency_1} with \(m\) replaced by \(m_n\) everywhere, 
\[
\mathbb{E}\left[h_{n,1} e^{-\lambda h_{n,1}}\right] - \mathbb{E}\left[(Z_{\mu}-h_0)e^{-\lambda(Z_{\mu}-h_0)}\right]\to 0
\]

\[
\implies \frac{\sum_{i=1}^n h_{n,i} e^{-\lambda h_{n,i}}}{n} - \mathbb{E}\left[(Z_{\mu}-h_0)e^{-\lambda(Z_{\mu}-h_0)}\right]\to 0.
\]
Now if $K \subset [0, \infty)$ is compact, then $\{a \mapsto a e^{-\lambda a} : \lambda \in K\}$ has finite bracketing number with brackets of the same form. Hence
\[
\sup_{\lambda \in K} \left\lvert\frac{1}{n} \sum_{i=1}^n h_{n,i} e^{-\lambda h_{n,i}} - \mathbb{E}\left[(Z_{\mu}-h_0)e^{-\lambda(Z_{\mu}-h_0)}\right]\right\rvert\to 0 \quad \text{a.s.}
\]
Since $[0, \infty)$ is countably compact we have
\[
\frac{1}{n} \sum_{i=1}^n h_{n,i} e^{-\lambda h_{n,i}} \to \mathbb{E}\left[(Z_{\mu}-h_0)e^{-\lambda(Z_{\mu}-h_0)}\right] \quad \text{uniformly on compacts almost surely.}
\]
\[
\implies \hat{\lambda}_n \to \lambda^{\circ} \quad \text{almost surely, and hence in probability.}
\]

\end{proof}

\begin{proof}[Proof of Lemma \ref{CLT_1}]
    Pick a \(\delta>0\) (any \(\delta\) will do, but one can choose \(\delta = 2\) for concreteness). We verify the Lindeberg condition: it suffices to show that
    \[\sum_{i=1}^n\mathbb{E}\left[\left(\frac{h_{n,i} e^{-\lambda h_{n,i}} - \mathbb{E}[h_{n,i} e^{-\lambda h_{n,i}}]}{\sqrt{n Var(h_{n,i} e^{-\lambda h_{n,i}})}}\right)^{2+\delta}\right] = \frac{1}{n^{\frac{\delta}{2}}}\times \mathbb{E}\left[\left(\frac{h_{n,1} e^{-\lambda h_{n,1}} - \mathbb{E}[h_{n,1} e^{-\lambda h_{n,1}}]}{\sqrt{Var(h_{n,1} e^{-\lambda h_{n,1}})}}\right)^{2+\delta}\right]\rightarrow 0.\]

    Indeed, for every \(\lambda>0\), the function \(x\mapsto xe^{-\lambda x}\) is bounded on the interval \([-h_0,\infty)\). And hence the expectation converges, by the Poisson convergence of \(h_{n,1}\) in total variation norm. Hence, the right hand side goes to \(0\).
\end{proof}

\begin{proof}[Proof of Theorem \ref{CLT_root}]
   We will modify the standard proof of the asymptotic normality of \(\mathcal{Z}\)-estimators to suit our needs. To that end, define
\[
\Psi_n(\lambda) = \frac{1}{n}\sum_{i=1}^n h_{n,i} e^{-\lambda h_{n,i}}  , 
\quad \Psi(\lambda) = \mathbb{E}\left[(Z_{\mu}-h_0)e^{-\lambda(Z_{\mu}-h_0)}\right]
\]
Now noting that \(\hat{\lambda}_n\) is a zero of \(\Psi_n(.)\) and converges in probability to \(\lambda^{\circ} = \log(\frac{\mu}{h_0})\), we expand \(\Psi_n(\hat{\lambda}_n)\) in a Taylor series about \(\lambda^{\circ}\). Then
\[
0 = \Psi_n(\hat{\lambda}_n) = \Psi_n(\lambda^{\circ}) + (\hat{\lambda}_n - \lambda^{\circ}) \dot{\Psi}_n(\lambda^{\circ}) 
+ \frac{1}{2} (\hat{\lambda}_n - \lambda^{\circ})^2 \ddot{\Psi}_n(\tilde{\lambda}_n),
\]
where $\tilde{\lambda}_n$ is a point between $\hat{\lambda}_n$ and $\lambda_0$. This can be rewritten as
\begin{equation}
\sqrt{n} (\hat{\lambda}_n - \lambda^{\circ}) = 
\frac{- \sqrt{n} (\Psi_n(\lambda^{\circ}) - \mathbb{E}[\Psi_n(\lambda^{\circ})])}{\dot{\Psi}_n(\lambda^{\circ}) + \frac{1}{2} (\hat{\lambda}_n - \lambda^{\circ}) \ddot{\Psi}_n(\tilde{\lambda}_n)} + \frac{-\sqrt{n}\mathbb{E}[\Psi_n(\lambda^{\circ})]}{\dot{\Psi}_n(\lambda^{\circ}) + \frac{1}{2} (\hat{\lambda}_n - \lambda^{\circ}) \ddot{\Psi}_n(\tilde{\lambda}_n)}.
\label{Z_Taylor}
\end{equation}

By arguments similar to those in Theorem \ref{poisson_consistency_2} (in particular, the property that for every positive \(\lambda\), the fourth moment is bounded uniformly in \(n\)), it follows that the following hold almost surely,
\[\dot{\Psi}_n(\lambda^{\circ}) \to \mathbb{E}\left[-(Z_{\mu}-h_0)^2e^{-\lambda^{\circ}(Z_{\mu}-h_0)}\right]\]
\[\ddot{\Psi}_n(\lambda^{\circ}) \to \mathbb{E}\left[(Z_{\mu}-h_0)^3e^{-\lambda^{\circ}(Z_{\mu}-h_0)}\right]\]

Then the consistency of \(\hat{\lambda}_n\) ensures that
the denominators in \eqref{Z_Taylor} converge almost surely to \\ \(\mathbb{E}\left[-(Z_{\mu}-h_0)^2e^{-\lambda^{\circ}(Z_{\mu}-h_0)}\right]\) . For the numerators, observe that by arguing as we did for expectation, the variance of \(h_{n,1}e^{-\lambda^{\circ}h_{n,1}}\) converges to that of \(-(Z_{\mu}-h_0)^2e^{-\lambda(Z_{\mu}-h_0)}\). Hence, the first numerator converges, by Lemma \ref{CLT_1}, to 
\[\mathcal{N}\left(0,\mathbf{Var}\left[(Z_{\mu}-h_0)e^{-\lambda(Z_{\mu}-h_0)}\right]\right)\]
For the second numerator, observe that 
\[\mathbb{E}[\Psi_n(\lambda^{\circ})] = \mathbb{E}[(Z_{\mu}-h_0)e^{-\lambda^{\circ}(Z_{\mu}-h_0)}] + O(C(\lambda^{\circ})m_n^{v(H)-2}p^{e(H)-1})\] with \(p_n = (\frac{c}{m_n})^{1/k(H)}\) and some constant $C(\lambda^{\circ})$. This follows by the standard bound on the total variation distance of the \(H\) counts from a Poisson random variable with the same mean. Now the first term is \(0\) by definition and the error is \(o(\frac{c(\lambda^{\circ})}{\sqrt{n}}\)) (for some constant $c(\lambda^{\circ})$) by assumption. Hence the second numerator goes to \(0\) in the limit. \\
Combining the above using Slutsky's theorem, we get
\[
\sqrt{n}(\hat{\lambda}_n - \lambda^{\circ}) \xrightarrow{d} \mathcal{N}\left(0, \frac{\mathbf{Var}\left[(Z_{\mu}-h_0)e^{-\lambda^{\circ}(Z_{\mu}-h_0)}\right]}{\mathbb{E}\left[(Z_{\mu}-h_0)^2e^{-\lambda^{\circ}(Z_{\mu}-h_0)}\right]^2}\right)
\]

\end{proof}

\begin{proof}[Proof of Theorem \ref{two_sample_asymptotics}]
The proof hinges on Theorem \ref{CLT_root}. However since the sample sizes are different we cannot apply the theorem directly. Instead we apply the following lemma whose proof is given in the appendix.

\begin{lemma}{\label{two_sample_asymptotics_lemma}}
Let $\hat{\theta}_1$ and $\hat{\theta}_2$ be estimators of real parameters $\theta_1$ and $\theta_2$ based on two \emph{independent} samples of sizes $n_1$ and $n_2$, respectively. Suppose there exist scaling sequences $a_{n_1} \to \infty$, $b_{n_2} \to \infty$ and finite positive constants $\sigma_1^2, \sigma_2^2$ such that
\[
a_{n_1}(\hat{\theta}_1 - \theta_1) \xrightarrow{d} N(0, \sigma_1^2), 
\qquad 
b_{n_2}(\hat{\theta}_2 - \theta_2) \xrightarrow{d} N(0, \sigma_2^2).
\]
Assume that the rate ratio satisfies $r_n := a_{n_1}/b_{n_2} \to \rho \in (0, \infty)$ and define the following statistic


\[
T_n := \frac{(\hat{\theta}_1 - \hat{\theta}_2)-(\theta_1-\theta_2)}{\sqrt{V_n^\ast}},
\qquad 
V_n^\ast := \frac{\sigma_1^2}{a_{n_1}^2} + \frac{\sigma_2^2}{b_{n_2}^2}.
\]
Then it follows that 
\[
T_n \xrightarrow{d} N(0,1).
\]
\end{lemma}
The proof of the theorem now follows by setting $a_{n_1}=\sqrt{n_1}$, $b_{n_2}=\sqrt{n_2}$ and observing that
\[
\sqrt{n_i}(\hat{\lambda}_{n_i} - \lambda_i^{\circ}) \xrightarrow{d} \mathcal{N}\left(0, \sigma^2_i\right)
\]
with $$\sigma^2_i=\frac{\mathbf{Var}\left[(Z_{\mu_i}-h_0)e^{-\lambda_i^{\circ}(Z_{\mu_i}-h_0)}\right]}{\mathbb{E}\left[(Z_{\mu_i}-h_0)^2e^{-\lambda_i^{\circ}(Z_{\mu_i}-h_0)}\right]^2}$$
for $i \in \{1,2\}$.

\end{proof}
\begin{proof}[Proof of Theorem \ref{normal_consistency_1}]
The core of the argument relies on the large deviation principle for general ERGMs developed in \cite{chatterjee2013estimating}. The expectation $\E[e^{-\frac{\lambda}{m^{v(H)-2}}(H_m-h_0)}]$ can be viewed as the partition function of a new ERGM whose Hamiltonian is the sum of the original ERGM Hamiltonian and a perturbation term related to the $H$-count.

\begin{align*}
\E\left[e^{-\frac{\lambda}{m^{v(H)-2}}(H_m-h_0)}\right] &= 
\frac{1}{Z_m(\beta)}\sum e^{-\frac{\lambda}{m^{v(H)-2}}(H_m(G)-h_0) + m^2\sum \beta_kt(T_k,G)}\\
&= e^{\frac{\lambda h_0}{m^{v(H)-2}}}\times\frac{\sum e^{-\frac{\lambda}{m^{v(H)-2}}H_m(G) + m^2\sum \beta_kt(T_k,G)}}{\sum e^{m^2\sum \beta_kt(T_k,G)}}\\
&= e^{\frac{\lambda h_0}{m^{v(H)-2}}}\times\frac{\sum \exp\{m^2(-\lambda t(H,G)/\lvert Aut(H)\rvert + O(\lambda/m) + \sum \beta_kt(T_k,G))\}}{\sum \exp\{m^2\sum \beta_kt(T_k,G)\}}
\end{align*}

From Theorem 3.1 of \cite{chatterjee2013estimating}, the normalized log-partition function of an ERGM converges to the supremum of a free energy functional over the space of graphons. Applying this result to our perturbed model, we get that
\begin{align}
&\frac{1}{m^2}\ln \left(\E_{\bbeta}\left[e^{-\frac{\lambda}{m^{v(H)-2}}(H_m-h_0)}\right]\right)\nonumber\\
\to & \frac{\lambda h_0}{m^{v(H)}} + \sup_{\tilde{h}\in\tilde{W}}\left(-\frac{\lambda}{|Aut(H)|} t(H,h) + \sum_{k=1}^K \beta_k t(T_k, h) - \frac{1}{2}I(h) \right) - \sup_{\tilde{h}\in\tilde{W}}\left(\sum_{k=1}^K \beta_k t(T_k, h) - \frac{1}{2}I(h) \right)\nonumber\\
=& \mathfrak{g}(\lambda)\label{var_formula}
\end{align}
for all $\lambda \in \mathbb{R}$.
Since $x\mapsto \ln(x)$ is strictly increasing, we need to show the convergence of critical points in the display above. The functions in the sequence above are all strictly convex in $\lambda$ by direct differentiation. The limiting function is convex being a sum of an affine part and a supremum over convex functions. Now the convergence of critical points will follow once we show that the limiting function is strictly convex. This is equivalent to showing that the function
\[
G(\lambda) := \sup_{\tilde{h}\in\tilde{W}}\left(-\frac{\lambda}{|Aut(H)|} t(H,h) + \sum_{k=1}^K \beta_k t(T_k, h) - \frac{1}{2}I(h) \right)
\]
is strictly convex in $\lambda$. Suppose, for the sake of contradiction, that $G(\lambda)$ is not strictly convex. Then there exist $\lambda_1 < \lambda_2$ and $\mu \in (0,1)$ such that
\[
\mu G(\lambda_1) + (1-\mu) G(\lambda_2) = G(\mu \lambda_1 + (1-\mu)\lambda_2).
\]
This implies that there must exist a single graphon $\tilde{w}^* \in \tilde{W}$ that simultaneously achieves the supremum for both \(\lambda_1\) and \(\lambda_2\).

Let $w^*$ be any representative of the class $\tilde{w}^*$. According to the Euler-Lagrange equations derived in Theorem 6.3 of \cite{chatterjee2013estimating}, any maximizing graphon for the ERGM with Hamiltonian defined by statistics $H, T_1, \dots, T_K$ and parameters $(-\frac{\lambda}{|Aut(H)|}, \beta_1, \dots, \beta_K)$ must satisfy:
\[
w^*(x,y) = \frac{\exp\left(2 \left(-\frac{\lambda}{|Aut(H)|}\Delta_H w^*(x,y) + \sum_{k=1}^K \beta_k \Delta_{T_k} w^*(x,y)\right)\right)}{1 + \exp\left(2 \left(-\frac{\lambda}{|Aut(H)|}\Delta_H w^*(x,y) + \sum_{k=1}^K \beta_k \Delta_{T_k} w^*(x,y)\right)\right)},
\]
where $\Delta_F w(x,y)$ is the operator giving the change in density of a graph $F$ when an edge is added at $(x,y)$:
$$
\Delta_F w(x,y) := \sum_{(r,s)\in E(F)} \int_{[0,1]^{|V(F)\setminus\{r,s\}|}} \prod_{\substack{(r',s')\in E(F) \\ (r',s')\neq(r,s)}} w(z_{r'}, z_{s'}) \prod_{\substack{v\in V(F) \\ v \neq r,s}} dz_v
$$
This equation must hold for our fixed $w^*$ for all $\lambda \in [\lambda_1, \lambda_2]$. However, since $H$ is a non-empty graph, the term $\Delta_H w^*(x,y)$ is not identically zero (as established in \cite{chatterjee2013estimating}). This means we can rearrange the equation to solve for $\lambda$:
\[
\lambda = \frac{|Aut(H)|}{\Delta_H w^*(x,y)} \left( \sum_{k=1}^K \beta_k \Delta_{T_k} w^*(x,y) - \frac{1}{2}\text{logit}(w^*(x,y)) \right).
\]
Since the right-hand side is a fixed value determined by $w^*$ and the $\beta_k$, $\lambda$ must be uniquely determined. This contradicts the assumption that the same $w^*$ is a maximizer for two distinct $\lambda$ values. Therefore, the function $G(\lambda)$ must be strictly convex.

\end{proof}

\begin{proof}[Sketch of proof of Theorem \ref{normal_consistency_2}]
 Our goal is to show the convergence in probability of the empirical critical point $\hat{\lambda}_n$ to the theoretical one $\lambda^{\circ}$ in the dense ERGM regime.

The proof hinges on the uniform convergence of the normalized log-partition function:
\[
\frac{1}{m_n^2}\log\left(\frac{1}{n}\sum_{i=1}^n e^{-\frac{\lambda}{m_n^{v(H)-2}} h_{n,i}}\right)
\]
to its limit, which, from Theorem \ref{normal_consistency_1}, is:
\[
\frac{\lambda h_0}{m_n^{v(H)}} + \sup_{\tilde{h}\in\tilde{W}}\left(-\frac{\lambda}{|Aut(H)|} t(H,h) + \sum_{k=1}^K \beta_k t(T_k, h) - \frac{1}{2}I(h) \right) - \sup_{\tilde{h}\in\tilde{W}}\left(\sum_{k=1}^K \beta_k t(T_k, w) - \frac{1}{2}I(h) \right).
\]

The core idea is to apply the log-sum-exp approximation framework from \cite{chatterjee2016nonlinear}. We rewrite the empirical objective function as a log-partition function over the space of all graphs on $m_n$ vertices:
\[
\frac{1}{m_n^2}\log\left(\sum_{G} \frac{N(G)}{n} e^{-\frac{\lambda}{m_n^{v(H)-2}} (H(G)-h_0)}\right) = \frac{\lambda h_0}{m_n^{v(H)}} + \frac{1}{m_n^2} \log \left(\sum_{x \in \{0,1\}^{\binom{m_n}{2}}} e^{f(x)}\right),
\]
where $x$ is the adjacency vector of a graph, $N(G)$ is its count in the sample, and the energy function is $f(x) = \log\frac{N(x)}{n} - \frac{\lambda}{m_n^{v(H)-2}} H(x)$. We decompose this energy function into two parts: 
$$f(x) = f_H(x) + \log\frac{N(x)}{n},\quad \quad f_H(x) = -\frac{\lambda}{m_n^{v(H)-2}}H(x)$$.

\begin{enumerate}
    \item \textbf{H-count Part Control:} The term $f_H(x) = -\frac{\lambda}{m_n^{v(H)-2}}H(x)$ is the standard $H$-count interaction from our maximum entropy problem. As shown in \cite{chatterjee2016nonlinear}, the function and its derivatives are well-behaved after appropriate scaling:
    \[
    \left\|f_H\right\|_\infty = O(m_n^2|\lambda|)\quad, \quad
    \left\|\frac{\partial f_H}{\partial x_e}\right\|_\infty = O(|\lambda|) \quad \text{and} \quad \left\|\frac{\partial^2 f_H}{\partial x_e\partial x_{e'}}\right\|_\infty = O\left(\frac{|\lambda|}{m_n}\right).
    \]

    \item \textbf{ERGM Part Control:} The term $\log\frac{N(x)}{n}$ depends on the empirical sample drawn from the ERGM $\P_{\bbeta}$. The key insight is that for a sufficiently large number of samples $n$ (specifically, when $\binom{m_n}{2} = o(\log n)$), the empirical distribution of graphs will be very close to the true probability distribution $\P_{\bbeta}$. Therefore, with high probability, $\log\frac{N(x)}{n}$ behaves like the log-likelihood of the ERGM itself:
    \[
    \log\frac{N(x)}{n} \approx m_n^2\sum_{k=1}^K \beta_k t(T_k,x) - \log Z_{m_n}(\bbeta),
    \]
    where $t(T_k,x)$ is the homomorphism density of the $k$-th sufficient statistic. The term $\log Z_{m_n}(\bbeta)$ is a constant with respect to $x$ and does not affect the derivatives. The Hamiltonian part, $m_n^2\sum \beta_k t(T_k,x)$, and its derivatives are well-controlled, so that we have:
    \[
    \left\|f_{ERGM}\right\|_\infty = O(m_n^2) \quad, \quad \left\|\frac{\partial f_{ERGM}}{\partial x_e}\right\|_\infty = O(1) \quad \text{and} \quad \left\|\frac{\partial^2 f_{ERGM}}{\partial x_e\partial x_{e'}}\right\|_\infty = O\left(\frac{1}{m_n}\right).
    \]
    where $f_{ERGM} = m_n^2\sum_{k=1}^K \beta_k t(T_k,x) - \log Z_{m_n}(\bbeta)$
    
    \item \textbf{Application of General Result:} The total energy function is $f(x) = f_H(x) + \log\frac{N(x)}{n}$. Combining the bounds, the overall energy function satisfies $\|f\|_\infty = O_P(m_n^2)$, $\|\partial f / \partial x_e\|_\infty = O_P(1)$ and $\|\partial^2 f / (\partial x_e \partial x_{e'})\|_\infty = O_P(1/m_n)$. These are precisely the conditions required by the general log-sum-exp approximation theorem (Theorem 1.6 in \cite{chatterjee2016nonlinear}). The theorem shows that the error terms in the approximation are negligible after normalization by $m_n^2$. This establishes the uniform convergence of the empirical objective function to the desired variational formula $\mathfrak{g}(\lambda)$. Since $\mathfrak{g}(\lambda)$ was shown to be strictly convex in the proof of Theorem \ref{normal_consistency_1}, the convergence of the empirical critical point $\hat{\lambda}_n$ to the unique critical point $\lambda^{\circ}$ of $\mathfrak{g}(\lambda)$ is guaranteed.
\end{enumerate}
\end{proof}

\begin{proof}[Proof of Theorem \ref{CLT_root_2}]
We begin by performing a Taylor expansion of the empirical critical point equation about $\lambda^{\circ}$:
\begin{align*}
&0 = \sum\limits_{i=1}^{n} h_{n,i} e^{-\frac{\lambda^{\circ}}{m_n^{v(H)-2}}h_{n,i}} + (\hat{\lambda}_n - \lambda^{\circ})\sum\limits_{i=1}^{n} \frac{-h_{n,i}^2}{m_n^{v(H)-2}}e^{-\frac{\tilde{\lambda}_n}{m_n^{v(H)-2}}h_{n,i}} \\
&\implies m_n^2(\hat{\lambda}_n - \lambda^{\circ}) = \frac{\sum\limits_{i=1}^{n}\frac{h_{n,i}}{m_n^{v(H)}}e^{-\frac{\lambda^{\circ}}{m_n^{v(H)-2}}h_{n,i}}}{\sum\limits_{i=1}^{n}\frac{h_{n,i}^2}{m_n^{2v(H)}}e^{-\frac{\tilde{\lambda}_n}{m_n^{v(H)-2}}h_{n,i}} }
\end{align*}

For $\lambda^{\circ}=0$, we get
\begin{equation}
m_n^2 \,\hat{\lambda}_n
=
\frac{\frac{1}{n}\displaystyle \sum_{i=1}^{n} \frac{h_{n,i}}{m_n^{v(H)}}}
     {\frac{1}{n}\displaystyle \sum\limits_{i=1}^{n}\frac{h_{n,i}^2}{m_n^{2v(H)}}e^{-\frac{\tilde{\lambda}_n}{m_n^{v(H)-2}}h_{n,i}} }=\frac{\frac{1}{\sqrt{n}}\displaystyle \sum_{i=1}^{n} \frac{h_{n,i}}{m_n^{v(H)-1}}}{\frac1n\displaystyle \sum\limits_{i=1}^{n}\frac{h_{n,i}^2}{m_n^{2v(H)}}e^{-\frac{\tilde{\lambda}_n}{m_n^{v(H)-2}}h_{n,i}}}\frac{1}{\sqrt{n}{m_n}}.\label{lambda_0_asymp}
\end{equation}

Define

\[
    Z_n := \frac{1}{\sqrt{n}} \sum_{i=1}^{n} \frac{h_{n,i}}{m_n^{v(H)-1}} .
\]

We show that
\begin{equation}
    \frac{m_n}{\sqrt[4]{n}} Z_n \xrightarrow{P} 0 .
    \label{num_conv_1}
\end{equation}
It suffices to show that the summands 
\(\tfrac{h_{n,i}}{m_n^{v(H)-1}}\) are subgaussian with mean zero and a variance proxy bounded by a constant independent of \(n\).  
Equivalently, there exists \(\sigma^2>0\) such that for all \(s\in\mathbb{R}\),

\[
    \mathbb{E}\!\left[e^{s h_{n,i}/m_n^{v(H)-1}}\right] \leq e^{\sigma^2 s^2/2},
\]
uniformly over \(i,n\).  
This ensures that \(Z_n\) itself is subgaussian with a variance proxy not depending on \(n\), hence tight.  
Consequently,

\[
    \frac{m_n}{\sqrt[4]{n}} Z_n \xrightarrow{P} 0.
\]

    To this end, we apply \textbf{Theorem 1} from \cite{ganguly2024sub} in the subcritical regime of the ferromagnetic ERGM. Recall from earlier that the partial derivatives of \(\frac{H(x)}{m_n^{v(H)-2}}\) are bounded in the \(l^{\infty}\) norm by a constant. Hence, the lipschitz vector of \(H(x)\) is bounded in \(l^{\infty}\) norm by \(m_n^{v(H)-2}\). By \(\textbf{Theorem 1}\) of \cite{ganguly2024sub}, we get the bound
    \[
    \mathbb{P}(\left\lvert H(x) - \mathbb{E}[H(x)]\right\rvert\geq t)\leq 2e^{-\frac{ct^2}{{m_n\choose 2}m_n^{v(H)-2}\cdot m_n^{v(H)-2}}}
    \]
    for all \(t\in\mathbb{R}\). Replacing \(t\) by \(m_n^{v(H)-1}t\), we get
    \(\mathbb{P}\left(\left\lvert\frac{H(x)}{m_n^{v(H)-1}}-\mathbb{E}\left[\frac{H(x)}{m_n^{v(H)-1}}\right]\right\rvert\geq t\right)\leq 2e^{-ct^2}\), which shows the constant variance proxy.\\
    
 In the next step we show that 
 \begin{equation}V_n:= \sqrt[4]{n}\times m_n\times\frac{1}{n} \sum_{i=1}^{n} \frac{h_{n,i}^2}{m_n^{2v(H)}}e^{-\frac{\tilde{\lambda}_n}{m_n^{v(H)-2}}h_{n,i}} = \Omega_P\left(\frac{1}{m_n}\right).\label{denom_bound}\end{equation}
We will need the empirical mean and second moment.
\[
\overline{h_n}:=\frac{1}{n}\sum_{i=1}^n\frac{h_{n,i}}{m_n^{v(H)}},\qquad
\overline{h_n^2}:=\frac{1}{n}\sum_{i=1}^n\frac{h_{n,i}^2}{m_n^{2v(H)}}.
\]
We will also use the empirical variance
\[
\widehat{\mathrm{Var}}(h_n):=\overline{h_n^2}-\overline{h_n}^2.
\]
For any \(\delta>0\), we have the following chain of inequalities:

\begin{align*}
  &\frac{1}{n}\sum_i \frac{h_{n,i}^2}{m_n^{2v(H)}} e^{-\frac{\tilde{\lambda}_n}{m_n^{v(H)-2}} h_{n,i}}\\
  \ge&
     e^{-\frac{\tilde{\lambda}_n}{m_n^{v(H)-2}}\left(\overline{h_n} - t\sqrt{\overline{h_n^2} - \overline{h_n}^2}\right)} \delta^2 \frac{\# \left\{ \left|\frac{h_{n,i}}{m_n^{v(H)}}\right| \ge \delta, h_{n,i} \ge \overline{h_n} - t \sqrt{\overline{h_n^2} - \overline{h_n}^2} \right\}}{n} \\
    \ge& e^{-\frac{\tilde{\lambda}_n}{m_n^{v(H)-2}}\left(\overline{h_n} - t\sqrt{\overline{h_n^2} - \overline{h_n}^2}\right)} \delta^2 \left(1 - \left( \frac{\# \left\{ \left|\frac{h_{n,i}}{m_n^{v(H)}}\right| < \delta \right\}}{n} + \frac{\# \left\{h_{n,i} < \overline{h_n} - t \sqrt{\overline{h_n^2} - \overline{h_n}^2}\right\}}{n}\right) \right) \\
    \ge & e^{-\frac{\tilde{\lambda}_n}{m_n^{v(H)-2}}\left(\overline{h_n} - t\sqrt{\overline{h_n^2} - \overline{h_n}^2}\right)} \delta^2 \left(1 - \left( \mathbb{P}\left(\left|\frac{h_{n,i}}{m_n^{v(H)}}\right|<\delta\right) + o_p(1) + 1/t^2 \right) \right) \\
    = & e^{-\frac{\tilde{\lambda}_n}{m_n^{v(H)-2}}\left(\overline{h_n} - t\sqrt{\overline{h_n^2} - \overline{h_n}^2}\right)} \delta^2 \left( \mathbb{P}\left(\left|\frac{h_{n,i}}{m_n^{v(H)}}\right| \ge \delta\right) - 1/t^2 - o_p(1) \right)
\end{align*}
where the first inequality follows since $\tilde{\lambda}_n<0$. Now from the central limit theorem for subgraph counts in the subcritical regime of a ferromagnetic ERGM (\textbf{Corollary 1.2} of \cite{winstein2025quantitative}), we get that 
$\mathbb{P}\left(\left|\frac{h_{n,i}}{m_n^{v(H)}}\right|\geq \frac{1}{m_n}\right)\rightarrow \mathbb{P}\left(\left|\mathcal{N}(0,c))\right|\geq 1\right) $
The limit is positive. Denote it by \(l\). Choose \(t>\frac{1}{\sqrt{l}}\).
Then for \(\delta=\frac{1}{m_n}\), the RHS is lower bounded, in probability, by
$$e^{-\frac{\tilde{\lambda}_n}{m_n^{v(H)-2}}\left(\overline{h_n} - t\sqrt{\overline{h_n^2} - \overline{h_n}^2}\right)} \frac{1}{m_n^2} \left( l - 1/t^2 - o_p(1) \right)$$

Now each $h_{n,i}$ has moments at most polynomial in $m_n$ (by the same subgaussian tail argument).  Using $\mathbb{E}[h_{n,i}]=0$ and $\mathbb{E}[h_{n,i}^2] = O(m_n^{2v(H)-2})$  we get that $\overline{h_n}$ converges to $0$ in probability for $n\gg m_n^{2v(H)-2}$. So,we have \(\overline{h_n}-t\sqrt{\overline{h_n^2}-\overline{h_n}^2} = O_p\left(t\sqrt{m_n^{2v(H)-2}}\right) = O_p(m_n^{v(H)-1})\) for fixed \(t\). Hence, the RHS is lower bounded by $e^{-O(\lvert\tilde{\lambda}_n\rvert m_n)}\frac{l-1/t^2}{m_n^2}$ in probability. Now,
\[\sqrt[4]{n}\times m_n\times  e^{-O(\lvert\tilde{\lambda}_n\rvert m_n)}\frac{l-1/t^2}{m_n^2} = \Omega_P\left(\frac{1}{m_n}\right)\]
where we have used the facts \(\tilde{\lambda}_n\xrightarrow{P}0\) and \(n\gg e^{{m_n\choose 2}}\).

Finally, we plug-in \eqref{num_conv_1} and \eqref{denom_bound} in \eqref{lambda_0_asymp} to conclude
that $m_n^2\hat{\lambda}_n\xrightarrow{P} 0$.

\end{proof}

\begin{proof}[Sketch of proof of Theorem \ref{sharp_rates}]
Before proving Theorem \ref{sharp_rates}, we give a brief outline of the proof below. This theorem sharpens the consistency result of Theorem \ref{normal_consistency_2} by establishing the precise asymptotic behavior of $m_n^2(\hat{\lambda}_n-\lambda^{\circ})$.

We begin by performing a Taylor expansion of the empirical critical point equation about $\lambda^{\circ}$:
\begin{align*}
&0 = \sum\limits_{i=1}^{n} h_{n,i} e^{-\frac{\lambda^{\circ}}{m_n^{v(H)-2}}h_{n,i}} + (\hat{\lambda}_n - \lambda^{\circ})\sum\limits_{i=1}^{n} \frac{-h_{n,i}^2}{m_n^{v(H)-2}}e^{-\frac{\tilde{\lambda}_n}{m_n^{v(H)-2}}h_{n,i}}\\
&\implies m_n^2(\hat{\lambda}_n - \lambda^{\circ}) = \frac{\sum\limits_{i=1}^{n}\frac{h_{n,i}}{m_n^{v(H)}}e^{-\frac{\lambda^{\circ}}{m_n^{v(H)-2}}h_{n,i}}}{\sum\limits_{i=1}^{n}\frac{h_{n,i}^2}{m_n^{2v(H)}}e^{-\frac{\tilde{\lambda}_n}{m_n^{v(H)-2}}h_{n,i}} }.
\end{align*}

To analyze the asymptotic behavior of the numerator and denominator separately, we employ two distinct variational formulations based on the log-sum-exp approximation.

For the numerator, we analyze the derivative of the normalized log-partition function. By Danskin's Theorem \cite{danskin1966theory} and the uniform convergence established in the proof of Theorem \ref{normal_consistency_2}, we have:
\begin{equation}
\frac{\sum\limits_{i=1}^{n}\frac{-h_{n,i}}{m_n^{v(H)}}e^{-\frac{\lambda^{\circ}}{m_n^{v(H)-2}}h_{n,i}}}{\sum\limits_{i=1}^{n}e^{-\frac{\lambda^{\circ}}{m_n^{v(H)-2}}h_{n,i}}} \xrightarrow{P} \frac{d}{d\lambda} \mathfrak{g}(\lambda) \bigg|_{\lambda=\lambda^{\circ}} = \frac{p_0^{e(H)}}{|Aut(H)|} - \frac{(u^*)^{e(H)}}{|Aut(H)|}
\label{num_conv}\end{equation}
where $u^*$ is the unique constant graphon that maximizes the variational problem defining $\mathfrak{g}(\lambda^\circ)$ from \eqref{var_form_lambda}.

For the denominator, we introduce a crucial auxiliary function by adding a squared $H$-count term:
\begin{align*}
\frac{1}{m_n^2} \log \sum_{G} \exp\left(-\frac{\lambda}{m_n^{v(H)-2}} h(G) + \log \frac{N(G)}{n} + \frac{\alpha}{m_n^{2v(H)-2}} h(G)^2\right)
\end{align*}
where $h(G) = H(G) - h_0$. We analyze this function by:
\begin{enumerate}
    \item Reparameterizing the exponent to isolate the terms involving the statistics $H(G)$ and $H(G)^2$, analogous to the proof of Theorem \ref{normal_consistency_1}. The term $\log \frac{N(G)}{n}$ is now approximated by the ERGM Hamiltonian $\sum_k \beta_k T_k(G)$.

    \item Showing this function is differentiable at $\alpha=0$ by analyzing the behavior of the squared $H$ statistic $T_H^2(x) = (H(x))^2$. As this is a combinatorial property, the derivative bounds remain valid:
    \[
    \left|\frac{\partial}{\partial x_{ij}}\left(\frac{\alpha}{m_n^{2v(H)-2}}T_H^2(x)\right)\right|\le O(|\alpha|), \quad
    \left|\frac{\partial^2}{\partial x_{ij}\partial x_{kl}}\left(\frac{\alpha}{m_n^{2v(H)-2}}T_H^2(x)\right)\right|\le \frac{O(|\alpha|)}{m_n}.
    \]
    
    \item Establishing that the auxiliary function converges uniformly to its limiting variational form:
    \begin{align*}
     \sup_{\tilde{h}\in\tilde{W}}\bigg(&-\left(\frac{\lambda^{\circ}}{|Aut(H)|}+\frac{2\alpha p_0^{e(H)}}{|Aut(H)|^2}\right) t(H,h) + \sum_{k=1}^K \beta_k t(T_k, h) \\
     &+ \frac{\alpha}{|Aut(H)|^2} t(H,h)^2 - \frac{1}{2} I(h)\bigg) + \frac{\lambda^{\circ} p_0^{e(H)}}{\lvert Aut(H)\rvert} + \frac{\alpha p_0^{2e(H)}}{\lvert Aut(H)\rvert^2} \\&- \sup_{\tilde{h}\in\tilde{W}}\bigg( \sum_{k=1}^K \beta_k t(T_k, h) 
      - \frac{1}{2} I(h)\bigg).
    \end{align*}
\end{enumerate}

Using Danskin's theorem \cite{danskin1966theory}, we compute the derivative of this limiting function with respect to $\alpha$ at $\alpha=0$, which captures the limiting behavior of the denominator. Since the supremum in the limit is uniquely attained at the constant graphon $u^*$, we obtain:
\begin{equation}
\frac{\sum\limits_{i=1}^{n}\frac{h_{n,i}^2}{m_n^{2v(H)}} e^{-\frac{\lambda^{\circ}}{m_n^{v(H)-2}} h_{n,i}}}{\sum\limits_{i=1}^{n} e^{-\frac{\lambda^{\circ}}{m_n^{v(H)-2}} h_{n,i}}}\xrightarrow{P} \left(\frac{(u^*)^{e(H)}}{|Aut(H)|} - \frac{p_0^{e(H)}}{|Aut(H)|}\right)^2
\label{denom_conv}\end{equation}

Combining the limits of the numerator and denominator yields the stated convergence:
\begin{equation}
m_n^2(\hat{\lambda}_n - \lambda^{\circ}) \xrightarrow{P} \frac{\frac{(u^*)^{e(H)}}{|Aut(H)|} - \frac{p_0^{e(H)}}{|Aut(H)|}}{\left(\frac{(u^*)^{e(H)}}{|Aut(H)|} - \frac{p_0^{e(H)}}{|Aut(H)|}\right)^2} = \frac{1}{\frac{(u^*)^{e(H)}}{|Aut(H)|} - \frac{p_0^{e(H)}}{|Aut(H)|}}.\label{sharp_rates_statement}
\end{equation}

\end{proof}

\begin{proof}[Proof of Theorem \ref{dense_two_sample}]
    Take \(p_0\) to be \((1-\epsilon)\), so that \(p, \tilde{p} < p_0\).
We apply Theorem \ref{sharp_rates} to the pair \((p,p_0)\) to
obtain
\begin{equation}
\label{limit_1}
m_n^2(\hat{\lambda}_n-\lambda^{\circ}(p))\xrightarrow{P} \frac{1}{\frac{(u^*(p))^{e(H)}}{\lvert Aut(H)\rvert}-\frac{p_0^{e(H)}}{\lvert Aut(H)\rvert}}
\end{equation}

and to the pair, \((\tilde{p},p_0)\) to obtain
\begin{equation}
\label{limit_2}
m_n^2(\tilde{\lambda}_n-\lambda^{\circ}(\tilde{p}))\xrightarrow{P} \frac{1}{\frac{(u^*(\tilde{p}))^{e(H)}}{\lvert Aut(H)\rvert}-\frac{p_0^{e(H)}}{\lvert Aut(H)\rvert}}
\end{equation}
where \(\lambda^{\circ}(p) = \lim \hat{\lambda}_n\) and \(\lambda^{\circ}(\tilde{p})=\lim \tilde{\lambda}_n\). It can be shown from the computations of Theorem \ref{sharp_rates} that 
\begin{itemize}
\item \(p = \tilde{p} \implies \lambda^{\circ}(p) = \lambda^{\circ}(\tilde{p})\) and \(u^*(p) = u^*(\tilde{p})\) 
\item \(p < \tilde{p} \implies \lambda^{\circ}(p) < \lambda^{\circ}(\tilde{p})\) and \(u^*(p) < u^*(\tilde{p})\)
\item \(p > \tilde{p} \implies \lambda^{\circ}(p) > \lambda^{\circ}(\tilde{p})\) and \(u^*({p}) > u^*({\tilde{p}})\)
\end{itemize}

Hence subtracting \eqref{limit_2} from \eqref{limit_1} gives the result.

\end{proof}

\section{Proofs of Theorems \ref{normal_consistency_2} and \ref{sharp_rates} \label{sec:3.8,3.10}}

\begin{proof}[Proof of Theorem \ref{normal_consistency_2}]
For a graph \(G\) on \(m\) vertices (with \(m=m_n\)), let
\[
x=(x_{ij})_{1\le i<j\le m}\in\{0,1\}^{\binom{m}{2}}
\]
be its edge–indicator vector. We use $F(x)$ to denote the count of a generic subgraph $F$ in the graph represented by $x$. The quantity we wish to analyze is:
\[
\frac{1}{n}\sum_{i=1}^n e^{-\frac{\lambda}{m_n^{v(H)-2}} h_{n,i}}
\]
where $h_{n,i} = H(\mathcal{G}_{n,i}) - h_0$. This can be rewritten as a log-partition function over all graphs:
\[
\sum_{x\in\{0,1\}^{\binom{m}{2}}} \frac{N(x)}{n} \exp\left(-\frac{\lambda}{m_n^{v(H)-2}} (H(x)-h_0)\right) = \exp\left(\frac{\lambda h_0}{m_n^{v(H)-2}}\right) \sum_{x\in\{0,1\}^{\binom{m}{2}}}\exp\!\Bigl(f(x)\Bigr),
\]
with the energy function defined as
\begin{equation}
\label{eq:f_def_gen}
f(x) = -\frac{\lambda}{m^{v(H)-2}} H(x) + \log\Bigl(\frac{N(x)}{n}\Bigr).
\end{equation}

\subsubsection*{Combination of H-Count and ERGM Bounds}

Under the assumption that the samples are drawn from an ERGM with parameters $\bbeta$, and provided that $\binom{m_n}{2} = o(\log n)$, the term $\log(N(x)/n)$ behaves, with high probability, like the log-likelihood of the ERGM. That is,
\[
\log\left(\frac{N(x)}{n}\right) \approx m_n^2\sum_{k=1}^K \beta_k t(T_k,x) - \log Z_m(\bbeta).
\]
In particular we will show that not only the functions are approximated by each other, but the same holds for the first and second derivatives. 
The term $\log Z_m(\bbeta)$ is a constant and can be ignored for derivative calculations. Let us denote the right hand side of the above display by $f_{ERGM}$. The function $f(x)$ in \eqref{eq:f_def_gen} thus decomposes as
\[
f(x) = f_H(x) + \log\left(\frac{N(x)}{n}\right),
\]
with
\(
f_H(x)=-\frac{\lambda}{m^{v(H)-2}}H(x) 
\). For the \(H\)-count part \(f_H(x)\), standard combinatorial arguments from \cite{chatterjee2016nonlinear} (Lemma 5.1) provide the bounds after appropriate scaling:
\[
\left\|f_H\right\|_\infty \leq m_n^2 \lvert \lambda\rvert \quad, \quad \left\|\frac{\partial f_H}{\partial x_e}\right\|_\infty \le C_H|\lambda|, \quad \text{and} \quad \left\|\frac{\partial^2 f_H}{\partial x_e\partial x_{e'}}\right\|_\infty \le  \frac{C_H'|\lambda|}{m_n}.
\]
The ERGM part, $f_{ERGM}(x)$, is a linear combination of subgraph counts. It is therefore also bounded similarly:
\[
\left\|f_{ERGM}\right\|_\infty \le m_n^2\sum_{k=1}^K |\beta_k| \left\|t(T_k,x)\right\|_\infty \le m_n^2\sum_{k=1}^K |\beta_k| = m_n^2\|\bbeta\|_1,
\]
\[
\left\|\frac{\partial f_{ERGM}}{\partial x_e}\right\|_\infty \le m_n^2\sum_{k=1}^K |\beta_k| \left\|\frac{\partial t(T_k,x)}{\partial x_e}\right\|_\infty \le \sum_{k=1}^K |\beta_k| C_k = O(\|\bbeta\|_1),
\]
\[
\left\|\frac{\partial^2 f_{ERGM}}{\partial x_e \partial x_{e'}}\right\|_\infty \le m_n^2\sum_{k=1}^K |\beta_k| \left\|\frac{\partial^2 t(T_k,x)}{\partial x_e \partial x_{e'}}\right\|_\infty \le \frac{1}{m_n}\sum_{k=1}^K |\beta_k| C'_k = O\left(\frac{\|\bbeta\|_1}{m_n}\right).
\]
Along with \textbf{Lemmas} \ref{weak_law}, \ref{weak_law_der} and \ref{weak_law_2der}, these imply

\[
\left\|\log\left(\frac{N(x)}{n}\right)\right\|_\infty \le O(m_n^2\|\bbeta\|_1) + o_P(1),
\]
\[
\left\|\frac{\partial}{\partial x_e}\log\left(\frac{N(x)}{n}\right)\right\|_\infty \le O(\|\bbeta\|_1) + o_P(1),
\]
\[
\left\|\frac{\partial^2}{\partial x_e \partial x_{e'}}\log\left(\frac{N(x)}{n}\right)\right\|_\infty \le  \frac{O(\|\bbeta\|_1)+o_P(1)}{m_n}.
\]

Since these combine additively, the overall bounds for the total energy function \(f(x)\) are:
\[
\left\|f\right\|_\infty \le m_n^2|\lambda| + O(m_n^2\|\bbeta\|_1) + o_P(1)
\]
\[
\left\|\frac{\partial f}{\partial x_e}\right\|_\infty \le C_H|\lambda| + O(\|\bbeta\|_1) + o_P(1)
\]
and
\[
\left\|\frac{\partial^2 f}{\partial x_e\partial x_{e'}}\right\|_\infty \le \frac{C_H'|\lambda| + O(\|\bbeta\|_1)+ o_P(1)}{m_n} .
\]
Here, the $o_P(1)$ terms account for the approximation of $\log(N(x)/n)$ by the ERGM log-likelihood, which holds uniformly for all $x$ under the condition $\binom{m_n}{2}=o(\log n)$ by adapting the lemmas in the appendix.

The next step is to bound the complexity of the gradient, $\nabla f$. Consider the function
\[g(x) = f_H(x) + f_{ERGM}(x)\]

Following the argument for Theorem 1.7 in \cite{chatterjee2016nonlinear}, the complexity of the gradient for a general ERGM Hamiltonian is well-controlled. The gradient vector $\nabla g(x)$ is a linear combination of the gradient vectors for each of the statistics $H, T_1, \dots, T_k$. As shown in Lemma 5.2 of \cite{chatterjee2016nonlinear}, each of these gradient maps has low complexity. The complexity of their linear combination can be bounded, leading to a bound on the size of the $\epsilon$-net for the image of $\nabla g$:
\[
\log |\mathcal{D}(\epsilon)| \le \frac{C'\,m\,B^4}{\epsilon^4}\log\frac{C'\,B^4}{\epsilon^4},
\]
where $C'$ is a constant and $B$ is an $O(1)$ term that depends on $|\lambda|$ and the norm of the parameter vector $\|\bbeta\|_1$:
\[
B = 1 + C_H|\lambda| + \sum_{k=1}^K |\beta_k| C_k.
\]

\textbf{Remark on boundedness of \(\lambda\): } Let $\hat{\lambda}_n$ denote the (unique) solution, and suppose $\hat{\lambda}_n \xrightarrow{P} \lambda^{\circ}$ by the consistency result established above. Choose $M > |\lambda^{\circ}| + 1$. Then for any $\varepsilon>0$, we have, for sufficiently large \(n\),
\[
P\!\left(|\hat{\lambda}_n| > M\right)
\;\le\;
P\!\left(|\hat{\lambda}_n - \lambda^{\circ}| > 1\right)
\;<\; \varepsilon,
\]
so $\{\hat{\lambda}_n\}_{n}$ is bounded in probability and, with probability tending to one, $\hat{\lambda}_n \in [-M,M]$. Consequently, in all preceding arguments, $\lvert\lambda\rvert$ can be bounded by \(M\) with high probability.
\vspace{0.5cm}

By \textbf{Lemma} \ref{weak_law_der}, the same \(\epsilon\)-net works for the image of \(\nabla f\) if \(\epsilon \) is not as small as \(o_p(\frac{1}{m_n})\). 

\subsubsection*{Application of the Log-Sum-Exp Approximation and Final Error Bounds}

With the derivative and complexity bounds for the energy function $f(x)$ established, we can now apply the general log-sum-exp approximation theorem (Theorem 1.6 in \cite{chatterjee2016nonlinear}). The theorem implies that:
\begin{align}
\log \sum_{x\in\{0,1\}^{\binom{m}{2}}}\exp(f(x))
&\le \sup_{x\in[0,1]^{\binom{m}{2}}}\Bigl(f(x)-I(x)\Bigr) \nonumber\\[1mm]
&\quad + \frac{1}{4}\Bigl(\sum_{e}b_e^2\Bigr)^{1/2}\epsilon
+3n_v\epsilon + \log |\mathcal{D}(\epsilon)| + S, \label{eq:upper_bound_final}
\end{align}
and
\begin{align}
\log \sum_{x\in\{0,1\}^{\binom{m}{2}}}\exp(f(x))
&\ge \sup_{x\in[0,1]^{\binom{m}{2}}}\Bigl(f(x)-I(x)\Bigr) - \frac{1}{2}\sum_{e} c_{ee}, \label{eq:lower_bound_final}
\end{align}
where:
\begin{itemize}
  \item \(n_v=\binom{m}{2}\sim \frac{m^2}{2}\) is the number of variables.
  \item \(b_e\) is the bound on the first derivative for coordinate \(e\), so that \(b_e\le B\) where $B$ is a constant depending on $\lambda$ and $\bbeta$.
  \item \(c_{ee}\) is the bound on the diagonal second derivative, which for our combined energy function is bounded by \(O(B/m)\) with high probability.
  \item \(S\) is the smoothness term which, after normalization by \(n_v\), is \(O(m^{-1/2})\) (see Lemma \ref{smoothness}).
\end{itemize}
Following the standard procedure for applying this theorem, we choose
\[
\epsilon = \Bigl(\frac{B^3\log m_n}{m_n}\Bigr)^{1/5}.
\]
After dividing by \(m^2\), the error terms in the upper bound \eqref{eq:upper_bound_final} converge to zero. The lower bound error term, $\frac{1}{m^2}\sum c_{ee} = \frac{1}{m^2} O(n_v \cdot B/m) = O(B/m)$, also vanishes. This leads to the following convergence result, stated as a lemma.

\begin{lemma}[Convergence of Partition Function for the Empirical ERGM]
\label{lem:partition_func_conv_gen}
Assume that \(m=m_n\to\infty\) and that
\[
\binom{m_n}{2}=o(\log n).
\]
Let the sample graphs be drawn from an ERGM with parameters $\bbeta = (\beta_1, \dots, \beta_K)$. Then there exist constants \(c,C>0\) (depending on the choice of statistics $H, T_1, \dots, T_K$) such that, with probability tending to 1 as \(n\to\infty\),
\[
-c\,B\, m^{-1} \le \frac{\log Z_m(\lambda)}{m^2} - L_m(\lambda) \le C\,B^{8/5} m^{-1/5} (\log m)^{1/5}\Bigl(1+\frac{\log B}{\log m}\Bigr) + C\,B^2 m^{-1/2},
\]
where
\[
Z_m(\lambda) =\sum_{G}\exp\!\Biggl(-\frac{\lambda}{m^{v(H)-2}}H(G) + \log\Bigl(\frac{N(G)}{n}\Bigr)\Biggr),
\]
\[
L_m(\lambda) = \sup_{x\in[0,1]^{\binom{m}{2}}}\Biggl\{\frac{1}{m^2}\left(-\frac{\lambda}{m^{v(H)-2}}H(x)+\log\Bigl(\frac{N(x)}{n}\Bigr)\right)-\frac{I(x)}{m^2}\Biggr\},
\]
and the constant $B$ is given by
\[
B = 1+\frac{|\lambda|}{\lvert Aut(H)\rvert}+\sum_{k=1}^K |\beta_k|.
\]
\end{lemma}

\textbf{The limit of $L_m$}

\medskip

To find the limit of $L_m$, we decompose the term $\frac{1}{m^2}\log\left(\frac{N(x)}{n}\right)$.
For any graph $x \in \{0,1\}^{\binom{m}{2}}$, which represents an adjacency vector, we can write the log of the empirical frequency as the sum of the log-probability under the $\mathbb{P}_{\beta}$ model and a deviation term:
\[
\frac{1}{m^2}\log\left(\frac{N(x)}{n}\right) = \frac{1}{m^2}\log p_{\beta}(x) + \frac{1}{m^2}\log\left(\frac{N(x)}{n}\right) - \frac{1}{m^2}\log p_{\beta}(x)
\]

Substituting this decomposition back into the expression for $L_m$, we obtain:
\begin{align*}
L_m =& \sup_{x\in[0,1]^{\binom{m}{2}}}\Biggl\{-\frac{\lambda}{m^{v(H)}}H(x) - \frac{I(x)}{m^2} + \frac{1}{m^2}\log p_{\beta}(x)  \\
& \qquad\qquad\qquad + \frac{1}{m^2}\log\left(\frac{N(x)}{n}\right) - \frac{1}{m^2}\log p_{\beta}(x)
 \Biggr\}\\
\le & \sup_{x\in[0,1]^{\binom{m}{2}}}\Biggl\{-\frac{\lambda}{m^{v(H)}}H(x) - \frac{I(x)}{m^2} + \frac{1}{m^2}\log p_{\beta}(x)\Biggr\} \\
& \qquad + \sup_{x\in[0,1]^{\binom{m}{2}}} \left\{ \frac{1}{m^2}\log\left(\frac{N(x)}{n}\right) - \frac{1}{m^2}\log p_{\beta}(x)
 \right\}.
\end{align*}
By the uniform weak law of large numbers (Lemma \ref{weak_law}), the ratio inside the logarithm converges to 1 in probability, uniformly over all graphs $x$. Therefore, the error term vanishes as $n \to \infty$:
\[
\frac{1}{m^2}\log\left(\frac{N(x)}{n}\right) - \frac{1}{m^2}\log p_{\beta}(x)
= o_P(1).
\]
A similar lower bound can be established. This means that the limit of $L_m$ is determined by the limit of the main variational term
$$\sup_{x\in[0,1]^{\binom{m}{2}}}\Biggl\{-\frac{\lambda}{m^{v(H)}}H(x) - \frac{I(x)}{m^2} + \frac{1}{m^2}\log p_{\beta}(x)\Biggr\}.$$
Now we take limits as \(m\rightarrow\infty\):
\begin{align*}
&\lim_{m\to\infty} \Biggl[  \sup_{x\in[0,1]^{\binom{m}{2}}}\Biggl\{-\frac{\lambda}{m^{v(H)}}H(x) + \frac{1}{m^2}\log p_{\beta}(x) - \frac{I(x)}{m^2}\Biggr\} \Biggr] \\
&= \lim_{m\to\infty} \Biggl[  \sup_{x\in[0,1]^{\binom{m}{2}}}\Biggl\{-\frac{\lambda}{m^{v(H)}}H(x) + \sum_{k=1}^{K}\beta_k t(T_k,x) - \frac{I(x)}{m^2}\Biggr\}  - \frac{1}{m^2}\log Z_m(\beta)\Biggr] \\
&= \sup_{x\in\cup_{m}[0,1]^{{m\choose 2}}}\Biggl\{-\frac{\lambda}{m^{v(H)}}H(x) + \sum_{k=1}^{K}\beta_k t(T_k,x) - \frac{I(x)}{m^2}\Biggr\} - \sup_{\tilde{h}\in\tilde{W}}\left( \sum_{k=1}^K \beta_k t(T_k, h)-\frac{1}{2} I(h)\right)\\
&= \sup_{x\in\cup_{m}[0,1]^{{m\choose 2}}}\Biggl\{-\frac{\lambda}{\lvert Aut(H)\rvert}t(H,x) + \sum_{k=1}^{K}\beta_k t(T_k,x) - \frac{I(x)}{m^2}\Biggr\} - \sup_{\tilde{h}\in\tilde{W}}\left( \sum_{k=1}^K \beta_k t(T_k, h)-\frac{1}{2} I(h)\right)\\
&= \sup_{\tilde{h}\in\tilde{W}}\left(-\frac{\lambda}{\lvert Aut(H) \rvert} t(H,h) + \sum_{k=1}^K \beta_k t(T_k, h)-\frac{1}{2} I(h)\right) - \sup_{\tilde{h}\in\tilde{W}}\left( \sum_{k=1}^K \beta_k t(T_k, h)-\frac{1}{2} I(h)\right).
\end{align*}
where we have used the result \(\frac{t(H,G)}{\lvert Aut(H)\rvert}  = \frac{H(G)}{(m)_{v(H)}} + O_H(1/m)\) to replace \(\frac{H(x)}{m^{v(H)}}\) by \(\frac{t(H,x)}{\lvert Aut(H)\rvert}\), and the last step follows by the following lemma whose proof is given in the appendix

\begin{lemma}{\label{density}}
Fix finite simple graphs $T_1,\dots,T_K$ and $\beta_1,\dots,\beta_K\in\mathbb{R}$.
Let
\[
F(h):=\sum_{k=1}^K \beta_k\,t(T_k,h)\;-\;\frac{1}{2}I(h)
\]
for \(h\in W\).
Let $S\subset W$ be the class of step functions on \([0,1]^2\) taking values in \([0,1]\).
Then
\[
\sup_{h\in W} F(h)\;=\;\sup_{h\in S} F(h),
\]

\end{lemma}

Note: The space \(S\) can be identified with the space \(\cup_m [0,1]^{m\choose 2}\) as follows: represent an element of \([0,1]^{m\choose 2}\) as the adjacency vector of a weighted graph, and then consider the adjacency matrix. The adjacency matrix can be represented as a step function on \([0,1]^2\) in a natural way. The other direction is also straightforward.

This shows that $\lim_{m\to\infty} L_m$ exists in probability and is equal to the desired variational formula, which in turn implies that
\[
\frac{1}{m_n^2}\log\left(\frac{1}{n}\sum_{i=1}^n e^{-\frac{\lambda}{m_n^{v(H)-2}} h_{n,i}}\right)
\]
converges in probability to $\mathfrak{g}(\lambda)$. The convergence of the root follows from an argument similar to the proof of Theorem \ref{normal_consistency_1}.

\end{proof}

\begin{proof}[Proof of Theorem \ref{sharp_rates}]

Expanding as before about $\lambda^{\circ}$, we get

\begin{align*}
&0 = \sum\limits_{i=1}^{n} h_{n,i} e^{-\frac{\lambda^{\circ}}{m_n^{v(H)-2}}h_{n,i}} + (\hat{\lambda}_n - \lambda^{\circ})\sum\limits_{i=1}^{n} \frac{-h_{n,i}^2}{m_n^{v(H)-2}}e^{-\frac{\tilde{\lambda}_n}{m_n^{v(H)-2}}h_{n,i}}\\
&\implies m_n^2(\hat{\lambda}_n - \lambda^{\circ}) = \frac{\sum\limits_{i=1}^{n}\frac{h_{n,i}}{m_n^{v(H)}}e^{-\frac{\lambda^{\circ}}{m_n^{v(H)-2}}h_{n,i}}}{\sum\limits_{i=1}^{n}\frac{h_{n,i}^2}{m_n^{2v(H)}}e^{-\frac{\tilde{\lambda}_n}{m_n^{v(H)-2}}h_{n,i}} }.
\end{align*}
As mentioned in the proof outline we will deal with the numerator and denominator of the RHS of the above display separately. For the numerator we will prove \eqref{num_conv}, while for the denominator we will establish \eqref{denom_conv}. Since the statement of the theorem follows from combining \eqref{num_conv} and \eqref{denom_conv}, the rest of the proof will be devoted to proving \eqref{num_conv} and \eqref{denom_conv}. From the proof of Theorem \ref{normal_consistency_2} we have:
\begin{equation}\frac{1}{m_n^2}\log(\frac{1}{n}\sum_{i=1}^n e^{-\frac{\lambda}{m_n^{v(H)-2}} h_{n,i}})\xrightarrow{P}\mathfrak{g}(\lambda)\label{scaled_part_conv_1}\end{equation}
where 
\begin{align*} \mathfrak{g}(\lambda) = &  \underbrace{\frac{p_0^{e(H)}}{\lvert Aut(H)\rvert}}_{C_1}\lambda  + \underbrace{\sup_{\tilde{h}\in\tilde{W}}\left(-\frac{\lambda}{\lvert Aut(H)\rvert} t(H,h) + \sum_{k=1}^K \beta_k t(T_k, h)-\frac{1}{2} I (h)\right)}_{G(\lambda)}\\
&- \sup_{\tilde{h}\in\tilde{W}}\left(\sum_{k=1}^K \beta_k t(T_k, h) -\frac{1}{2} I (h)\right).
\end{align*}

Define the following quantities:
\begin{align*}
    g(\lambda, \tilde{h}) &= -\frac{\lambda}{\lvert Aut(H)\rvert} t(H,h) + \sum_{k=1}^K \beta_k u^{e(T_k)} -\frac{1}{2} I (h)\\
    G(\lambda) &= \sup_{\tilde{h}\in\tilde{W}} g(\lambda, \tilde{h})\\
    H^*(\lambda) & = \argmax_{\tilde{h}\in\tilde{W}} g(\lambda, \tilde{h})
\end{align*}

(Here, \(\mathrm{arg max}\) is the set of all \(\tilde{h}\in \tilde{W}\) where the supremum is attained. The set is non-empty by the compactness of \(\tilde{W}\)). Let \(\overline{\co}\) denote the closed convex hull. Then by Danskin's Theorem \cite{danskin1966theory}, the subdifferential of $G$ at $\lambda^{\circ}$ is:
    $$ \subdiff G(\lambda^{\circ}) = \overline{\co} \left\{ \nabla_\lambda g(\lambda^{\circ}, \tilde{h}) \mid \tilde{h} \in H^*(\lambda^{\circ}) \right\} = \overline{\co} \left\{ -\frac{t(H,h)}{\lvert Aut(H)\rvert} \mid \tilde{h} \in H^*(\lambda^{\circ}) \right\} $$
        
We are given that the supremum is attained at a unique point $u^* \in [0,1]\subseteq \tilde{W}$ at $\lambda^{\circ}$. Thus, $H^*(\lambda^{\circ}) = \{u^*\}$ and 
    $$ \subdiff G(\lambda^{\circ}) = \overline{\co} \left\{ -\frac{t(H,u^*)}{\lvert Aut(H)\rvert} \right\} = \left\{ -\frac{t(H,u^*)}{\lvert Aut(H)\rvert} \right\}.$$
Since $\subdiff G(\lambda^{\circ})$ is a singleton, $G$ is differentiable at $\lambda^{\circ}$ with $G'(\lambda^{\circ}) = -\frac{t(H,u^*)}{\lvert Aut(H)\rvert}$. Thus we have  $\mathfrak{g}(\lambda) = C_1 \lambda + G(\lambda) + \text{const}$. is differentiable at $\lambda^{\circ}$ with derivative:
    $$ \mathfrak{g}'(\lambda^{\circ}) = C_1 + G'(\lambda^{\circ}) = \frac{p_0^{e(H)}}{\lvert Aut(H)\rvert} - \frac{t(H,u^*)}{\lvert Aut(H)\rvert} = \frac{p_0^{e(H)}}{\lvert Aut(H)\rvert} - \frac{(u^*)^{e(H)}}{\lvert Aut(H)\rvert}$$

Further, observe that

\begin{itemize}
\item The left hand side of (\ref{scaled_part_conv_1}) is a sequence of convex functions.
\item The right hand side of (\ref{scaled_part_conv_1}), being a pointwise limit of convex functions, is convex.
\item The left hand side is a sequence of differentiable functions.
\item We have shown that the right hand side is differentiable at $\lambda^{\circ}$ as the supremum is attained at a unique point.
\end{itemize}

Take any subsequence of the left hand side. Then it has a further subsequence that converges almost surely. Using the above conditions on the subsequence, we get that the subsequential derivatives at $\lambda^{\circ}$ converge almost surely to the derivative of the right hand side. Hence we get that every subsequence of derivatives of the left hand side has a further subsequence converging almost surely. Therefore, we have the convergence in probability of derivatives at $\lambda^{\circ}$ :

$$\frac{\sum\limits_{i=1}^{n}\frac{-h_{n,i}}{m_n^{v(H)}}e^{-\frac{\lambda^{\circ}}{m_n^{v(H)-2}}h_{n,i}}}{\sum\limits_{i=1}^{n}e^{-\frac{\lambda^{\circ}}{m_n^{v(H)-2}}h_{n,i}}}\xrightarrow{P} \frac{p_0^{e(H)}}{\lvert Aut(H)\rvert} - \frac{(u^*)^{e(H)}}{\lvert Aut(H)\rvert}$$

Thus the convergence outlined in \eqref{num_conv} is established. Next to analyze convergence in  \eqref{denom_conv}, we require the limit of 
$$
\frac{\sum\limits_{i=1}^{n}\frac{h_{n,i}^2}{m_n^{2v(H)}}e^{-\frac{\tilde{\lambda}_n}{m_n^{v(H)-2}}h_{n,i}}}{\sum\limits_{i=1}^{n}e^{-\frac{\tilde{\lambda}_n}{m_n^{v(H)-2}}h_{n,i}}}
$$

To that end, we consider 

\begin{align*}
& \frac{1}{m_n^2} \log \sum_{G} e^{-\frac{\lambda}{m_n^{v(H)-2}} \left\{ H(G) - m_n^{v(H)} \frac{p_0^{e(H)}}{\lvert Aut(H)\rvert}\right\}  + \log \frac{N(G)}{n} + \frac{\alpha}{m_n^{2v(H)-2}} \left\{ H(G) - m_n^{v(H)}\frac{p_0^{e(H)}}{\lvert Aut(H)\rvert} \right\}^2} \\
=& \frac{1}{m_n^2} \log \sum_{G} e^{m_n^2 \left[ -\frac{\lambda}{m_n^{v(H)}} \left\{ H(G) - m_n^{v(H)} \frac{p_0^{e(H)}}{\lvert Aut(H)\rvert} \right\} + \frac{1}{m_n^2} \log \frac{N(G)}{n} + \frac{\alpha}{m_n^{2v(H)}} \left\{ H(G) - m_n^{v(H)} \frac{p_0^{e(H)}}{\lvert Aut(H)\rvert} \right\}^2 \right]} \\
=& \frac{1}{m_n^2} \log \sum_{G} e^{m_n^2 \left[ \left\{-\frac{\lambda}{\lvert Aut(H)\rvert} - \frac{2\alpha p_0^{e(H)}}{\lvert Aut(H)\rvert^2}\right\} \frac{H(G)}{m_n^{v(H)}/\lvert Aut(H)\rvert} + \frac{1}{m_n^2} \log \frac{N(G)}{n} + \frac{\alpha}{\lvert Aut(H)\rvert^2} \left( \frac{H(G)}{m_n^{v(H)}/\lvert Aut(H)\rvert} \right)^2 \right]}  \\
&+\frac{\lambda p_0^{e(H)}}{\lvert Aut(H)\rvert} + \frac{\alpha p_0^{2e(H)}}{\lvert Aut(H)\rvert^2}
\end{align*}

As before, we denote this by
\begin{equation}
\label{eq:f_def2}
\frac{1}{m_n^2} \log \sum_{G} e^{f(G)} + \frac{\lambda p_0^{e(H)}}{\lvert Aut(H)\rvert} + \frac{\alpha p_0^{2e(H)}}{\lvert Aut(H)\rvert^2}
\end{equation}

where 
\begin{align}f(x)=&m_n^2 \left[ \left\{-\frac{\lambda}{\lvert Aut(H)\rvert} - \frac{2\alpha p_0^{e(H)}}{\lvert Aut(H)\rvert^2}\right\} \frac{H(x)}{m_n^{v(H)}/\lvert Aut(H)\rvert} + \frac{1}{m_n^2} \log \frac{N(x)}{n}\right.\nonumber\\
&\left.+ \frac{\alpha}{\lvert Aut(H)\rvert^2} \left( \frac{H(x)}{m_n^{v(H)}/\lvert Aut(H)\rvert} \right)^2 \right]
\label{energy_function}\end{align}

\subsection*{Squared \(H\) Statistic \(T_H(x)\)}
Consider
\[
T_H(x)=\Bigl(H(x)\Bigr)^2.
\]
Its first derivative is
\[
\frac{\partial T_H(x)}{\partial x_{e}} = 2\, H(x) \frac{\partial H(x)}{\partial x_{e}},
\]
and using \(H(x)\le \binom{m}{v(H)}=O(m^{v(H)})\) and \(\Bigl|\frac{\partial H(x)}{\partial x_{e}}\Bigr|\le C\,m^{v(H)-2}\) we obtain
\[
\Bigl|\frac{\partial T_H(x)}{\partial x_{e}}\Bigr|\le 2C\, O(m^{2v(H)-2}).
\]
After multiplying by the normalization factor \(\alpha/m^{2v(H)-2}\), we have
\[
\Bigl|\frac{\partial}{\partial x_{e}}\Bigl(\frac{\alpha}{m^{2v(H)-2}}T_H(x)\Bigr)\Bigr|\le O(|\alpha|).
\]
Similarly, the second derivative of \(T_H(x)\) is bounded by \(O(m^{2v(H)-3})\), so that after the factor \(\alpha/m^{2v(H)-2}\) we obtain
\[
\Bigl|\frac{\partial^2}{\partial x_{e}\partial x_{e'}}\Bigl(\frac{\alpha}{m^{2v(H)-2}}T_H(x)\Bigr)\Bigr|\le \frac{C''|\alpha|}{m_n},
\]
for some constant \(C''\).

\subsubsection*{Combination of Edge, \(H\), and Squared \(H\) Count Bounds}

The energy function in \eqref{energy_function} now decomposes as
\[
f(x)=\log\Bigl(\frac{N(x)}{n}\Bigr)+f_H(x)+f_S(x),
\]
where
\[
f_H(x)=-\frac{\lambda+\frac{2\alpha p_0^{e(H)}}{\lvert Aut(H)\rvert}}{m^{v(H)-2}}\,H(x),\quad
f_S(x)=\frac{\alpha}{m^{2v(H)-2}}\,T_3(x),\quad 
T_3(x)=\Bigl(H(x)\Bigr)^2.
\]

For the \(H\) part \(f_H(x)\), standard combinatorial arguments (as in \cite{chatterjee2016nonlinear}) yield the bounds
\[
\|H\| \le C\, m^{v(H)},\quad
\Bigl\|\frac{\partial H}{\partial x_e}\Bigr\|\le C\, m^{v(H)-2},\quad
\Bigl\|\frac{\partial^2 H}{\partial x_e\partial x_{e'}}\Bigr\|\le Cm^{v(H)-3},
\]
for edges \(e,e'\) . It follows that
\[
\left\|f_H\right\|_\infty \leq m_n^2 \lvert \lambda+\frac{2\alpha p_0^{e(H)}}{\lvert Aut(H)\rvert}\rvert
\]
\[
\left\|\frac{\partial f_H}{\partial x_e}\right\| \le \frac{|\lambda+\frac{2\alpha p_0^{e(H)}}{\lvert Aut(H)\rvert}|}{m^{v(H)-2}}\,C\,m^{v(H)-2} = C\,|\lambda + \frac{2\alpha p_0^{e(H)}}{\lvert Aut(H)\rvert}|,
\]
and
\[
\left\|\frac{\partial^2 f_H}{\partial x_e\partial x_{e'}}\right\| \le \frac{|\lambda + \frac{2\alpha p_0^{e(H)}}{\lvert Aut(H)\rvert}|}{m^{v(H)-2}}\cdot Cm^{v(H)-3} = \frac{C\,|\lambda + \frac{2\alpha p_0^{e(H)}}{\lvert Aut(H)\rvert}|}{m_n}.
\]
For \(\log\left(\frac{N(x)}{n}\right)\), as shown previously,
\[
\left\|\log\left(\frac{N(x)}{n}\right)\right\|_\infty \le O(m_n^2\|\bbeta\|_1) + o_P(1),
\]
\[
\left\|\frac{\partial}{\partial x_e}\log\left(\frac{N(x)}{n}\right)\right\|_\infty \le O(\|\bbeta\|_1) + o_P(1),
\]
\[
\left\|\frac{\partial^2}{\partial x_e \partial x_{e'}}\log\left(\frac{N(x)}{n}\right)\right\|_\infty \le  \frac{O(\|\bbeta\|_1)+o_P(1)}{m_n}.
\]

Thus, its contribution to the gradient is essentially constant and (up to an \(o_P(1)\) error) adds a fixed vector whose covering number is one.
Combining these with the bounds for the squared \(H\) term \(f_S(x)\), we get
\[
\left\|f\right\|_\infty \le  m_n^2|\lambda+\frac{2\alpha p_0^{e(H)}}{\lvert Aut(H)\rvert}| + O(m_n^2\|\bbeta\|_1) + m_n^2|\alpha|+ o_P(1),
\]
the derivative satisfies
\[
\left|\frac{\partial f}{\partial x_e}\right| \le O(\|\bbeta\|_1) + C\,|\lambda + \frac{2\alpha p_0^{e(H)}}{\lvert Aut(H)\rvert}| + O(|\alpha|) + o_P(1),
\]
and the second derivative is bounded by
\[
\left|\frac{\partial^2 f}{\partial x_e\partial x_{e'}}\right| \le \frac{O(\|\bbeta\|_1)+C\,|\lambda + \frac{2\alpha p_0^{e(H)}}{\lvert Aut(H)\rvert}|+C''|\alpha|+ o_P(1)}{m_n} .
\]
As before, the next step is to bound the complexity of the gradient, $\nabla f$.  Following the previous line of arguments one obtains that there exists a constant \(C'>0\) such that the net satisfies
\[
\log |\mathcal{D}(\epsilon)| \le \frac{C'B^4\,m\,}{\epsilon^4}\log\frac{C'B^4}{\epsilon^4},
\]
with
\[
B=1+|\frac{\lambda}{\lvert Aut(H)\rvert}+\frac{2\alpha p_0^{e(H)}}{\lvert Aut(H)\rvert^2}|+\frac{|\alpha|}{\lvert Aut(H)\rvert}+\sum_{k=1}^K |\beta_k| C_k.
\]

Applying the log-sum approximation  (Theorem 1.6 in \cite{chatterjee2016nonlinear}), we get the following result.

\begin{align*}&\frac{1}{m_n^2}\log\left(\frac{1}{n}\sum_{i=1}^n e^{-\frac{\lambda}{m_n^{v(H)-2}} h_{n,i} + \frac{\alpha}{m_n^{2v(H)-2}}h_{n,i}^2}\right)\\
\xrightarrow{P}&\sup_{\tilde{h}\in\tilde{W}}\bigg(-\left(\frac{\lambda}{|Aut(H)|}+\frac{2\alpha p_0^{e(H)}}{|Aut(H)|^2}\right) t(H,h) + \sum_{k=1}^K \beta_k t(T_k, h) \\
     &+ \frac{\alpha}{|Aut(H)|^2} t(H,h)^2 - \frac{1}{2} I(h)\bigg) + \frac{\lambda p_0^{e(H)}}{\lvert Aut(H)\rvert} + \frac{\alpha p_0^{2e(H)}}{\lvert Aut(H)\rvert^2} \\&- \sup_{\tilde{h}\in\tilde{W}}\bigg( \sum_{k=1}^K \beta_k t(T_k, h) 
      - \frac{1}{2} I(h)\bigg)\\
\end{align*}
uniformly in \(\lambda\) (see proof of Theorem \ref{normal_consistency_2} for details). Since we also have \(\tilde{\lambda}_n\xrightarrow{P} \lambda^{\circ}\), we get

\begin{align}
&\frac{1}{m_n^2}\log\left(\frac{1}{n}\sum_{i=1}^n 
e^{-\frac{\tilde{\lambda}_n}{m_n} h_{n,i} + \frac{\alpha}{m_n^{2v(H)-2}}h_{n,i}^2}\right)\\
\xrightarrow{P}&\sup_{\tilde{h}\in\tilde{W}}\bigg(-\left(\frac{\lambda^{\circ}}{|Aut(H)|}
+\frac{2\alpha p_0^{e(H)}}{|Aut(H)|^2}\right) t(H,h) 
+ \sum_{k=1}^K \beta_k t(T_k, h) \nonumber \\
&+ \frac{\alpha}{|Aut(H)|^2} t(H,h)^2 - \frac{1}{2} I(h)\bigg) 
+ \frac{\lambda^{\circ} p_0^{e(H)}}{\lvert Aut(H)\rvert} 
+ \frac{\alpha p_0^{2e(H)}}{\lvert Aut(H)\rvert^2} \nonumber \\
&- \sup_{\tilde{h}\in\tilde{W}}\bigg( \sum_{k=1}^K \beta_k t(T_k, h) 
- \frac{1}{2} I(h)\bigg)\label{scaled_part_conv_2}
\end{align}

Let $F(\alpha)$ be defined as follows:

\begin{align*}
 F(\alpha) &= \underbrace{-\sup_{\tilde{h}\in\tilde{W}}\bigg( \sum_{k=1}^K \beta_k t(T_k, h) 
      - \frac{1}{2} I(h)\bigg)+\frac{\lambda^{\circ} p_0^{e(H)}}{\lvert Aut(H)\rvert}}_{\text{Const w.r.t } \alpha} + \frac{\alpha p_0^{2e(H)}}{\lvert Aut(H)\rvert^2} \\
&\quad + \sup_{\tilde{h}\in\tilde{W}}\underbrace{\left(-\left(\frac{\lambda^{\circ}}{\lvert Aut(H)\rvert}+\frac{2\alpha p_0^{e(H)}}{\lvert Aut(H)\rvert^2}\right) t(H,h) + \sum_{k=1}^K \beta_k t(T_k, h) + \frac{\alpha}{\lvert Aut(H)\rvert^2} t(H,h)^2-\frac{1}{2} I (h)\right)}_{g(\alpha, \tilde{h})}
\end{align*}

We note that $F$ is convex in $\alpha$, being the pointwise limit of convex functions . Alternatively, $g(\alpha, \tilde{h})$ is linear in $\alpha$, hence $\sup_{\tilde{h}} g(\alpha, \tilde{h})$ is convex. $F(\alpha)$ is the sum of this convex function and a linear term $\frac{\alpha p_0^{2e(H)}}{\lvert Aut(H)\rvert^2}$ (plus constants), hence convex. We analyze its differentiability at $\alpha=0$. Define 

\begin{align*}
G(\alpha):&= \sup_{\tilde{h}\in\tilde{W}} g(\alpha, \tilde{h})\\
H^*(0) &= \argmax_{\tilde{h}\in\tilde{W}} g(0, \tilde{h})
\end{align*}
and we note that
\begin{align*}
\nabla_\alpha g(\alpha, \tilde{h})&= -\frac{2p_0^{e(H)}}{\lvert Aut(H)\rvert^2} t(H,h) + \frac{1}{\lvert Aut(H)\rvert^2} t(H,h)^2\\
g(0, \tilde{h}) &= -\frac{\lambda^{\circ}}{\lvert Aut(H)\rvert} t(H,h) + \sum_{k=1}^K \beta_k t(T_k, h)  - \frac{1}{2} I(h),
\end{align*}
   Further we are given that this supremum is attained at a unique point $u^* \in \tilde{W}$, which implies that  $H^*(0) = \{u^*\}$. By Danskin's Theorem \cite{danskin1966theory}, $G(\alpha)$ is differentiable at $\alpha=0$, and its derivative is:
    $$ G'(0) = \nabla_\alpha g(0, \tilde{h}^*) = -\frac{2p_0^{e(H)}}{\lvert Aut(H)\rvert^2} t(H,u^*) + \frac{1}{\lvert Aut(H)\rvert^2} t(H,u^*)^2.$$
    Since $F(\alpha) = \text{Const} + \frac{\alpha p_0^{2e(H)}}{\lvert Aut(H)\rvert^2} + G(\alpha)$, its derivative at $\alpha=0$ is given by:
    \begin{align*} F'(0) &= \frac{p_0^{2e(H)}}{\lvert Aut(H)\rvert^2} + G'(0)\\
    =&\frac{p_0^{2e(H)}}{\lvert Aut(H)\rvert^2} - \frac{2p_0^{e(H)}}{\lvert Aut(H)\rvert^2} t(H,u^*) + \frac{1}{\lvert Aut(H)\rvert^2}t(H,u^*)^2 \\ =&\frac{p_0^{2e(H)}}{\lvert Aut(H)\rvert^2} - \frac{2p_0^{e(H)}}{\lvert Aut(H)\rvert^2} u^{*{e(H)}} + \frac{1}{\lvert Aut(H)\rvert^2} u^{*{2e(H)}} \end{align*}

Further, observe that

\begin{itemize}
\item The left hand side of (\ref{scaled_part_conv_2}) is a sequence of convex functions  of \(\alpha\).
\item The right hand side of (\ref{scaled_part_conv_2}), being a pointwise limit of convex functions, is convex.
\item The left hand side of (\ref{scaled_part_conv_2}) is a sequence of differentiable functions of \(\alpha\).
\item From \cite{chatterjee2013estimating}, we know that the right hand side is differentiable at $0$ as the supremum is attained at a unique point.
\end{itemize}

By passing to subsequences as before, we get

\begin{align*}\frac{\sum\limits_{i=1}^{n}\frac{h_{n,i}^2}{m_n^{2v(H)}} e^{-\frac{\lambda}{m_n^{v(H)-2}} h_{n,i} + \frac{\alpha}{m_n^{2v(H)-2}}h_{n,i}^2}}{\sum\limits_{i=1}^{n} e^{-\frac{\lambda}{m_n^{v(H)-2}} h_{n,i} + \frac{\alpha}{m_n^{2v(H)-2}}h_{n,i}^2}}&\xrightarrow{P} \frac{p_0^{2e(H)}}{\lvert Aut(H)\rvert^2} - \frac{2p_0^{e(H)}u^{*{e(H)}}}{\lvert Aut(H)\rvert^2} + \frac{u^{*{2e(H)}}}{\lvert Aut(H)\rvert^2}\\
&= \left(\frac{u^{*{e(H)}}}{\lvert Aut(H)\rvert} - \frac{p_0^{e(H)}}{\lvert Aut(H)\rvert}\right)^2
\end{align*}

\end{proof}

\section{Proof of technical lemmas \label{sec:proof_lemmas}}

\begin{lemma}[Uniform Weak Law of Large Numbers]{\label{weak_law}}
Let \(\mathcal{G}_{n,1},\dots, \mathcal{G}_{n,n}\) be independent random graphs sampled from \(\mathbb{P}_{\beta}\) model, where \(m_n\to\infty\) as \(n\to\infty\). For any graph \(G\) on \(m_n\) vertices, we have
\[
\mathbb{E}\bigl[\mathds{1}(\mathcal{G}_{n,i} = G)\bigr] = p_{\beta}(G).
\]
Define
\[
X_{i,G} := \frac{\mathds{1}(\mathcal{G}_{n,i}=G)}{p_{\beta}(G)}.
\]
Then, if
\[
{m_n\choose 2} = o(\log n),
\]
for every \(\epsilon>0\) we have
\[
\mathbb{P}\Biggl(\sup_{G\in \mathcal{G}_{m_n}} \Bigl|\frac{1}{n}\sum_{i=1}^{n} X_{i,G} - 1\Bigr| > \frac{\epsilon}{m_n}\Biggr)
\rightarrow 0
\]

\end{lemma}

\begin{proof}
For each fixed graph \(G\) on \(m_n\) vertices, note that
\[
\mathbb{E}[X_{i,G}] = \frac{\mathbb{E}\bigl[\mathds{1}(\mathcal{G}_{n,i}=G)\bigr]}{p_{\beta}(G)} = 1.
\]
Thus, for each \(G\), the random variables \(\{X_{i,G}\}_{i=1}^n\) are i.i.d. with mean 1.

Moreover, the variance of the indicator is
\[
\operatorname{Var}\bigl(\mathds{1}(\mathcal{G}_{n,i}=G)\bigr) = p_{\beta}(G)\Bigl(1-p_{\beta}(G)\Bigr),
\]
so that
\[
\operatorname{Var}(X_{i,G}) = \frac{1}{p_{\beta}(G)} - 1.
\]
By Chebyshev's inequality, for any fixed \(G\) and \(\epsilon>0\)
\[
\mathbb{P}\Biggl(\Bigl|\frac{1}{n}\sum_{i=1}^{n} X_{i,G} - 1\Bigr| > \frac{\epsilon}{m_n}\Biggr)
\le \frac{m_n^2}{n\epsilon^2}\operatorname{Var}(X_{1,G}).
\]
A union bound over the (at most) \(2^{\binom{m_n}{2}}\) graphs on \(m_n\) vertices then yields
\[
\mathbb{P}\Biggl(\sup_{G\in \mathcal{G}_{m_n}} \Bigl|\frac{1}{n}\sum_{i=1}^{n} X_{i,G} - 1\Bigr| > \frac{\epsilon}{m_n}\Biggr)
\le 2^{\binom{m_n}{2}} \cdot \frac{m_n^2}{n\epsilon^2}\left(\frac{1}{\min_{G}(p_{\beta}(G))}-1\right).
\]
Under the assumption \({m_n\choose 2}=o(\log n)\), the right-hand side converges to zero as \(n\to\infty\). This completes the proof.
\end{proof}

\begin{lemma}[Uniform Convergence of the Log-Derivative]{\label{weak_law_der}}
Let \(\mathcal{G}_{n,1},\dots, \mathcal{G}_{n,n}\) be independent random graphs sampled from the \(\mathbb{P}_{\beta}\) model, with \(m_n\to\infty\) in such a way that
\[
\binom{m_n}{2} = o(\log n).
\]
For each graph \(G\) on \(m_n\) vertices, let 
\[
N(G)=\sum_{i=1}^{n} \mathds{1}\bigl(\mathcal{G}_{n,i}=G\bigr)
\]
be the number of occurrences of \(G\) in the sample. Extend the function 
\[
\log\Bigl(\frac{N(G)}{n}\Bigr)
\]
to \([0,1]^{\binom{m_n}{2}}\) multilinearly and denote the extension by \(\log\Bigl(\frac{N(x)}{n}\Bigr)\).

Then, for every \(\epsilon>0\) we have
\[
\mathbb{P}\!\Biggl(\sup_{x\in[0,1]^{\binom{m_n}{2}}}\; \sup_{1\le i<j\le m_n}\; \Bigl|\frac{\partial}{\partial x_{ij}}\log\Bigl(\frac{N(x)}{n}\Bigr)-\frac{\partial}{\partial x_{ij}}\log p_{\beta}(x)\Bigr| > \frac{\epsilon}{m_n}\Biggr)
\rightarrow 0
\]

\end{lemma}

\begin{proof}
By the uniform weak law Lemma \ref{weak_law}, for every graph \(G\) on \(m_n\) vertices we have
\[
\frac{N(G)}{n} = p_{\beta}(G)\Bigl(1+\delta_n(G)\Bigr),
\]
with \(m_n\delta_n(G) \to 0\) in probability uniformly over \(G\). Taking logarithms, we write
\[
\log\frac{N(G)}{n} = \log p_{\beta}(G) + \eta_n(G),
\]
where \(\eta_n(G)=\log(1+\delta_n(G))\) and \(m_n\eta_n(G)\to 0\) uniformly in probability over all graphs \(G\).

By multilinear extension, for any \(x\in[0,1]^{\binom{m_n}{2}}\) the extended function \(\log\Bigl(\frac{N(x)}{n}\Bigr)\) agrees with the above on the vertices (i.e., on the set of all graphs). For an edge \(e=ij\), define the finite-difference derivative
\[
D_e(x) := \frac{\partial}{\partial x_{ij}}\log\Bigl(\frac{N(x)}{n}\Bigr)
=\log\Bigl(\frac{N(x+e)}{n}\Bigr) - \log\Bigl(\frac{N(x-e)}{n}\Bigr),
\]
where \(x+e\) denotes that the \(e^{\text{th}}\) coordinate is set to 1 and \(x-e\) that it is set to 0.
Then we have
\[
\begin{aligned}
\log\Bigl(\frac{N(x+e)}{n}\Bigr) - \log\Bigl(\frac{N(x-e)}{n}\Bigr)
& = \log p_{\beta}(x+e) - \log p_{\beta}(x-e) + \Delta_e(x),
\end{aligned}
\]
where we have defined
\[
\Delta_e(x) = \eta_n(x+e)-\eta_n(x-e).
\]
Note that \(p_{\beta}(x)\) is also multilinear, so that its derivative is also given by a finite-difference.\\
Since the set of vertices (i.e., graphs on \(m_n\) vertices) is finite with cardinality at most \(2^{\binom{m_n}{2}}\) and there are \(\binom{m_n}{2}\) possible edges, a union bound shows that there exists a function \(\delta(n)\to 0\) such that
\[
\mathbb{P}\Biggl(\max_{x\in\{0,1\}^{\binom{m_n}{2}}}\; \max_{e}\; |\Delta_e(x)| > \frac{\epsilon}{m_n}\Biggr)
\le 2^{\binom{m_n}{2}}\binom{m_n}{2}\,\delta(n).
\]
By the multilinearity of the extension, the same bound applies when taking the supremum over all \(x\in[0,1]^{\binom{m_n}{2}}\). Hence,
\[
\mathbb{P}\!\Biggl(\sup_{x\in[0,1]^{\binom{m_n}{2}}}\; \sup_{1\le i<j\le m_n}\; \Bigl|D_e(x)-\frac{\partial}{\partial x_{ij}}\log p_{\beta}(x)| > \frac{\epsilon}{m_n}\Biggr)
\le 2^{\binom{m_n}{2}}\binom{m_n}{2}\,\delta(n).
\]
Under the assumption \(\binom{m_n}{2} = o(\log n)\), the right-hand side converges to zero as \(n\to\infty\). This completes the proof.
\end{proof}

\begin{lemma}[Uniform Convergence of the Discrete Second Derivative]{\label{weak_law_2der}}
Let \(\mathcal{G}_{n,1}, \dots, \mathcal{G}_{n,n}\) be independent random graphs sampled from the \(\mathbb{P}_{\beta}\) model with \(m_n\to\infty\) in such a way that
\[
\binom{m_n}{2} = o(\log n).
\]
For each graph \(G\) on \(m_n\) vertices, define
\[
N(G)=\sum_{i=1}^{n} \mathds{1}\bigl(\mathcal{G}_{n,i}=G\bigr)
\]
to be the number of occurrences of \(G\) in the sample, and extend the function
\[
\log\Bigl(\frac{N(G)}{n}\Bigr)
\]
to \([0,1]^{\binom{m_n}{2}}\) by multilinear interpolation, denoting the extension by \(\log\Bigl(\frac{N(x)}{n}\Bigr)\).

For any two distinct edges \(e\) and \(e'\) (viewed as coordinates), define the discrete second derivative by
\begin{align*}
\frac{\partial^2}{\partial x_{e}\partial x_{e'}} \log\Bigl(\frac{N(x)}{n}\Bigr) :=& \log\Bigl(\frac{N(x+e+e')}{n}\Bigr) - \log\Bigl(\frac{N(x+e-e')}{n}\Bigr) - \log\Bigl(\frac{N(x-e+e')}{n}\Bigr)\\
&+ \log\Bigl(\frac{N(x-e-e')}{n}\Bigr),
\end{align*}
where for a given \(x\in [0,1]^{\binom{m_n}{2}}\), the notation \(x\pm e \pm e'\) indicates that the coordinates corresponding to edges \(e\) and \(e'\) are set to the indicated values (either 0 or 1), while the remaining coordinates remain as in \(x\).

Then, for every \(\epsilon>0\) we have
\[
\mathbb{P}\!\Biggl(\sup_{x\in [0,1]^{\binom{m_n}{2}}}\; \sup_{\substack{e,e'\\e\neq e'}} \left|\frac{\partial^2}{\partial x_{e}\partial x_{e'}} \log\Bigl(\frac{N(x)}{n}\Bigr) - \frac{\partial^2}{\partial x_{e}\partial x_{e'}} \log p_{\beta}(x)\right | > \frac{\epsilon}{m_n}\Biggr)
\rightarrow 0
\]

\end{lemma}

\begin{proof}
For vertices of the hypercube, that is, for \(x\in\{0,1\}^{\binom{m_n}{2}}\) corresponding to graphs, the uniform weak law (see previous lemmas) shows that
\[
\frac{N(G)}{n} = p_{\beta}(G)\Bigl(1+\delta_n(G)\Bigr),
\]
with error terms \(m_n\delta_n(G)\to 0\) uniformly in probability over all graphs \(G\). Consequently, taking logarithms, we have
\[
\log\frac{N(G)}{n} = \log p_{\beta}(G) + \eta_n(G),
\]
where \(\eta_n(G)=\log\bigl(1+\delta_n(G)\bigr)\) satisfies \(m_n\eta_n(G)\to 0\) uniformly in probability.

Therefore, for \(x\in\{0,1\}^{\binom{m_n}{2}}\) we have
\begin{align*}
\frac{\partial^2}{\partial x_e\partial x_{e'}} \log\Bigl(\frac{N(x)}{n}\Bigr)
= &\log\left(\frac{p_{\beta}(x + e + e')p_{\beta}(x - e - e')}{p_{\beta}(x + e - e')p_{\beta}(x - e + e')}\right) + \eta_n(x+e+e') - \eta_n(x+e-e')\\
&- \eta_n(x-e+e') + \eta_n(x-e-e').
\end{align*}
By the uniform convergence of \(m_n\eta_n(\cdot)\) to 0 in probability, for every \(\epsilon>0\) there exists \(\delta(n)\to 0\) such that
\begin{align*}
&\mathbb{P}\Biggl(\max_{x\in \{0,1\}^{\binom{m_n}{2}}}\, \max_{\substack{e,e'\\e\neq e'}} \left|\eta_n(x+e+e') - \eta_n(x+e-e') - \eta_n(x-e+e') + \eta_n(x-e-e')\right| > \frac{\epsilon}{m_n}\Biggr)\\
\le &2^{\binom{m_n}{2}}\,\binom{m_n}{2}^2\, \delta(n).
\end{align*}

Since the set of vertices \(\{0,1\}^{\binom{m_n}{2}}\) has cardinality at most \(2^{\binom{m_n}{2}}\) and there are at most \(\binom{m_n}{2}^2\) pairs \((e,e')\), a union bound over these finite sets yields the stated probability bound.

Finally, by the multilinear extension the same bound applies when taking the supremum over all \(x\in [0,1]^{\binom{m_n}{2}}\). Under the assumption \(\binom{m_n}{2}=o(\log n)\), the factor \(2^{\binom{m_n}{2}}\binom{m_n}{2}^2\) grows slower than any positive power of \(n\), ensuring that the right-hand side tends to zero as \(n\to\infty\).

This completes the proof.
\end{proof}
 
\begin{lemma}[Controlling the smoothness term]\label{smoothness}
The smoothness term $\mathcal{S}$ in formula \eqref{eq:upper_bound_final}, after normalization by $n_v$, is $O(m^{-1/2})$.
\end{lemma}

\begin{proof}

We will invoke Theorem 1.6 from Chatterjee \& Dembo \cite{chatterjee2016nonlinear} (henceforth NLLD). Let $f: [0,1]^{n_v} \to \mathbb{R}$ be a function that is twice continuously differentiable in $(0,1)^{n_v}$, such that $f$ and its first and second order derivatives extend continuously to the boundary. Define the following quantities:
\begin{itemize}
    \item $a := \|f\|_\infty$, the supremum norm of the function.
    \item $b_e := \|\partial f / \partial x_e\|_\infty$ for any edge $e$.
    \item $c_{ee'} := \|\partial^2 f / (\partial x_e \partial x_{e'})\|_\infty$ for any edges $e$ and $e'$.
    \item $I(x) := \sum_{e}^{} (x_e \log x_e + (1-x_e)\log(1-x_e))$ is the entropy function for $x \in [0,1]^{{m \choose 2}}$.
\end{itemize}
Theorem 1.6 of NLLD provides an upper bound for the free energy $F = \log \sum_{x \in \{0,1\}^{n_v}} e^{f(x)}$:
\begin{equation}
F \le \sup_{x \in [0,1]^{n_v}} (f(x) - I(x)) + \text{complexity term} + \text{smoothness term}
\end{equation}
The smoothness term, which we denote $S_{\text{NLLD}}$, is given by:
\begin{align*}
    S_{\text{NLLD}} = & \quad 4 \left( \sum_{e}^{} (ac_{ee} + b_e^2) + \frac{1}{4} \sum_{e,e'}^{} (ac_{ee'}^2 + b_e b_{e'} c_{ee'} + 4b_e c_{ee}) \right)^{1/2} \tag{$S_1$}\\
    & + \frac{1}{4} \left( \sum_{e}^{} b_e^2 \right)^{1/2} \left( \sum_{e}^{} c_{ee}^2 \right)^{1/2} \tag{$S_2$} \\
    & + 3 \sum_{e}^{} c_{ee} + \log 2 \tag{$S_3$}
\end{align*}
Our term $\mathcal{S}$ in formula \eqref{eq:upper_bound_final} corresponds to this $S_{\text{NLLD}}$.

We apply the theorem to the following function main text:
\begin{equation}
f(x) = -\frac{\lambda}{m^{v(H)-2}}H(x) + \log\left(\frac{N(x)}{n}\right)
\end{equation}
Here, $x = (x_e)_{e \in E_m}$ is the vector of edge indicators for a graph on $m$ vertices. The number of variables is $n_v = \Theta(m^2)$. The derivative bounds are:
\begin{itemize}
    \item $a = \|f\|_\infty = O(m^2 B)$, where $B = 1 + C_H|\lambda| + \sum_{k=1}^K |\beta_k| C_k$ is an $O(1)$ constant.
    \item $b_e = \|\partial f / \partial x_e\|_\infty = O(B)$. 
    \item $c_{ee'} = \|\partial^2 f / (\partial x_e \partial x_{e'})\|_\infty = O(B/m)$. 
\end{itemize}
The crucial property that generalizes is the \textbf{sparsity of the second derivative interactions}. The term $c_{ee'}$ is \(O(B/m)\) if edges $e$ and $e'$ share a vertex, which happens in at most \(O(m^3)\) ways, and it is \(O(B/m^2)\) if they do not share a vertex, which happens in at most \(O(m^4)\) ways. This sparsity is the key to the final scaling.

\subsubsection*{Calculation of the Smoothness Term Components}
We now calculate each component of $S_{\text{NLLD}}$ using these scalings. The indices $e,e'$ now run over the $n_v = O(m^2)$ edges.

\textbf{Component $S_1$}: This is the dominant term. The sum under the square root has several parts:
\begin{itemize}
    \item $\sum_{e}^{} (ac_{ee} + b_e^2)$: Each term is $O(m^2 B \cdot B/m) + O(B^2) = O(mB^2)$. The sum over $n_v$ terms is $O(n_v \cdot mB^2) = O(m^3 B^2)$.
    \item $\sum_{e,e'}^{} ac_{ee'}^2$: The vertex sharing pairs \((e,e')\) contribute $O(m^3)$ terms. Each term is $O(m^2 B \cdot (B/m)^2) = O(B^3)$. The remaining pairs \((e,e')\) contribute $O(m^4)$ terms. Each term is $O(m^2 B \cdot (B/m^2)^2) = O(B^3/m^2)$. The sum is $O(m^3 \cdot B^3) + O(m^4 \cdot B^3/m^2) = O(m^3 B^3)$.
    \item $\sum_{e,e'}^{} b_e b_{e'} c_{ee'}$: The vertex sharing pairs contribute $O(m^3)$ terms. Each term is $O(B \cdot B \cdot B/m) = O(B^3/m)$. The remaining pairs contribute $O(m^4)$ terms. Each term is $O(B \cdot B \cdot B/m^2) = O(B^3/m^2)$. The sum is $O(m^3 \cdot B^3/m) + O(m^4 \cdot B^3/m^2) = O(m^2 B^3)$.
    \item $\sum_{e,e'}^{} 4b_e c_{ee} = n_v \sum_{e}^{} 4b_e c_{ee}$: The inner sum is $O(n_v \cdot B \cdot B/m) = O(m^2 \cdot B^2/m) = O(m B^2)$. The total is $O(m^2 \cdot mB^2) = O(m^3 B^2)$.
\end{itemize}
The sum under the square root is $O(m^3B^2 + m^3B^3 + m^2B^3) = O(m^3B^3)$ assuming $B \ge 1$. Thus, the entire $S_1$ component is $O((m^3 B^3)^{1/2}) = O(m^{3/2} B^{3/2})$.

\textbf{Component $S_2$}: $\frac{1}{4} \left( \sum_{e}^{} b_e^2 \right)^{1/2} \left( \sum_{e}^{} c_{ee}^2 \right)^{1/2}$
\begin{itemize}
    \item The first square root is $(O(n_v B^2))^{1/2} = (O(m^2 B^2))^{1/2} = O(mB)$.
    \item The second square root is $(O(n_v (B/m)^2))^{1/2} = (O(m^2 \cdot B^2/m^2))^{1/2} = O(B)$.
\end{itemize}
Thus, component $S_2$ is $O(mB \cdot B) = O(m B^2)$.

\textbf{Component $S_3$}: $3 \sum_{e}^{} c_{ee} + \log 2$
The sum has $n_v$ terms, each of size $O(B/m)$. The sum is $O(n_v \cdot B/m) = O(m^2 \cdot B/m) = O(mB)$.

\subsubsection*{Total Smoothness Term and Normalization}
Combining the components gives the total smoothness term:
$$ S_{\text{NLLD}} = O(m^{3/2} B^{3/2}) + O(m B^2) + O(mB) $$
For large $m$ and $B \ge 1$, the dominant term is $S_1 = O(m^{3/2} B^{3/2})$. The main text normalizes the log-partition function by $m^2$, which corresponds to normalizing error terms by $m^2$. The relevant normalization factor for the smoothness term from NLLD is the number of variables, $n_v = \binom{m}{2} = O(m^2)$.
$$ \frac{S_{\text{NLLD}}}{n_v} = \frac{O(m^{3/2} B^{3/2})}{O(m^2)} = O(m^{3/2 - 2} B^{3/2}) = O(m^{-1/2} B^{3/2}) $$
Since $B$ is an $O(1)$ constant with respect to $m$, we arrive at the final scaling:
$$ \frac{S_{\text{NLLD}}}{n_v} = O(m^{-1/2}) $$

\end{proof}

\begin{proof}[Proof of Lemma \ref{lambda_cont}]
Let $h_0>0$ be a constant and define $\mu(\lambda)=h_0 e^{\lambda}$. Let
$Z_{\mu(\lambda)} \sim \mathrm{Poisson}(\mu(\lambda))$. Define the deterministic functions
\[
f(\lambda):=\mathbb{E}\Big[(Z_{\mu(\lambda)}-h_0)^2e^{-\lambda(Z_{\mu(\lambda)}-h_0)}\Big],
\qquad
g(\lambda):=\mathbf{Var}\Big((Z_{\mu(\lambda)}-h_0)e^{-\lambda(Z_{\mu(\lambda)}-h_0)}\Big).
\] As a first step we will prove that the functions $f(\lambda)$ and $g(\lambda)$ are continuous.\\

\textbf{Continuity of $f(\lambda)$}\\

We may rewrite
\[
f(\lambda)
= e^{\lambda h_0}\,\mathbb{E}\Big[(Z-h_0)^2 e^{-\lambda Z}\Big],
\qquad Z\sim \mathrm{Poisson}(\mu(\lambda)).
\]
Using the Poisson moment generating function we want to compute
\[
\mathbb{E}[(Z-h_0)^2 e^{-\lambda Z}]
= \mathbb{E}[Z^2 e^{-\lambda Z}]
-2h_0 \mathbb{E}[Z e^{-\lambda Z}]
+ h_0^2 \mathbb{E}[e^{-\lambda Z}].
\]
For a Poisson$(\mu)$ random variable we have
\[
\mathbb{E}[e^{-\lambda Z}]
= \exp\!\big(\mu(e^{-\lambda}-1)\big).
\]
Differentiating with respect to $\lambda$ yields
\[
\mathbb{E}[Z e^{-\lambda Z}]
= \mu e^{-\lambda}
\exp\!\big(\mu(e^{-\lambda}-1)\big),
\]
and a second differentiation gives
\[
\mathbb{E}[Z^2 e^{-\lambda Z}]
= \big(\mu e^{-\lambda}+\mu^2 e^{-2\lambda}\big)
\exp\!\big(\mu(e^{-\lambda}-1)\big).
\]
Substituting these expressions (with $\mu=\mu(\lambda)$) into the
previous expansion, we obtain

\[\mathbb{E}[(Z-h_0)^2 e^{-\lambda Z}]
= \exp\!\big(\mu(\lambda)(e^{-\lambda}-1)\big)
\Big[
\mu(\lambda)e^{-\lambda}
+ \mu(\lambda)^2 e^{-2\lambda}
-2h_0 \mu(\lambda)e^{-\lambda}
+ h_0^2
\Big].\]
Consequently,
\[
f(\lambda)
= e^{\lambda h_0}
\exp\!\big(\mu(\lambda)(e^{-\lambda}-1)\big)
\Big[
\mu(\lambda)e^{-\lambda}
+ \mu(\lambda)^2 e^{-2\lambda}
-2h_0 \mu(\lambda)e^{-\lambda}
+ h_0^2
\Big].
\]

Finally, since the map $\lambda\mapsto \mu(\lambda)=h_0 e^{\lambda}$ is continuous, it is easy to observe that $f(\lambda)$ is continuous in $\lambda$.\\ 

\textbf{Continuity of $g(\lambda)$}\\

Defining $Y_\lambda := (Z_{\mu(\lambda)}-h_0)\,e^{-\lambda(Z_{\mu(\lambda)}-h_0)}$, we will show that the map $
\lambda \mapsto \mathrm{Var}(Y_\lambda)
$
is continuous. It suffices to prove continuity of $\lambda\mapsto \mathbb{E}[Y_\lambda]$
and $\lambda\mapsto \mathbb{E}[Y_\lambda^2]$. We have
\[
\mathbb{E}[Y_\lambda]
=\mathbb{E}\Big[(Z-h_0)e^{-\lambda(Z-h_0)}\Big]
= e^{\lambda h_0}\,\mathbb{E}\Big[(Z-h_0)e^{-\lambda Z}\Big],
\qquad Z\sim \mathrm{Poisson}(\mu(\lambda)).
\]
Expanding gives
\[
\mathbb{E}\Big[(Z-h_0)e^{-\lambda Z}\Big]
=\mathbb{E}[Z e^{-\lambda Z}] - h_0\,\mathbb{E}[e^{-\lambda Z}].
\]
Using
\[
\mathbb{E}[e^{-\lambda Z}]=\exp\!\big(\mu(e^{-\lambda}-1)\big),
\qquad
\mathbb{E}[Z e^{-\lambda Z}]
=\mu e^{-\lambda}\exp\!\big(\mu(e^{-\lambda}-1)\big),
\]
one obtains
\[
\mathbb{E}[Y_\lambda]
= e^{\lambda h_0}\exp\!\big(\mu(\lambda)(e^{-\lambda}-1)\big)\Big(\mu(\lambda)e^{-\lambda}-h_0\Big).
\]
It is immediate that
$\lambda\mapsto \mathbb{E}[Y_\lambda]$ is continuous. Next note that
\[
Y_\lambda^2 = (Z-h_0)^2 e^{-2\lambda(Z-h_0)}
= e^{2\lambda h_0}\,(Z-h_0)^2 e^{-2\lambda Z},
\]
so that
\[
\mathbb{E}[Y_\lambda^2]
= 
e^{2\lambda h_0}\mathbb{E}[(Z-h_0)^2 e^{-2\lambda Z}]
=e^{2\lambda h_0}(\mathbb{E}[Z^2 e^{-2\lambda Z}]
-2h_0\,\mathbb{E}[Z e^{-2\lambda Z}]
+h_0^2\,\mathbb{E}[e^{-2\lambda Z}]).
\]
Using the Poisson moment generating function with $t=-2\lambda$ gives
\[
\mathbb{E}[e^{-2\lambda Z}]
=\exp\!\big(\mu(e^{-2\lambda}-1)\big),
\quad
\mathbb{E}[Z e^{-2\lambda Z}]
=\mu e^{-2\lambda}\exp\!\big(\mu(e^{-2\lambda}-1)\big),
\]
and
\[
\mathbb{E}[Z^2 e^{-2\lambda Z}]
=\big(\mu e^{-2\lambda}+\mu^2 e^{-4\lambda}\big)\exp\!\big(\mu(e^{-2\lambda}-1)\big).
\]
Therefore, with $\mu=\mu(\lambda)$,
\begin{align*}
\mathbb{E}[Y_\lambda^2]
&= e^{2\lambda h_0}\exp\!\big(\mu(\lambda)(e^{-2\lambda}-1)\big)
\Big[
\mu(\lambda)e^{-2\lambda}
+\mu(\lambda)^2 e^{-4\lambda}
-2h_0\mu(\lambda)e^{-2\lambda}
+h_0^2
\Big].
\end{align*}
It follows that $\lambda\mapsto \mathbb{E}[Y_\lambda^2]$ is continuous.
Since $\lambda\mapsto \mathbb{E}[Y_\lambda]$ and $\lambda\mapsto \mathbb{E}[Y_\lambda^2]$
are continuous
\[
\lambda\mapsto \mathrm{Var}(Y_\lambda)
=\mathbb{E}[Y_\lambda^2]-\big(\mathbb{E}[Y_\lambda]\big)^2
\]
is continuous.\\

\textbf{Convergence of the estimators of $f(\lambda)$ and $g(\lambda)$}\\

Note that conditioning on $\hat\lambda$, the quantity $\hat\mu=\mu(\hat\lambda)$ is fixed, and
$Z_{\hat\mu}\mid \hat\lambda \sim \mathrm{Poisson}(\hat\mu)$ by definition. Therefore,
\[
\mathbb{E}\!\left[(Z_{\hat{\mu}}-h_0)^2e^{-\hat{\lambda}(Z_{\hat{\mu}}-h_0)}\mid \hat{\lambda}\right]
= f(\hat\lambda),
\]
and similarly
\[
\mathbf{Var}\!\left[(Z_{\hat{\mu}}-h_0)e^{-\hat{\lambda}(Z_{\hat{\mu}}-h_0)}\mid \hat{\lambda}\right]
= g(\hat\lambda).
\]
Since $\hat\lambda\xrightarrow{P}\lambda^{\circ}$ and $f$ is continuous, the continuous mapping theorem yields
\[
f(\hat\lambda)\xrightarrow{P} f(\lambda^{\circ})
=\mathbb{E}\Big[(Z_{\mu}-h_0)^2e^{-\lambda^{\circ}(Z_{\mu}-h_0)}\Big],
\]
where $\mu=\mu(\lambda^{\circ})$. Likewise, continuity of $g$ implies
\[
g(\hat\lambda)\xrightarrow{P} g(\lambda^{\circ})
=\mathbf{Var}\Big((Z_{\mu}-h_0)e^{-\lambda^{\circ}(Z_{\mu}-h_0)}\Big).
\]

\end{proof}

\begin{proof}[Proof of Lemma \ref{two_sample_asymptotics_lemma}]
Let $\Delta_n := (\hat{\theta}_1 - \hat{\theta}_2) - (\theta_1 - \theta_2)$.
Define
\[
U_n := a_{n_1}(\hat{\theta}_1 - \theta_1),
\qquad 
V_n := b_{n_2}(\hat{\theta}_2 - \theta_2).
\]
By assumption, $U_n \xrightarrow{d} N(0, \sigma_1^2)$ and $V_n \xrightarrow{d} N(0, \sigma_2^2)$. 
Since the samples are independent, $U_n$ and $V_n$ are independent for each $n$. Consider the linear combination
\[
a_{n_1}\Delta_n = U_n - r_n V_n,
\qquad \text{where } r_n = \frac{a_{n_1}}{b_{n_2}}.
\]
Let $\varphi_{U_n}(t) = \mathbb{E}[e^{itU_n}]$ and $\varphi_{V_n}(t) = \mathbb{E}[e^{itV_n}]$ denote the characteristic functions. Independence gives
\[
\varphi_{a_{n_1}\Delta_n}(t)
= \mathbb{E}[e^{it(U_n - r_n V_n)}]
= \varphi_{U_n}(t)\, \varphi_{V_n}(-r_n t).
\]
By the marginal central limit theorems and the convergence $r_n \to \rho$,
\[
\varphi_{U_n}(t) \to e^{-\frac{1}{2}\sigma_1^2 t^2},
\qquad 
\varphi_{V_n}(-r_n t) \to e^{-\frac{1}{2}\sigma_2^2 \rho^2 t^2}.
\]
Hence
\[
\varphi_{a_{n_1}\Delta_n}(t)
\to 
\exp\!\left[-\frac{1}{2}(\sigma_1^2 + \rho^2\sigma_2^2)t^2\right],
\]
which is the characteristic function of $N(0,\, \sigma_1^2 + \rho^2\sigma_2^2)$. 
By L\'evy's continuity theorem,
\[
a_{n_1}\Delta_n \xrightarrow{d} N\!\left(0,\, \sigma_1^2 + \rho^2\sigma_2^2\right).
\]
Equivalently, letting
\[
V_n^\ast := \frac{\sigma_1^2}{a_{n_1}^2} + \frac{\sigma_2^2}{b_{n_2}^2},
\]
we have
\[
\frac{\Delta_n}{\sqrt{V_n^\ast}} \xrightarrow{d} N(0,1).
\]


\end{proof}

\begin{proof}[Proof of Lemma \ref{density}]
For a finite partition \(P\) of \([0,1]\), let \(\mathcal{A} = \sigma (P\times P)\) be the sigma algebra generated by \(P\times P\). Now given an \(h\in W\), consider \(h_{\mathcal{A}} = \mathbb{E}[h\mid\mathcal{A}]\), where the randomness is wrt the Lebesgue measure on \([0,1]^2\). Since $I$ is convex, Jensen's inequality yields
$\E\big[I(h)\mid\mathcal{A}\big]\ge I\big(\E[h\mid\mathcal{A}]\big)=I(h_{\mathcal{A}})$.
Integrating both sides gives $I(h)\ge I(h_{\mathcal{A}})$.
 For any finite simple graph \(H\), the map $h\mapsto t(H,h)$ is continuous wrt the cut metric.
By the Frieze--Kannan weak regularity lemma, we can find a sequence of finite partitions \(P_n\) such that for the corresponding sigma algebras \(\mathcal{A}_n\) have the property that \(h_{\mathcal{A}_n}\rightarrow h\) wrt the cut metric.
Putting these together, we get
\begin{align*} 
F(h_{\mathcal{A}_n}) &= \sum_{k=1}^K \beta_k\,t(T_k,h_{\mathcal{A}_n})\;-\;\frac{1}{2}I(h_{\mathcal{A}_n})\\
& \geq \sum_{k=1}^K \beta_k\,t(T_k,h_{\mathcal{A}_n})\;-\;\frac{1}{2}I(h)
\end{align*}
Hence, \(\limsup F(h_{\mathcal{A}_n}) \geq F(h)\), so that \(F(h)\) is upper bounded by the supremum over step-functions.
\end{proof}

\section{Existence and Uniqueness of roots \label{sec:root_unicity}}

\subsection{Conditions for existence and uniqueness of roots}

We begin by stating the two lemmas required to establish the existence and uniqueness of the root for the moment-generating function equations arising in the maximum entropy framework.

\begin{lemma}\label{uniqueness_1}
Let \(X\) be a real valued random variable such that \(\mathbb{P}(X\neq0)>0\) , and let its Laplace transform exist in an open interval \(I\). Define the function \(f(\lambda) = \mathbb{E}[Xe^{-\lambda X}]\) on \(I\). Then, \(\exists\) at most one \(\lambda^{*}\in I\) such that \(f(\lambda^{*})=0\). Furthermore, if such a \(\lambda^{*}\) exists, it is the unique minimizer of the convex function \(M(\lambda) = \mathbb{E}[e^{-\lambda X}]\) on \(I\).
\end{lemma}

\begin{proof}
    Under the assumption that the Laplace transform exists on \(I\), we can differentiate under the expectation:
    \[
    f'(\lambda) = \frac{d}{d\lambda}\mathbb{E}[Xe^{-\lambda X}] = \mathbb{E}[-X^2e^{-\lambda X}]
    \]
    Since \(\mathbb{P}(X\neq 0)>0\), the term \(X^2e^{-\lambda X}\) is strictly positive with non-zero probability. Therefore, \(\mathbb{E}[X^2e^{-\lambda X}]>0\), which implies \(f'(\lambda)<0\) for all \(\lambda \in I\). Hence \(f(\lambda)\) has at most one root. Since \(M'(\lambda) = - f(\lambda)\), this means that if such a \(\lambda^{*}\) exists, it must be the unique minimizer of the convex function \(M(\lambda)\).
\end{proof}

\begin{lemma}\label{existence_1}
    Let \(X\) be a discrete random variable with finite support \(\mathcal{X} = \{x_1,\ldots , x_k\}\). Define \(f(\lambda) = \mathbb{E}[Xe^{-\lambda X}]\). A solution \(\lambda^{*}\in\mathbb{R}\) to \(f(\lambda^{*}) = 0\) exists if and only if:
    \begin{itemize}
        \item \(\min(\mathcal{X}) < 0 < \max(\mathcal{X})\), OR
        \item \(\mathbb{P}(X=0) = 1\) (i.e., \(X\) is the zero constant).
    \end{itemize}
\end{lemma}

\begin{proof}
    We first prove \textbf{sufficiency}.
    
    \textit{Case 1: \(\mathbb{P}(X=0)=1\).} If \(X \equiv 0\), then \(f(\lambda) = \mathbb{E}[0 \cdot e^{-\lambda \cdot 0}] = 0\) for all \(\lambda \in \mathbb{R}\). Thus, any \(\lambda^* \in \mathbb{R}\) is a root.

    \textit{Case 2: \(\min(\mathcal{X}) < 0 < \max(\mathcal{X})\).} 
    Let \(x_{\min} = \min(\mathcal{X})\) and \(x_{\max} = \max(\mathcal{X})\). By assumption, \(x_{\min} < 0\) and \(x_{\max} > 0\). 
    The function \(f(\lambda) = \sum_{x \in \mathcal{X}} \mathbb{P}(X=x) x e^{-\lambda x}\) is a finite sum of continuous functions, hence \(f\) is continuous on \(\mathbb{R}\).
    
    As \(\lambda \to \infty\), we factor out the term with the smallest exponent:
    \[ f(\lambda) = e^{-\lambda x_{\min}} \left[ \mathbb{P}(X=x_{\min})x_{\min} + \sum_{x > x_{\min}} \mathbb{P}(X=x)x e^{-\lambda(x - x_{\min})} \right] \]
    Since \(x - x_{\min} > 0\) for all terms in the summation, \(e^{-\lambda(x - x_{\min})} \to 0\) as \(\lambda \to \infty\). Because \(x_{\min} < 0\) and \(\mathbb{P}(X=x_{\min}) > 0\), the leading term inside the bracket is strictly negative. Thus, \(\lim_{\lambda \to \infty} f(\lambda) = -\infty\).

    As \(\lambda \to -\infty\), we factor out the term with the largest exponent:
    \[ f(\lambda) = e^{-\lambda x_{\max}} \left[ \mathbb{P}(X=x_{\max})x_{\max} + \sum_{x < x_{\max}} \mathbb{P}(X=x)x e^{-\lambda(x - x_{\max})} \right] \]
    Since \(x - x_{\max} < 0\) and \(\lambda \to -\infty\), the exponents \(-\lambda(x - x_{\max})\) tend to \(-\infty\), causing the sum to vanish. Because \(x_{\max} > 0\) and \(\mathbb{P}(X=x_{\max}) > 0\), the leading term is strictly positive. Thus, \(\lim_{\lambda \to -\infty} f(\lambda) = \infty\).
    
    By the Intermediate Value Theorem, since \(f\) is continuous and takes both positive and negative values on \(\mathbb{R}\), there exists at least one \(\lambda^* \in \mathbb{R}\) such that \(f(\lambda^*) = 0\).

    We now prove \textbf{necessity}.
    Suppose \(X\) is not almost surely zero. If \(\min(\mathcal{X}) \geq 0\), then every \(x \in \mathcal{X}\) is non-negative and at least one \(x > 0\). Then \(x e^{-\lambda x} \geq 0\) for all \(x, \lambda\), and \(\mathbb{E}[X e^{-\lambda X}] > 0\), so no root exists. Similarly, if \(\max(\mathcal{X}) \leq 0\), then \(x e^{-\lambda x} \leq 0\) for all \(x, \lambda\) with at least one \(x < 0\). Thus \(\mathbb{E}[X e^{-\lambda X}] < 0\), and no root exists. Therefore, the support must strictly straddle zero for a root to exist for non-constant \(X\).
\end{proof}

\subsection{Existence and uniqueness of roots in network models}

Let $\mathcal{G}_{n,1}, \dots, \mathcal{G}_{n,n}$ be $n$ independent and identically distributed realizations of random graphs on $m_n$ vertices. Let $h_{n,i} = \mathcal{T}(H, \mathcal{G}_{n,i}) - h_0$ denote observed statistic, where $h_0$ is chosen such that $\mathbb{E}[h_{n,i}] = 0$ under the null hypothesis.

We define an auxiliary discrete random variable $Y_n$ representing the empirical distribution of the subgraph counts. $Y_n$ takes values in the multiset $\mathcal{H}_n = \{h_{n,1}, \dots, h_{n,n}\}$ with uniform probability:
\[
\mathbb{P}(Y_n = h_{n,i}) = \frac{1}{n}, \quad \text{for } i=1,\dots,n.
\]
The estimating equation for the Lagrange multiplier $\hat{\lambda}_n$ takes different forms depending on the regime's normalization, but all can be cast in the form of an expectation over $Y_n$.

\subsubsection{Verification of Conditions}

To invoke Lemma \ref{uniqueness_1} and Lemma \ref{existence_1}, we must verify that the support of $Y_n$ strictly straddles zero with high probability as $n \to \infty$. That is, we define the event:
\[
E_n := \left\{ \min_{1 \le i \le n} h_{n,i} < 0 < \max_{1 \le i \le n} h_{n,i} \right\}.
\]
We justify that $\lim_{n \to \infty} \mathbb{P}(E_n) = 1$ under any hypothesis where $h_0$ is in the interior of the support of $H(G)$.

\subsubsection*{Regime 1: Fixed Number of Vertices ($m$ fixed)}
In this regime, the sample space of graphs is independent of $n$. The probability mass on either side of zero is a strictly positive constant:
\[
\pi_+ := \mathbb{P}(h_{n,1} > 0) > 0, \quad \pi_- := \mathbb{P}(h_{n,1} < 0) > 0.
\]
The probability that the sample fails to straddle zero is bounded by $(1-\pi_-)^n + (1-\pi_+)^n$, which vanishes as $n \to \infty$. Thus, $\mathbb{P}(E_n) \to 1$.

\subsubsection*{Regime 2: Sparse Regime ($m_n \to \infty$)}
Here, the uncentered subgraph count $H(\mathcal{G}_{n,i})$ converges in distribution to a Poisson random variable $Z$ with mean $c^{v(H)}/\lvert Aut(H)\rvert $, and $h_0 = c_0^{v(H)}/\lvert Aut(H)\rvert$. The variable $h_{n,i}$ behaves asymptotically as $Z - h_0$. Since a Poisson distribution with parameter $c^{v(H)}/\lvert Aut(H)\rvert $ is non-degenerate and supported on integers, it assigns positive probability to values strictly less than $h_0$ and strictly greater than $h_0$.
\[
\lim_{n \to \infty} \mathbb{P}(h_{n,1} > 0) = \mathbb{P}(Z > h_0) > 0, \quad \lim_{n \to \infty} \mathbb{P}(h_{n,1} < 0) = \mathbb{P}(Z < h_0) > 0.
\]
Following the same logic as the fixed case, $\mathbb{P}(E_n) \to 1$.

\subsubsection*{Regime 3: Dense Regime ($m_n \to \infty$)}
In the dense regime (e.g., Erd{\"o}s-R{\'e}nyi with constant $p$) we restrict ourselves to the subcritical case of the ferromagnetic ERGM. The Central Limit Theorem for subgraph counts states that the standardized variable converges to a Normal distribution:
\[
\frac{H(\mathcal{G}_{n,i})- \mathbb{E}[H(\mathcal{G}_{n,i})]}{\sqrt{\Var(H(\mathcal{G}_{n,i}))}} \xrightarrow{d} \mathcal{N}(0, 1).
\]
The limiting Normal distribution has support on the entire real line. Hence, the limiting probabilities of $h_{n,i} = H(\mathcal{G}_{n,i})-h_0$ being positive or negative are strictly positive:
\[
\lim_{n \to \infty} \mathbb{P}(h_{n,1} > 0) > 0, \quad \lim_{n \to \infty} \mathbb{P}(h_{n,1} < 0) > 0.
\]
Therefore, the probability that the maximum is positive and the minimum is negative approaches 1 as $n \to \infty$.

\subsubsection{Proof of Lemma \ref{exis_uniq}}

\begin{proof}

We now apply the lemmas to prove existence and uniqueness.\\

\begin{enumerate}
\item \textbf{Estimating Equations:}\\

\begin{itemize}
    \item \textbf{Fixed/Sparse Regime:} The equation is $\sum_{i=1}^n h_{n,i} e^{-\hat{\lambda}_n h_{n,i}} = 0$. This is equivalent to $\mathbb{E}[Y_n e^{-\hat{\lambda}_n Y_n}] = 0$.
    \item \textbf{Dense Regime:} The estimator $\hat{\lambda}_n$ is the critical point of $\frac{1}{n} \sum_{i=1}^n \exp\left(-\frac{\lambda}{m_n^{v(H)-2}} h_{n,i}\right)$. Differentiating with respect to $\lambda$, the defining equation is:
    \[
    \sum_{i=1}^n h_{n,i} e^{-\frac{\hat{\lambda}_n}{m_n^{v(H)-2}} h_{n,i}} = 0.
    \]
    By defining a scaled parameter $\tilde{\lambda} = \frac{\hat{\lambda}_n}{m_n^{v(H)-2}}$, this reduces to the form $\sum h_{n,i} e^{-\tilde{\lambda} h_{n,i}} = 0$, which is equivalent to $\mathbb{E}[Y_n e^{-\tilde{\lambda} Y_n}] = 0$.\\
\end{itemize}
\item \textbf{Existence of root:}\\

Assume the event $E_n$ holds. The support of the empirical variable $Y_n$ is $\mathcal{H}_n$. By definition of $E_n$, we have $\min(\mathcal{H}_n) < 0 < \max(\mathcal{H}_n)$. By Lemma \ref{existence_1}, this condition is necessary and sufficient for the existence of a real root to the equation $\mathbb{E}[Y_n e^{-\lambda Y_n}] = 0$ (and consequently for the scaled equation in the dense regime). Since $\lim_{n \to \infty} \mathbb{P}(E_n) = 1$, the root exists with probability approaching 1.

\item \textbf{Uniqueness of root}
Consider the function $f(\lambda) = \sum_{i=1}^n h_{n,i} e^{-\lambda h_{n,i}}$. Differentiating with respect to $\lambda$:
\[
f'(\lambda) = -\sum_{i=1}^n h_{n,i}^2 e^{-\lambda h_{n,i}}.
\]
Conditioned on $E_n$, the values $h_{n,i}$ are not all zero. Therefore, $h_{n,i}^2 > 0$ for at least some indices, implying $f'(\lambda) < 0$ for all $\lambda \in \mathbb{R}$. The function is strictly decreasing. By Lemma \ref{uniqueness_1}, if a root exists, it must be unique.

\end{enumerate}

Combining the asymptotic probability of $E_n$ with the deterministic results of Lemmas \ref{uniqueness_1} and \ref{existence_1} conditioned on $E_n$: with probability approaching 1 as $n \to \infty$, the Lagrange multiplier $\hat{\lambda}_n$ exists and is the unique solution to the estimating equation.
\end{proof}

%

\end{document}